\documentclass[12pt]{article}

\usepackage[utf8]{inputenc}

\usepackage{amsmath}
\usepackage{amssymb}
\usepackage{amsthm}
\usepackage{geometry}
\usepackage{longtable}
\usepackage{multirow}
\usepackage{array}

\usepackage{graphicx}
\usepackage{float}

\usepackage{indentfirst}

\usepackage[english]{babel}
\usepackage[backend=bibtex]{biblatex}
\numberwithin{equation}{section}

\newcommand{\Z}{\mathbb{Z}}

\newcommand{\N}{\mathbb{N}}

\newcommand{\Q}{\mathbb{Q}}
\newcommand{\mc}[1]{\mathcal{#1}}
\newcommand{\C}{\mathbb{C}}
\newcommand{\A}{\mathbb{A}}

\newcommand{\M}{\mathcal{M}}
\newcommand{\diag}{\mathrm{diag}}
\newcommand{\modc}[1]{\;(\mathrm{mod} \;#1)}
\newcommand{\cls}{\mathrm{cls}}
\newcommand{\spn}{\mathrm{spn}}
\newcommand{\gen}{\mathrm{gen}}
\renewcommand{\vec}[1]{\mathbf{#1}}
\newcommand{\lcm}{\mathrm{lcm}}

\newcommand{\ord}{\mathrm{ord}}

\newcommand{\rank}{\mathrm{rank}}

\newcommand{\idele}{id\`ele}
\newcommand{\mf}{\mathfrak}

\let\latexchi\chi
\makeatletter
\renewcommand\chi{\@ifnextchar_\sub@chi\latexchi}
\newcommand{\sub@chi}[2]{
	\@ifnextchar^{\subsup@chi{#2}}{\latexchi^{}_{#2}}%
}
\newcommand{\subsup@chi}[3]{
	\latexchi_{#1}^{#3}%
}
\makeatother

\newtheorem{lemma}{Lemma}
\numberwithin{lemma}{section}
\newtheorem{theorem}[lemma]{Theorem}
\newtheorem{proposition}[lemma]{Proposition}
\newtheorem{corollary_lem}{Corollary}[lemma]
\newtheorem{corollary_prop}{Corollary}[lemma]
\theoremstyle{definition}
\newtheorem*{definition}{Definition}
\theoremstyle{remark}
\newtheorem*{remark}{Remark}

\makeatletter
\renewcommand*\env@matrix[1][*\c@MaxMatrixCols c]{%
	\hskip -\arraycolsep
	\let\@ifnextchar\new@ifnextchar
	\array{#1}}
\makeatother

\title{Near-miss Identities and Spinor Genus Classification of Ternary Quadratic Forms with Congruence Conditions}

\author{Kush Singhal}

\date{}

\begin{document}
	\maketitle
	\begin{abstract}
		In this paper, near-miss identities for the number of representations of some integral ternary quadratic forms with congruence conditions are found and proven. The genus and spinor genus of the corresponding lattice cosets are then classified. Finally, a complete genus and spinor genus classification for all conductor 2 lattice cosets of 2-adically unimodular lattices is given.
	\end{abstract}
	
	\section{Introduction}
	
	For $\vec a \in \N^3$ and $\vec x \in \Z^3$, we consider the ternary quadratic form \[Q_{\vec a}(\vec x) = a_1x_1^2 + a_2x_2^2 + a_3x_3^2\]
	We may study the number of \emph{representations} $\vec x \in \Z^3$ of $n$ by $Q_{\vec a}$, i.e. solutions to $Q_{\vec a}(\vec x) = n$. The number of such representations is denoted by $r_{\vec a}(n)$. We may also consider representations of $n$ by $Q_{\vec a}$ under certain congruence conditions, i.e. given a modulus $N\in \N$ and a vector $\vec h \in \N^3$, we can consider \[r_{\vec a, \vec h, N}(n) = \#\{\vec x \in \Z^3 : Q_{\vec a}(\vec x) = n, \vec x \equiv \vec h \modc{N}\}.\]
	
	Representations with congruence conditions are useful, for instance, in the study of quadratic polynomials, which are generalizations of quadratic forms with a non-zero linear contribution of $\vec x$. By completing squares, many quadratic polynomials may be transformed into quadratic forms under congruence conditions (see for instance \cite{chan_oh} and \cite{haensch_kane}).
	
	For certain special values of $\vec a$, formulae for $r_{\vec a}(n)$ have been found and proven in \cite{kane_bringmann}. Specifically, these special values of $\vec a$ are those which have class number 1, i.e. where quadratic forms that are $p$-adically equivalent to $Q_{\vec a}$ are integrally equivalent as well. Moreover, formulae for $r_{\vec a, \vec h, N}(n)$ of the form 
	\begin{equation}\label{eqn::coeff_form_clsnbr1}
		r_{\vec a, \vec h, N}(n) = d_{\vec a, \vec h, N}(n) r_{\vec a}(n)
	\end{equation}
	are also given in \cite{kane_bringmann} for certain values of $\vec a, \vec h$, and $N$, where $d_{\vec a, \vec h, N}(n)$ only depends on the value of $n$ modulo $M$ for some fixed $M$ depending on $\vec a$, $\vec h$, and $N$.  
	
	In this paper, we have found and proven a number of similar identities, where instead of holding for all $n\in \N$, these identities hold for almost all $n\in \N$. Explicitly, we have found values of $\vec a, \vec h, N$, and $d_{\vec a, \vec h, N}(n)$ for which an identity of the shape given in equation (\ref{eqn::coeff_form_clsnbr1}) holds for all positive integers $n$ except those lying in a specific square class. This is the content of Theorem \ref{thm::spn_identities_elementary_form}. 
	
	\begin{theorem}\label{thm::spn_identities_elementary_form}
		For all 85 values of $\vec a, \vec h$, and $N$ given in table \ref{table::list_spn1_identities} in appendix \ref{sctn::list_spn1_identities}, the identity  \[r_{\vec a, \vec h, N}(n) = d_{\vec a, \vec h, N}(n) r_{\vec a}(n)\]
		holds for all $n\in \N\backslash \{tk^2 : k\in \N\}$ for some positive integer $t$, where the value of $d_{\vec a, \vec h, N}(n)$ depends only on $n\modc{M}$ for some $M\in \N$. Here, the values of $t$, $M$, and $d_{\vec a, \vec h, N}(n)$ are also given in table \ref{table::list_spn1_identities} (the value of $d_{\vec a, \vec h, N}(n)$ is given under the column heading ``\emph{coeff.}'').
	\end{theorem}
	
	To explain the existence of such identities, we need to consider the lattice coset $L_{\vec a}+ \vec v$ where $\vec v = \frac1N\vec h$ and $L_{\vec a}$ is a lattice on a quadratic space defined by the quadratic form $Q_{\vec a}$ (see Section \ref{sctn::Algebraic_Theory}). Vectors in this lattice coset $L_{\vec a} + \vec v$ correspond to $\vec x\in (\frac1N \Z)^3$ such that $\vec x - \frac1N \vec h \in \Z^3$, which captures the congruence conditions stated above. We also consider the theta series $\Theta_{L'+\vec v'}$ of an arbitrary lattice coset $L'+\vec v'$, which are generating functions of the value $\#\{\vec x\in L'+v': Q_{L'}(\vec x) = n\}$. In the case of $L_{\vec a} + \vec v$ (with $\vec v = \frac1N \vec h$ and $\vec h \in L_{\vec a}$), the corresponding theta series is simply a generating function for $r_{\vec a, \vec h, N}(n)$. Identities of the form of equation (\ref{eqn::coeff_form_clsnbr1}) which hold for all $n\in \N$ (given in \cite{kane_bringmann}) can then be seen as resulting from the generalization of the Siegel-Weil Mass Formula to the case of lattice cosets (proven by Shimura in \cite{shimura}), one version of which may be stated as: 
	\begin{equation}\label{eqn::siegel_weil_mass_formula}
		\mc E = \mc E_{\gen^+(L_{\vec a} + \vec v)} := \frac{1}{\sum_{L'+ \vec v' \in \mc G} |O^+(L'+ \vec v')|^{-1}} \sum_{L'+\vec v' \in \mc G} \frac{\Theta_{L' + \vec v'}}{|O^+(L'+\vec v')|} 
	\end{equation}
	is an Eisenstein series, where $\mc G$ is a set of representatives of the classes in the genus of $L_{\vec a}+\vec v$, and $O^+(L'+\vec v')$ is the set of integral rotations of $L'+\vec v'$ (see Section \ref{sctn::classifications}). If $L_{\vec a} + \vec v$ has class number 1, i.e. if there is only one class in the genus of $L_{\vec a} + \vec v$, then the average given on the right collapses to simply the theta series of $L_{\vec a} + \vec v$. Since Eisenstein series are generally well-known, comparing coefficients would then yield the desired identities involving $r_{\vec a, \vec h, N}(n)$ and $r_{\vec a}(n)$.
	
	A further generalization of equation (\ref{eqn::siegel_weil_mass_formula}) has also been conjectured in \cite{haensch_kane}, which may be stated as
	\begin{equation}\label{eqn::conjecture_avg_over_spn}
		\frac{1}{\sum_{L'+ \vec v' \in \mc S} |O^+(L'+ \vec v')|^{-1}} \sum_{L'+\vec v' \in \mc S} \frac{\Theta_{L' + \vec v'}}{|O^+(L'+\vec v')|} = \mc E_{\gen^+(L_{\vec a} + \vec v)} + \mc U_{\spn^+(L_{\vec a} + \vec v)}
	\end{equation}
	where $\mc S$ is a set of class representatives in the spinor genus of $L_{\vec a} + \vec v$ (see Section \ref{sctn::Algebraic_Theory}), $\mc E_{\gen^+(L_{\vec a} + \vec v)}$ is as given in equation (\ref{eqn::siegel_weil_mass_formula}), and $\mc U_{\spn^+(L_{\vec a} + \vec v)}$ is a linear combination of unary theta series (see Section \ref{sctn::theta_series}). The coefficients of a unary theta series are zero except for those corresponding to a specific square class, and so the coefficients of $\mc U$ are zero except for some specific square classes. It is known that equation (\ref{eqn::conjecture_avg_over_spn}) holds in the case of lattices (i.e. when $\vec v\in L_{\vec a}$), and this conjecture is a generalization to lattice cosets. Again, if our lattice coset has only one class in its spinor genus, then the average on the left collapses as before, thus giving us an equation of the form 
	\begin{equation}\label{eqn::pseudo_spn_class1_identity}
		\Theta_{\vec a, \vec h, N} = \mc E + \mc U
	\end{equation}
	where $\Theta_{\vec a, \vec h, N}$ is the generating function of $r_{\vec a, \vec h, N}(n)$. This again yields an identity involving $r_{\vec a, \vec h, N}(n)$, but now we have some correction terms (given by coefficients of $\mc U$) corresponding to certain square classes of $n$. Of course, if the lattice coset has class number 1 as well, then $\mc U \equiv 0$. Since we want to consider examples with non-trivial $\mc U$, throughout the rest of the paper whenever a lattice coset is said to have spinor class number 1, what is really meant is that it has spinor class number 1 but not class number 1. Equation (\ref{eqn::pseudo_spn_class1_identity}) leads to the following definition:
	\begin{definition}
		An identity involving $\Theta_{\vec a, \vec h, N}$ of the form given in equation (\ref{eqn::pseudo_spn_class1_identity}) to be a \emph{pseudo spinor class 1 identity}.
	\end{definition}
	Thus lattice cosets with spinor class number 1 (but not class number 1) are expected to satisfy pseudo spinor class 1 identities with non-zero $\mc U$. In view of the above discussion and definition, Theorem \ref{thm::spn_identities_elementary_form} may be reformulated and made more concrete. This is done in Theorem \ref{thm::spn_identities_concrete_form}, a proof of which is given in Section \ref{sctn::comp_search}. Theorem \ref{thm::spn_identities_elementary_form} follows immediately from Theorem \ref{thm::spn_identities_concrete_form}.
	
	To illustrate, consider $\vec a = (1,1,1)^T$, $\vec h = (1,0,0)^T$, and $N=4$. Phrasing this in elementary terms, given $n$, we want to count the number of $(x,y,z)$ such that $n = x^2+y^2+z^2$, where $x\equiv 1\modc{4}$, and $y$ and $z$ are divisible by 4. By Proposition \ref{prop::2cls_in_genus}, the corresponding lattice coset has spinor class number 1. As expected, the theta series $\Theta_{\vec a, \vec h, 4}$ is found to satisfy a pseudo spinor class 1 identity, which is given in Appendix \ref{sctn::list_spn1_identities}. Comparing coefficients then yields the identity \[ r_{\vec a, \vec h, 4}(n) = \begin{cases} 0 & \text{if } n \not \equiv 1 \modc{8},\\ \frac1{12}r_{\vec a}(n) & \text{if } n \equiv 1 \modc{8} \text{ and } n \ne k^2 \text{ for all } k\in\N \text{,} \\ \frac{1}{12}r_{\vec a}(n) + \frac{(-1)^{(k-1)/2}}{2} k & \text{if } n \equiv 1 \modc{8} \text{ and } n = k^2 \text{ for some odd } k\in\N. \end{cases} \]
	In the notation of Theorem \ref{thm::spn_identities_elementary_form}, we have $t=1$, $M=8$, and $d_{\vec a, \vec h, 4}(n) = \frac1{12}$ for $n\equiv 1\modc{8}$ and $d_{\vec a, \vec h, 4}(n)=0$ otherwise. Since $r_{\vec a}(n)$ is well-known (see \cite{kane_bringmann}), we now have a formula for $r_{\vec a, \vec h, 4}(n)$. An example of a more complicated identity is given by $$r_{\vec a,\vec h, 6}(n) = \begin{cases} 0 & \text{if } n\not \equiv9 \modc{36} \text{, or if } n \equiv 81 \modc{108},\\
		\frac{1}{8} r_\vec a(n)& \text{if } n\equiv9,45 \modc{216} \text{ and } n\ne 9k^2 \text{ where } k \equiv \pm 1\modc{6}, \\
		\frac{1}{8} r_\vec a(n)+ \frac12 k & \text{if } n\equiv9 \modc{216} \text{ and } n= 9k^2 \text{ for some } k\equiv 1\modc{6},\\
		\frac{1}{8} r_\vec a(n)- \frac12 k& \text{if } n\equiv9 \modc{216} \text{ and } n= 9k^2 \text{ for some } k\equiv 5\modc{6},\\
		\frac{1}{16} r_\vec a(n)& \text{if } n\equiv117 \modc{216},\\
		\frac{1}{4} r_\vec a(n)& \text{if } n\equiv153 \modc{216},
	\end{cases} $$
	where $\vec a = (1,3,9)^T$ and $\vec h = (0,0,1)^T$. The value of $d_{\vec a, \vec h, 4}(n)$ can be determined from the above expression. As mentioned previously, formulae for $r_{\vec a}(n)$ involving Hurwitz class numbers may be found in \cite{kane_bringmann}. 
	
	For most of the values $\vec a, \vec h, N$ in Theorem \ref{thm::spn_identities_concrete_form}, a spinor classification has also been found to check whether these lattice cosets have spinor class number 1. While in most cases the lattice coset indeed had spinor class number 1, in certain cases there were multiple classes in the spinor genus (see the discussion after Proposition \ref{prop::spnnum=3} in Section \ref{sctn::classify_found_cosets}). Thus the converse of equation (\ref{eqn::conjecture_avg_over_spn}), i.e. whether a lattice coset has only one class in its spinor genus if its theta series satisfies a pseudo spinor class 1 identity as given in equation (\ref{eqn::pseudo_spn_class1_identity}), is false.
	
	The method used to find the spinor classification of lattice cosets relies on the computation of the image of the group of $p$-adic rotations under the spinor norm map (see Sections \ref{sctn::Algebraic_Theory} and \ref{sctn::lattice_cosets} for details). Unlike in the case of lattices for which an algorithm is known, the 2-adic image under the spinor norm map in the case of lattice cosets cannot be computed very easily, and in some cases, it was not possible at all with currently known results. Assuming one can overcome this difficulty in computing the image under the spinor norm map, Section \ref{sctn::lattice_cosets} gives a technique for computing the spinor genus. 
	
	On the other hand, the genus classification is quite simple to carry out, and Section \ref{sctn::lattice_cosets} gives an algorithm which computes the genus of a lattice coset. This algorithm relies on Proposition \ref{prop::matrix_reduction_of_isometries}, which characterizes the group of $p$-adic rotations of a rank 3 lattice by a group of $3\times 3$ matrices modulo a power of the prime $p$. This latter group is straightforward to compute by hand in small cases (see for example Section \ref{sctn::classify_cond2}), and by computer in larger cases. The action of the group of $p$-adic rotations on a lattice coset is equivalent to the action of this group of matrices on 3-dimensional vectors whose entries lie in $\Z/p^k\Z$ for some $k$ (given by the proposition). Even though the algorithm given assumes the lattice $L_{\vec a}$ is of class number 1, it may be generalized to lattices with higher class numbers by simply running the algorithm on lattice cosets of a class representative of all classes in the genus of $L_{\vec a}$. 
	
	The paper is structured as follows: In Section \ref{sctn::preliminaries}, some important and useful results from the analytic and algebraic theory of quadratic forms have been recalled. Section \ref{sctn::comp_search} discusses pseudo spinor class 1 identities that were found, and also gives a proof of all 85 identities (see Theorem \ref{thm::spn_identities_concrete_form}). Section \ref{sctn::lattice_cosets} then goes into detail regarding the classification of the found lattice cosets, along with some general discussion on the classification of lattice cosets. In particular, Section \ref{sctn::classify_found_cosets} gives a classification of many of the lattice cosets satisfying pseudo-spinor class 1 identities, and Section \ref{sctn::classify_cond2} classifies the genera and spinor genera of all 2-adically unimodular lattice cosets with conductor 2. Finally, Appendix \ref{sctn::list_spn1_identities} lists all 85 of the found identities, and Appendices \ref{sctn::gen_info}, \ref{sctn::upperbnd_img_spn_norm}, \ref{sctn::grouping_spn_gen} contain information about the genus, the image of the spinor norm map, and the action of (adelized) rotations on lattice cosets respectively, all of which are required in the genus and spinor genus classification given in Section \ref{sctn::classify_found_cosets}. 
	
	\subsection*{Acknowledgements}
	
	The author would like to thank Dr Ben Kane for all of the generous help and advice given throughout this project. This project was supported by the Summer Research Fellowship scheme of the Faculty of Science at HKU.  
	
	\section{Preliminaries}\label{sctn::preliminaries}
	\subsection{Analytic Theory}\label{sctn::Analytic_Theory}
	\subsubsection{Modular Forms} 
	Some important definitions and preliminary results for modular forms are reproduced here; for further details one may refer to any book on modular forms (for instance \cite{koblitz} or \cite{ono}).
	
	For $\gamma = \left( \begin{smallmatrix} a&b\\c&d \end{smallmatrix} \right) \in SL_2(\Z)$ and $k\in \Z$, define the \emph{slash} operator on a meromorphic function $f:\mathbb{H}\to \C$ by \[f\big|[\gamma]_{k/2}:= j(\gamma, z)^{-k} f\left(\frac{az+b}{cz+d} \right)\]
	where $j(\gamma, z) := \left( \tfrac{c}{d}\right) \epsilon_d^{-1} \sqrt{cz+d} $, and where $\epsilon_d$ is 1 if $d\equiv 1\modc{4}$ and is $i$ if $d\equiv -1 \modc{4}$. Define the following \emph{congruence subgroups} of $SL_2(\Z)$
	\begin{align*}
		\Gamma_0(N) &:= \left\{\left( \begin{smallmatrix} a&b\\c&d \end{smallmatrix} \right) \in SL_2(\Z) : N|c \right\},\\
		\Gamma_1(N) &:= \left\{\left( \begin{smallmatrix} a&b\\c&d \end{smallmatrix} \right) \in \Gamma_0(N) : a,d\equiv 1\modc{N} \right\}.
	\end{align*}
	\begin{definition}[Modular Forms]
		Suppose $k\in \Z$, $\chi:\N\to \C^*$ a Dirichlet character, and $\Gamma \le \Gamma_0(4)$ a (congruence) subgroup of finite index containing $\left( \begin{smallmatrix} 1&1\\0&1 \end{smallmatrix} \right)$. Suppose $f:\mathbb{H} \to \C$ is a meromorphic function satisfying \[f\big|[\gamma]_{k/2} = \chi(d)f \qquad \qquad \forall \gamma = \left( \begin{smallmatrix} a&b\\c&d \end{smallmatrix} \right) \in \Gamma.\]
		Such a meromorphic function $f$ is said to satisfy \emph{modularity of weight $k/2$ and character $\chi$ for $\Gamma$}.\\
		If $f:\mathbb{H}\to \C$ is a holomorphic function satisfying modularity of weight $k/2$ and character $\chi$ for the congruence subgroup $\Gamma$, and if $f|[\gamma]_{k/2}$ is bounded as $z$ approaches $i\infty$ for all $\gamma\in SL_2(\Z)$, then $f$ is said to be a \emph{modular form} of weight $k/2$ and character $\chi$ for $\Gamma$. The space of modular forms of weight $k/2$, character $\chi$, and congruence subgroup $\Gamma$, will be denoted by $\M_{k/2}(\Gamma, \chi)$.
	\end{definition}
	
	If $f$ is a modular form (of some weight and character) for the congruence subgroup $\Gamma$ containing $\left( \begin{smallmatrix} 1&1\\0&1 \end{smallmatrix} \right)$, then $f$ admits a Fourier series expansion \[f(z) = \sum_{n=0}^\infty a_n q^n.\] Here, $q:= e^{2\pi i z}$.
	
	It is well-known that the space of modular forms $\M_{k/2}(\Gamma, \chi)$ is a finite-dimensional complex vector space. Moreover, as a result of the \emph{valence formula} for modular forms, we have the following lemma (c.f. \cite[Lemma 2.1]{kane_bringmann}). 
	\begin{lemma}\label{lem::valence_formula}
		Suppose $f(z) = \sum_{n=0}^\infty a_n q^n$ is a modular form of weight $k/2$ and character $\chi$ for the congruence subgroup $\Gamma$, and if $a_n = 0$ for all $n\le \frac{k}{12}[SL_2(\Z) : \Gamma]$, then $f$ is identically zero.
	\end{lemma}
	In particular, this means that to prove any identity involving modular forms, it suffices to show that the first few coefficients of the corresponding Fourier series expansions are equal.
	
	\subsubsection{Theta Series}\label{sctn::theta_series}
	For the quadratic form $Q_{\vec a}$ with the congruence conditions $\vec h$ and $N$, we may consider the generating function of $r_{\vec a, \vec h, N}(n)$, which is a \emph{theta series} given by \[\Theta_{\vec a, \vec h, N}(z) := \sum_{n=0}^\infty r_{\vec a, \vec h, N}(n) q^n\]
	(where $q:= e^{2\pi i z}$). Without congruence conditions, we have \[\Theta_{\vec a}(z) = \Theta_{\vec a, \vec 0, 1}(z) = \sum_{n=0}^\infty r_{\vec a}(n)q^n.\]
	
	Let $\chi_{m} := \left(\tfrac{m}{\cdot}\right)$ denote the Kronecker symbol; this is a Dirichlet character if $m\ne 0$ and $m\not\equiv 3\modc{4}$. Also, let $\Gamma_{N,M}$ denote the congruence subgroup given by \[\Gamma_{N,M} := \Gamma_0(N) \cap \Gamma_1(M) = \{\left( \begin{smallmatrix} a&b\\c&d \end{smallmatrix} \right) \in SL_2(\Z): a,d\equiv 1\modc{M}, M|c, N|c\}.\]
	It is easy to check that whenever $4|N$, the congruence subgroup $\Gamma_{N,M}$ is a subgroup of $\Gamma_0(4)$ containing $\left( \begin{smallmatrix} 1&1\\0&1 \end{smallmatrix} \right)$.
	
	We have the following lemma (for a proof see, for instance, \cite{kane_bringmann}).
	\begin{lemma}\label{lem::ThetaSeries}
		The theta series $\Theta_{\vec a, \vec h, N}$ is a modular form of weight $3/2$ and character $\chi_{4a_1a_2a_3}$ for the congruence subgroup $\Gamma_{4N^2l, N}$, where $l = l_{\vec a} := \lcm(a_1, a_2, a_3)$.\\
		In particular, the theta series $\Theta_{\vec a}$ is a modular form of weight $3/2$ and character $\chi_{4a_1a_2a_3}$ for the congruence subgroup $\Gamma_{4l,1} = \Gamma_0(4l)$.
	\end{lemma}
	In view of Lemma \ref{lem::valence_formula}, the following will be useful (for a proof see \cite[Lemma 2.2]{kane_bringmann}).
	\begin{lemma}\label{lem::grpindex_congsubgrp}
		If $N,M\in \N$ with $M|N$, then $ [SL_2(\Z) : \Gamma_{N,M}] = N M \left(\prod_{p|N} \left(1 + \tfrac1p\right)\right) \left(\prod_{p|M}\left(1 - \tfrac1p\right) \right) $.
	\end{lemma}
	
	We will also consider \emph{unary theta series}.
	\begin{definition}[Unary Theta Series]
		Given a character $\chi$ and an integer $t\in \N$, the \emph{unary theta series} $\theta_{\chi, t}$ is given by the following Fourier series expansion \[\theta_{\chi_{}, t}(z) = \sum_{n=0}^\infty \chi_{}(n) n q^{tn^2}.\]
		It is a modular form of weight $3/2$ and character $\chi_{}\chi_{-4t}$ for the congruence subgroup $\Gamma_0(4tf_{\chi_{}^2}) $, where $f_{\chi_{}}$ is the conductor of the character $\chi$ (this follows from \cite[Theorem 1.44]{ono} and \cite[Proposition 2.22(1)]{ono}).
	\end{definition}
	
	\subsubsection{Operators on Coefficients of Modular Forms}\label{sctn::sieve_modular_forms}
	In this paper, we will need to consider the following sieving operator defined using the Fourier series expansion of a modular form: \[f\big| S_{M,m}(z) = \sum_{n\equiv m\modc{M}} a_nq^n\]
	where $M,m\in \N$ are given, and $f(z) = \sum_{n=1}^\infty a_nq^n$ (with $q:= e^{2\pi i z}$) is the Fourier expansion of the modular form $f$. The following lemma describes the action of the sieve $|S_{M,m}$ on spaces of modular forms.
	\begin{lemma}\label{lem::sieve_modularity}
		Suppose $f\in \M_{k/2}(\Gamma_{N,L}, \chi_{})$ for some $N,L\in \N$, some weight $k/2$, and some character $\chi$ with conductor $f_{\chi_{}}$. Then, for any $M,m\in \N$, the function $f|S_{M,m}$ is a modular form. More precisely, we have $f|S_{M,m} \in \M_{k/2}(\Gamma_{N',L'}, \chi)$, where \[N' = \begin{cases}
			\lcm(N,M^2, f_{\chi_{}} M, ML) & \text{if } M \not\equiv 2\modc{4},\\
			\lcm(N,4M^2, f_{\chi_{}} M, ML) & \text{otherwise},\\
		\end{cases} \quad \text{ and } \quad L' = \begin{cases}
			L & \text{if } M|24,\\
			\lcm(M, L) & \text{otherwise.}	
		\end{cases}\]
	\end{lemma}
	A proof may be found, for instance, in \cite[Lemma 2.3]{kane_bringmann}.
	
	\subsection{Algebraic Theory} \label{sctn::Algebraic_Theory}
	On the algebraic side of things, we will use the theory of quadratic spaces and lattices as described in O'Meara's standard text \cite{omeara}; of course, rather than working in an arbitrary algebraic number field, we will only work in $\Q$ and the $p$-adic numbers $\Q_p$. The ring of integers in $\Q$ and $\Q_p$ will be denoted as usual by $\Z$ and $\Z_p$. We consider a \emph{quadratic space} $V = \Q^k$ which is a vector space equipped with the quadratic form $Q$ and associated symmetric bilinear form $B$. In this paper, we consider free \emph{lattices $L$ on $V$}, which are essentially free $\Z$-modules of $V$ generated by some basis $\{\vec v_1, ..., \vec v_k\}$ for $V$. Given a basis $\{\vec v_1, ..., \vec v_k\}$ for $L$, we may write $L\cong A$ where $A$ is the $k\times k$ Gram matrix $A = (B(\vec v_i, \vec v_j))$. The \emph{discriminant} $dL$ of a lattice is simply the determinant of $A$, and is well-defined and independent of the choice of basis up to multiplication by a unit. The \emph{localization} of a lattice $L$ at the prime $p$ is $L_p = L \otimes_\Z \Z_p$; the localization $L_p$ is a $\Z_p$-lattice in the localization $V_p = \Q_p^k$. The \emph{rank} of a lattice $L$ (resp. $L_p$) is the dimension of the quadratic space $\Q L$ (resp. $\Q_p L_p$), or equivalently is the rank when considered as a free module over $\Z$ (resp. over $\Z_p$). For a lattice $L$ and some $\vec v\in V$, we have the \emph{lattice coset} $L+\vec v$, which may be considered to be a coset in the quotient group $(L + \Z\vec v) / L$. The smallest positive integer $N$ such that $N\vec v \in L$ is called the \emph{conductor} of $\vec v$. 
	
	An \emph{isometry} $\sigma$ of the quadratic space $V$ equipped with the quadratic form $Q$ is a linear transformation $\sigma:V\to V$ satisfying $Q(\sigma \vec v) = Q(\vec v)$ for all $\vec v\in V$; the group of isometries of $V$ is denoted by $O(V)$. The subgroup $O^+(V)$ of \emph{rotations} of $V$ are those isometries with determinant 1. The coordinate representation $X\in SL_k(\Q)$ of a rotation must satisfy $X^TAX = A$, where $L\cong A$ in some basis. We similarly define the groups $O(V_p)$ and $O^+(V_p)$ of $p$-adic isometries and $p$-adic rotations respectively. The subgroup $O^+(L)$ of (integral) rotations of a lattice $L$ is \[O^+(L) := \{\sigma\in O^+(V): \sigma L = L\},\]
	while the subgroup $O^+(L_p)$ of $p$-adically integral rotations of a localized lattice $L_p$ is \[O^+(L_p) := \{\sigma \in O^+(V_p): \sigma L_p = L_p\}.\]
	As a consequence of \cite[82:12]{omeara}, an element of $O^+(V)$ belongs to $O^+(L)$ if its coordinate representation in a $\Z$-basis of $L$ (which serves as a basis for $V$) belongs to $SL_k(\Z)$ ($k=\rank L$). Similarly, an element of $O^+(V_p)$ belongs to $O^+(L_p)$ if its coordinate representation in a $\Z_p$-basis for $L_p$ belongs to $SL_k(\Z_p)$. Finally, the subgroup $O^+(L+\vec v)$ of $O^+(L)$ consists of those rotations $\sigma$ of $L$ satisfying $\sigma(L+\vec v) = L+\vec v$, or equivalently, satisfying $\sigma \vec v - \vec v \in L$. The subgroup $O^+(L_p+\vec v)$ of $O^+(L_p)$ may be defined similarly.
	
	In this paper, we will only be working with lattices and lattice cosets on quadratic spaces with rank 3. For any other unexplained notation please see \cite{omeara}.
		
	\subsubsection{Spinor Norm Map}\label{sctn::spinor_norm_map}
	
	As we will be using the spinor norm map repeatedly in later sections, some important definitions and results are given here. For more details see \cite{omeara}.
	
	For $\vec u\in V_p$, let $\tau_{\vec u}$ denote the \emph{symmetry} in $O(V_p)$ defined by $\tau_{\vec u}\vec x = \vec x - 2\frac{B(\vec u, \vec x)}{Q(\vec u)} \vec u$. It is known that the group $O(V_p)$ is generated by the symmetries in it. 
	\begin{definition}[Spinor norm map $\theta$]
		The spinor norm map $\theta: O(V_p) \to \Q_p^\times / (\Q_p^\times)^2 $ is defined on symmetries $\tau_{\vec u}$ by $\theta(\tau_{\vec u}) = Q(\vec u) (\Q_p^\times)^2 $. For arbitrary elements of $O(V_p)$, it is defined multiplicatively.
	\end{definition}
	
	The subgroup $\theta(O^+(L_p + \vec v))$ of $\Q_p^\times / (\Q_p^\times)^2$ for a prime $p$ is important in the classification of lattice cosets (see equation (\ref{eqn::num_spn_in_gen})). If $p$ does not divide the conductor of $\vec v$, then $\vec v\in L_p$, and so $\theta(O^+(L_p + \vec v)) = \theta(O^+(L_p))$. This latter quantity is easily computed given the Gram matrix of $L$, using the algorithm given below. To use the algorithm, we would need to decompose $L_p$ as follows.
	
	Suppose lattice $L$ has rank $k$. If $p$ is odd, then $L_p$ may be diagonalized in the form $L_p \cong \diag\{a_1, a_2, ..., a_k\}$. If $p=2$, then $L_p \cong A$ (in some basis) where $A$ is a block diagonal matrix whose diagonal entries are either a $1\times 1$ matrix $(qx)$ (where $2\nmid x$), or are a $2\times 2$ matrix of the form $q\left(\begin{smallmatrix} a&b\\b&c \end{smallmatrix} \right)$ where $b$ is even, and $a,c,ac-b^2$ are all odd. Here, $q$ is some power of $2$. This decomposition of $A$ corresponds to the \emph{Jordan decomposition} of the lattice $L_p$, i.e. writing $L_p$ as an orthogonal sum of unary and binary lattices (see \cite{omeara}; for an algorithm for such a decomposition see \cite[4.4, Chap.15]{conway}). 
	
	With this diagonalization, we may apply the below algorithm adapted from \cite[9.5, Chap.15]{conway} to compute $\theta(O^+(L_p))$. Here, for $p=2$, $A_q$ denotes the block diagonal matrix formed by all of those $1\times 1$ and $2\times 2$ matrices described above with given $q$ (a power of 2). 
	
	\begin{lemma}[Algorithm to compute $\theta(O^+(L_p))$] \label{lem::conway_algorithm}
		Suppose we have $L_p \cong A$, where $A$ is diagonalized as described above. First, we create two sets $S_1$ and $S_2$ as follows:
		\begin{itemize}
			\item[List 1:] If $p\ne 2$, then $S_1 = \{a_1, a_2, ..., a_k\}$ (i.e. all diagonal entries of $A$). 
			If $p=2$, then $S_1$ consists of those $qx$ that appear in the diagonalization of $A$ (i.e. those $qx$ corresponding to a unary lattice in the Jordan decomposition). 
			
			\item[List 2:] If $p\ne 2$, then $S_2=\emptyset$. Otherwise, $S_2$ consists of $2q(\Q_2^\times)^2, 3\cdot 2q(\Q_2^\times)^2, 5\cdot 2q(\Q_2^\times)^2, 7\cdot 2q(\Q_2^\times)^2$ for every $2\times 2$ matrix $q\left(\begin{smallmatrix} a&b\\b&c \end{smallmatrix} \right)$ occurring in the diagonalization of $A$.
		\end{itemize}
		Note that $S_1$ consists of elements of $\Q_p$ while $S_2$ consists of square classes in the quotient group $\Q_p^\times / (\Q_p^\times)^2$. Now, we create a supplement set $T$ given by 
		\begin{itemize}
			\item If $p\ne 2$ and there are two diagonal entries $a_i$ and $a_j$ ($i\ne j$) such that $a_i/a_j \in \Z_p^\times$, or if $p=2$ and the diagonal block matrix $\diag\{A_q, A_{2q}, A_{4q}, A_{8q}\}$ is of size $3\times 3$ or larger, then $T = \Z_p^\times (\Q_p^\times)^2$. 
			
			\item If $p=2$, then we include in $T$ the square classes as given by the following table 
			\begin{table}[H]
				\centering
				\begin{tabular}{|c|c|c|c|c|c|}
					\hline 
					Include & $2(\Q_2^\times)^2$ & $3\cdot 2(\Q_2^\times)^2$ & $5(\Q_2^\times)^2$ & $3(\Q_2^\times)^2$ & $7(\Q_2^\times)^2$\\\hline 
					\small{if $S_1$ contains} & \multirow{3}{*}{$(\Q_2^\times)^2$} & \multirow{3}{*}{$5(\Q_2^\times)^2$} & \multirow{3}{*}{$\{1,4,16\}\Z_2^\times$} & \multirow{3}{*}{$\{2,8,10,40\}(\Z_2^\times)^2$} & \multirow{3}{*}{$\{6,14,24,56\}(\Z_2^\times)^2$} \\
					\small{two elements} & & & & & \\
					\small{whose ratio is in:} & & & & & \\\hline  
				\end{tabular}
			\end{table}
		\end{itemize}
		Finally, let $S_1^2 = \{xy (\Q_p^\times)^2 : x,y\in S_1\}$ and $S_2^2 = \{xy : x,y\in S_2\}$. Then, $\theta(O^+(L_p))$ is the subgroup of $\Q_p^\times / (\Q_p^\times )^2$ generated by $S_1^2 \cup S_2^2 \cup T$. 
	\end{lemma}
	
	This algorithm gives an efficient method to compute $\theta(O^+(L_p))$. An extremely useful corollary is 
	\begin{corollary_lem}\label{cor::spnmap_img_of_modular_lattice}
		If $p$ is an odd prime, and if $L_p \cong \diag\{a_1, a_2, a_3\}$ with $a_i\in \Z_p^\times $ for $i=1,2,3$, then $\theta(O^+(L_p)) = \Z_p^\times (\Q_p^\times)^2$.
	\end{corollary_lem}
	\begin{proof}
		Clearly, $S_1 = \{a_1, a_2, a_3\} \subset \Z_p^\times$ and $S_2=\emptyset$ in the above lemma. Since $a_1/a_0 \in \Z_p^\times$, we have $T = \Z_p^\times (\Q_p^\times)^2$. The result follows from the lemma.
	\end{proof}
	For another proof of Corollary \ref{cor::spnmap_img_of_modular_lattice} that does not use the algorithm, see \cite[92:5]{omeara}.
	
	Lemma \ref{lem::conway_algorithm} allows us to compute $\theta(O^+(L_p + \vec v))$ for all primes $p$ not dividing the conductor of $\vec v$ (for then $L_p+v = L_p$). However, for primes dividing the conductor, no such easy result is known. The only result that can be used to compute $\theta(O^+(L_p + \vec v))$ is the following lemma, which is adapted from Theorem 2 in \cite{teterin}. 
	\begin{lemma}
		The group $\theta(O^+(L_p + \vec v)) $ for odd primes $p$ is generated by the spinor norm applied to the product of any two symmetries in $O(V_p)$ preserving $L_p+\vec v$.
	\end{lemma}
	This reduces the problem of finding $\theta(O^+(L_p + \vec v)) $ to the problem of characterizing symmetries that preserve the lattice coset.
	
	Another method for computing $\theta(O^+(L_p +\vec v))$ is to consider $M_p := \Z_p \vec v + L_p$. It is easy to see that $O^+(L_p + \vec v) \subseteq O^+(M_p)$, and so \(\theta(O^+(L_p + \vec v)) \subseteq \theta(O^+(M_p))\).
	In many cases, the lattice $M_p$ might have small $\theta(O^+(M_p))$, which then makes it easier to evaluate $\theta(O^+(L_p +\vec v))$. 
	
	\subsubsection{Classifications}\label{sctn::classifications}
	
	Let $O_{\A}(V)$ and $O_{\A}'(V)$ denote the \emph{adelization} of all rotations and the kernel of the spinor norm map $\theta$ in $V$ respectively. Then, the (proper) class, spinor genus, and genus of a lattice coset $L+\vec v$ is given by 
	\begin{align}
		\cls^+(L+\vec v) &= \text{orbit of } L+\vec v \text{ under the action of } O^+(V) \label{eqn::cls_defn}\\
		\spn^+(L+\vec v) &= \text{orbit of } L+\vec v \text{ under the action of } O^+(V)O'_{\A}(V) \label{eqn::spn_defn}\\ 
		\gen^+(L+\vec v) &= \text{orbit of } L+\vec v \text{ under the action of } O_{\A}(V)   \label{eqn::gen_defn}
	\end{align}
	respectively; see \cite{chan_oh} for further details. The \emph{class number} (resp. \emph{spinor class number}) of a lattice coset $L+\vec v$ is the number of classes in the genus (resp. spinor genera) of $L+\vec v$.
	
	The number of spinor genera in the genus of a lattice coset $L+\vec v$ is given by the following group index (c.f. \cite{feixu}) 
	\begin{equation}\label{eqn::num_spn_in_gen}
		\left[J_\Q : \Q^\times \prod_p \theta(O^+(L_p + \vec v))\right]
	\end{equation}
	where the product runs over all primes $p$, and $J_\Q$ is the set of all \idele s over $\Q$. This is extremely useful in classifying the classes in a genus into its various spinor genera.
	
	Throughout the rest of the paper, for the quadratic form $Q_{\vec a}$, $L = L_{\vec a}$ will denote the lattice whose gram matrix is $\diag \;\vec a$ in the standard basis $\vec e_1, \vec e_2, \vec e_3$ of $V = \Q^3$. For the congruence conditions $\vec h, N$ (where we may assume that $\vec h \in L$) we will consider the lattice coset $NL+\vec h$ (which is a lattice coset of the lattice $NL$). Since $O^+(NL) = O^+(L)$, in practice we will only need to consider $NL+\vec h$ and $L$ (and their localizations). Finally, we will assume that $\vec h$ is \emph{primitive}, i.e. those $\vec h = (h_1, h_2, h_3)^T$ such that $\gcd(h_i, N)=1$ for some $i=1,2,3$, or equivalently those $\vec h$ that have conductor $N$.
	
	\section{Pseudo Spinor Class 1 Identities} \label{sctn::comp_search}
	Using a computer, lattice cosets were found whose corresponding theta series satisfied identities of the type 
	\begin{equation}\label{eqn::pseudo_spn1_id}
		\Theta_{\vec a, \vec h, N} = \mc E + \mc U
	\end{equation}
	where $\mc E$ is an Eisenstein series and $\mc U$ is a unary theta series. Such an identity was called a \emph{pseudo spinor class 1 identity} since the theta series of lattice cosets with spinor class number 1 (are expected to) satisfy such identities. In this computer search, we have the following restrictions on $\mc E$ and $\mc U$.
	\begin{itemize}
		\item $\mc E$ is a linear combination of $\Theta_{\vec a}\big|S_{M,m}$ for $M,m\in \N$. Essentially, we are adding the extra restriction that all solutions $\vec x$ to $Q_\vec a(\vec x) = n $ with $\vec x\equiv \vec h\modc{N}$ are primitive solutions.
		\item $\mc U$ consists of unary theta series for a fixed value of $t\in \N$, where the unary theta series are of the form $c \cdot \theta_{\chi, t}$ or $c \cdot \theta_{\chi, t}\big| S_{M,m}$ for some constant $c$, some character $\chi$, and some $M,m\in \N$.
	\end{itemize}
	These restrictions were placed only to simplify the search process. 
	The results of the computer search are given in the following theorem.
	\begin{theorem}\label{thm::spn_identities_concrete_form}
		For all 85 values of $\vec a, \vec h$, and $N$ given in Table \ref{table::list_spn1_identities} in Appendix \ref{sctn::list_spn1_identities}, the theta series $\Theta_{\vec a, \vec h, N}$ satisfies a pseudo-spinor class 1 identity. More concretely, we have the following identity \[\Theta_{\vec a, \vec h, N} = \mc E + \mc U\] 
		where $\mc E$ and $\mc U$ are given in Table \ref{table::list_spn1_identities}.
	\end{theorem}
	\begin{proof}
		For each given $\vec a$, $\vec h$, and $N$, we consider $f_{\vec a, \vec h, N} := \Theta_{\vec a, \vec h, N} - \mc E - \mc U$. Using the results listed in Sections \ref{sctn::theta_series} and \ref{sctn::sieve_modular_forms}, and noting that if $\Gamma'\le \Gamma$ then $\M_{k/2}(\Gamma, \chi) \subseteq \M_{k/2}(\Gamma', \chi)$, it can be seen that $f_{\vec a, \vec h, N}$ is a modular form of weight $3/2$ and congruence subgroup $\Gamma$, where $\Gamma$ is given in Table \ref{table::list_spn1_identities}. Lemma \ref{lem::grpindex_congsubgrp} is then used to compute $[SL_2(\Z): \Gamma]$. In view of Lemma \ref{lem::valence_formula}, it suffices to check that all coefficients $c_n$ of $f_{\vec a, \vec h, N}$ are zero for all positive integers $n$ less than \[\frac1{12} \cdot \frac32 [SL_2(\Z): \Gamma] = \frac18 [SL_2(\Z): \Gamma].\]
		The value $\frac18 [SL_2(\Z): \Gamma]$ is given in the last column of Table \ref{table::list_spn1_identities}. Using a computer, it can then be checked that $c_n = 0$ for all $n\le \frac18 [SL_2(\Z):\Gamma]$, which then establishes the required identity. Such a computer verification up to the valence bound was carried out for all 85 identities. 
	\end{proof}
	
	\section{Classification of Lattice Cosets} \label{sctn::lattice_cosets}	
	The general strategy to classify the spinor genus of the lattice coset $NL_{\vec a} + \vec h$ is to first find the genus of $NL_{\vec a} + \vec h$, and then to calculate the number of spinor genera in the genus (using equation (\ref{eqn::num_spn_in_gen})). Once we know the genus and the number of spinor genera, it is then a matter of finding which lattice cosets in the genus belong to the same orbit under the group action $O^+(V)O'_{\A}(V)$. If we can find elements of $O^+(V)O'_{\A}(V)$ that map certain lattice cosets to other lattice cosets, then by knowing the number of orbits expected (i.e. the number of spinor genera), we can establish the decomposition into spinor genera. This general strategy is used in \cite{liang} as well as in \cite{haensch_kane}.
	
	The classification of the genus is fairly straightforward, and makes use of the following results.
	
	\begin{lemma}\label{lem::characterize_O(L)}
		Suppose $L \cong \diag\{a_1, a_2, a_3\}$ in the basis $\vec v_1, \vec v_2, \vec v_3$ for $V$. The group $O^+(L)$ is then given by \[O^+(L) = \{\sigma\in O^+(V): \sigma(\vec v_i) = \pm \vec v_{j_i} \text{, where } a_i = a_{j_i} \text{, and } j_1,j_2,j_3 \text{ is a permutation of } 1,2,3\}.\]
	\end{lemma}
	\begin{proof}
		It is easy to see that $O^+(L)$ contains all isometries given in the above set. If $\sigma\in O^+(L)$, then $L\cong \left< \sigma \vec v_1 \right> \perp \left< \sigma \vec v_2 \right> \perp \left< \sigma \vec v_3\right> \cong \left<\vec v_1 \right> \perp \left<\vec v_2 \right> \perp \left<\vec v_3\right>$, where $Q(\sigma \vec v_i) = Q(\vec v_i)$. However, such decompositions are unique by \cite[105:1]{omeara}. The result follows.
	\end{proof}
	Thus the action of $O^+(L)$ on $\vec v = (v_1, v_2, v_3)$ is essentially to permute $v_1, v_2, v_3$ among themselves, along with some sign changes. Now, if $L_{\vec a}$ has class number 1, then the proper genus of the lattice coset $NL_{\vec a}+\vec h$ consists of classes of the form $\cls^+(NL_{\vec a} + \vec h')$. Since $\cls^+(NL_{\vec a} + \vec h')$ is the orbit of $NL_{\vec a} + \vec h'$ under $O^+(L_{\vec a})$, and since $O^+(L_{\vec a})$ is a finite group by the lemma, the class is easily enumerated once we have a class representative.
	
	\begin{lemma}\label{lem::Mass_Formula}
		Suppose $L$ is a definite lattice with class number 1, and consider the lattice coset $NL+\vec h$, where $\vec h\in L$. Suppose $\mc G$ is a set of class representatives of all classes in the genus of $NL+\vec h$. Then, 
		\begin{equation}\label{eqn::mass_formula}
			\sum_{NL+\vec h' \in \mc G} \frac{1}{|O^+(NL+ \vec h')|} = \frac{1}{|O^+(L)|} \prod_{p|N} [O^+(L_p) : O^+(NL_p + \vec h)].
		\end{equation}
	\end{lemma}
	\begin{proof}
		From the Mass Formula for lattices as well as for lattice cosets we have the following two equations (c.f. \cite[Corollary 2.5]{liang}).
		\begin{align*}
			\sum_{NL+\vec h' \in \mc G} \frac{1}{|O^+(NL+ \vec h')|} &= 2\mu_{\infty_\Q}(O^+(V_{\infty_\Q}))^{-1} \prod_p \mu_p(O^+(NL_p+\vec h))^{-1}\\
			\frac{1}{|O^+(L)|} &= 2\mu_{\infty_\Q}(O^+(V_{\infty_\Q}))^{-1} \prod_p \mu_p(O^+(L_p))^{-1}
		\end{align*}
		where the product runs over all primes $p$, and $\mu$ is the Tamagawa measure on $O^+_{\A}(V)$. Notice that there is only one term on the left of the second equation as $L$ has class number 1. Dividing gives 
		\[\sum_{NL+\vec h' \in \mc G} \frac{1}{|O^+(NL+ \vec h')|} = \frac{1}{|O^+(L)|} \prod_{p} \frac{\mu_p(O^+ (L_p))}{\mu_p(O^+(N L_p + \vec h))}.\]
		For $p\nmid N$, we have $NL_p + \vec h = L_p$, while for $p|N$ we have \[\mu_p(O^+(L_p)) = \mu_p(O^+(NL_p+\vec h)) [O^+(L_p) : O^+(NL_p + \vec h)].\]
		Substituting back gives the required result. 
	\end{proof}
	Thus, if we are given a list of distinct class representatives in a genus, we can use the above lemma to check whether these exhaust all classes in the genus. 
	Let us now try and compute $[O^+(L_p) : O^+(NL_p + \vec h)]$.
	
	Suppose $L\cong \diag\{a_1, a_2, a_3\} =: A$ in the standard basis $\vec e_1, \vec e_2, \vec e_3$ for $\Q^3$, and consider the lattice coset $NL+\vec h$ where $\vec h$ has conductor $N$. For the statement (and proof) of Proposition \ref{prop::matrix_reduction_of_isometries}, $L$ need not have class number 1. Suppose $p$ is a prime dividing $N$. Let  $k := \ord_p N$, $\ell := \ord_p(\lcm\{a_1,a_2,a_3\})$, and let $\delta$ be given by \[\delta = \begin{cases}
		0 & \text{if } p \ne 2,\\
		1 & \text{if } p=2 \text{ and } k\ge 2,\\
		2 & \text{if } p=2 \text{ and } k = 1.
	\end{cases}\]
	Let $q = p^k$ and $q' = p^{k+\ell+\delta}$. Consider the following two groups:
	\begin{equation}\label{eqn::defn_G}
		G = G_{L,q',q} := \{X + qM_3(\Z/q'\Z) : X \in SL_3(\Z/q'\Z), X^TAX \equiv A \modc{q'}\},	
	\end{equation}
	\begin{equation}\label{eqn::defn_H}
		H = H_{L,\vec h, q',q} = \{X\in G : X\vec h \equiv \vec h \modc{q}\}.
	\end{equation}
	Essentially, $G$ is the reduction modulo $q$ of the group of matrices in $M_3(\Z/q'\Z)$ that preserve $A$ and have determinant 1.
	\begin{proposition}\label{prop::matrix_reduction_of_isometries}
		Under the above assumptions and notation, there is a natural surjective homomorphism $\Psi$ of groups from $O^+(L_p)$ onto $G$, where $\Psi$ maps $\sigma \in O^+(L_p)$ to the coordinate matrix of $\sigma$ in the basis $\{\vec e_i\}$, reduced modulo $q$. Moreover, $\Psi$ maps $O^+(NL_p+\vec h)$ onto $H$.
	\end{proposition}
	\begin{proof}
		First consider the homomorphism $\Psi_1$ mapping $\sigma \in O^+(L_p)$ to its coordinate matrix in the standard basis; this image is in $SL_3(\Z_p)$ since $\sigma $ preserves $L_p = \Z_p\vec e_1 + \Z_p \vec e_2 + \Z_p\vec e_3$. The image in $SL_3(\Z_p)$ of this map is \[G' = \{X\in SL_3(\Z_p): X^TAX = A, \det X = 1\}\]
		since $\sigma$ is a rotation. Next, consider the homomorphism $\Psi_2$ which takes a matrix in $G'$ and reduces modulo $q$ to a matrix in $G$. Note that, as $\Z_p/q\Z_p \cong \Z/q\Z$, reduction modulo $q$ does indeed give an element of $\Z/q\Z$. As $\Psi = \Psi_2 \circ \Psi_1$, it follows that $\Psi$ is a well-defined homomorphism from $O^+(L)$ into $G$. The surjectivity of $\Psi$ follows if $\Psi_2$ is surjective; this we now show.
		
		Without loss of generality, we may suppose that $a_1\in \Z_p^\times$, and $1 = |a_1|_p \ge |a_2|_p \ge |a_3|_p$. Then, $\ell = \ord_pa_3$. Consider any $X\in G$, and we may as well suppose that $X\in M_3(\Z_p)$, that $X^TAX\equiv A\modc{q'}$, and that $\det X \equiv 1\modc{q'}$. We will be done if we can find $Y\in G'$ such that $Y\equiv X\modc{q}$. 
		
		Let $X = (x_{ij})$. Comparing the $(1,2)$-entry in $X^TAX\equiv A\modc{M}$ gives $\sum_{i=1}^3a_ix_{i1}x_{i2} \equiv 0\modc{q'}$, and so $\eta := -\frac1{q'} \sum_{i=1}^3 a_ix_{i1} x_{i2} \in \Z_p$. As $\sum_{i=1}^3 a_ix_{i1}^2 \equiv a_1 \modc{q'}$ where $p|q'$ and $a_1\in \Z_p^\times$, it follows that the ideal $\left<a_ix_{i1}:i=1,2,3\right>$ in $\Z_p$ contains a unit, and so is equal to $\Z_p$. Thus, there exists $\xi_i\in \Z_p$ such that $$\sum_{i=1}^3 \xi_i\cdot a_ix_{i1} = \eta = -\frac1{q'} \sum_{i=1}^3 a_ix_{i1} x_{i2}$$
		and so $\sum_{i=1}^3 a_ix_{i1} (x_{i2} + q'\xi_i) = 0 $. Let $X_1$ be the matrix obtained from $X$ by adding $q'\xi_i$ to the $(i,2)$ entry. Then, $X_1\equiv X\modc{q'}$, which implies $X_1^TAX_1 \equiv A \modc{q'}$. Moreover, by construction, the $(1,2)$-entry (and hence the $(2,1)$-entry) of $X_1^TAX_1$ is zero. Let $X_1 = (x'_{ij})$ (so $x'_{ij} = x_{ij}$ for $j\ne 2$, and $x'_{i2} = x_{i2} + \xi_iq'$). Since the $(i,j)$-entry of $(X_1^T)^{-1}$ is $(-1)^{i+j}/\det X_1$ times the determinant of the $2\times 2$ matrix obtained from $X_1$ by deleting the $i$'th row and $j$'th column, and since $\det X_1 \equiv 1\modc{q'}$, the following three congruences then follow from $AX_1 \equiv (X_1^T)^{-1}A \modc{q'}$
		\begin{align*}
			a_1x'_{13} &\equiv a_3(x_{21}'x_{32}' - x_{31}'x_{22}') \modc{q'},\\	
			a_2x'_{23} &\equiv a_3(x_{12}'x_{31}' - x_{11}'x_{32}') \modc{q'},\\
			a_3x'_{33} &\equiv a_3(x_{11}'x_{22}' - x_{21}'x_{12}') \modc{q'}.
		\end{align*}
		Now, notice that $q'\Z_p = p^\delta a_3q\Z_p$, and so there exists $\alpha_1,\alpha_2,\alpha_3\in \Z_p$ such that 
		\begin{align*}
			a_1x'_{13} &= a_3(x_{21}'x_{32}' - x_{31}'x_{22}') + p^\delta a_3q \alpha_1,\\		
			a_2x'_{23} &= a_3(x_{12}'x_{31}' - x_{11}'x_{32}') + p^\delta a_3q \alpha_2, \\
			a_3x'_{33} &= a_3(x_{11}'x_{22}' - x_{21}'x_{12}') + p^\delta a_3q\alpha_3.
		\end{align*}
		A straightforward computation then yields 
		\[\sum_{i=1}^3 a_i x'_{ij} \left(x'_{i3} - p^\delta q \frac{a_{3}}{a_i}\alpha_i \right) = 0 \qquad \text{for } j=1,2,\]
		where $a_3/a_1, a_3/a_2 \in \Z_p$. Consider $X_2 = (x_{ij}'')$ given by $x_{ij}'' = x_{ij}'$ for $j=1,2$, and $x_{i3}'' = x'_{i3} - p^\delta q \frac{a_{3}}{a_i}\alpha_i$; then $X_2 \equiv X_1 \equiv X \modc{p^\delta q}$ and $X_2^TAX_2 = \diag\{\mu_1, \mu_2, \mu_3\}$ where 
		\begin{align*}
			\mu_1 &= \sum_{i=1}^3 a_ix_{i1}^2 \\
			\mu_2 &= \sum_{i=1}^3 a_i(x'_{i2})^2 = \sum_{i=1}^3 a_i(x_{i2} + q'\xi_i)^2 \\
			\mu_3 &= \sum_{i=1}^3 a_i(x''_{i3})^2 = \sum_{i=1}^3 a_i \left(x_{i3} - p^\delta q \frac{a_{3}}{a_i} \alpha_i\right)^2 .
		\end{align*}
		For $j=1,2$, it is clear that $\mu_j \equiv \sum_{i=1}^3 a_ix_{ij}^2 \equiv a_j \modc{q'}$. For $\mu_3$, a simple computation gives 
		\begin{align*}
			\mu_3 &= \sum_{i=1}^3 a_ix_{i3}^2 + p^\delta a_3q\left(-2\alpha_1 x_{13} + p^\delta q\left( \frac{a_3}{a_1} \right) \alpha_1^2 - 2\alpha_2 x_{23} + p^\delta q\left( \frac{a_3}{a_2} \right) \alpha_2^2 - 2\alpha_3 x_{33} + p^\delta q\alpha_3^2\right) \\ &\equiv a_3\modc{q'}.
		\end{align*}
		Thus $\mu_j/a_j \equiv 1 \modc{p^\delta q}$ for $j=1,2,3$. The local square theorem (see \cite[63:1]{omeara}) then implies that $\mu_j = a_j(\mu_j')^2$ for some $\mu_j'\in \Z_p^\times $ satisfying $\mu_j' \equiv 1\modc{q}$ (the $p^\delta$ term guarantees these congruences when $p=2$). Therefore, the matrix $Y := (\diag\{\mu_1', \mu_2', \mu_3'\})^{-1}X_2$ satisfies $Y\equiv X\modc{q}$ and $$Y^TAY = (\diag\{(\mu_1')^{-2}, (\mu_2')^{-2}, (\mu_3')^{-2}\})X^TAX = A.$$ 
		Finally, as $\det Y\equiv \det X \equiv 1\modc{q}$ and $(\det Y)^2 = 1$, it follows that $\det Y = 1$ if $q>2$. For $q=2$ however, notice that we may flip the sign of any of the $\mu_j$'s if necessary so that $\det Y = 1$ in this case as well; as $-1\equiv 1 \modc{q}$ we still maintain $Y\equiv X\modc{q}$. Hence, we have found $Y\in G'$ such that $Y\equiv X\modc{q}$, and so $\Psi$ is surjective.
		
		Finally, it is clear that $\Psi$ maps $O^+(NL_p+\vec h)$ into $H$, since any $\sigma \in O^+(NL_p + \vec h)$ preserves $\vec h$ modulo $NL_p = N\Z_p \vec e_1 + N\Z_p \vec e_2 + N\Z_p \vec e_3$. On the other hand, for $X\in H$, by the surjectivity of $\Psi$ we have $\sigma \in O^+(L_p)$ such that $\sigma \vec h \equiv X\vec h \equiv \vec h \modc{q}$ (here, $\sigma \vec h = Y\vec h$ where $Y = \Psi_1\sigma \equiv X\modc{q}$ is the coordinate matrix of $\sigma$ in the standard basis $\vec e_i$). Thus $\sigma \vec h \equiv \vec h\modc{NL_p}$ (noting that $N\Z_p = q\Z_p$), and so $\sigma(NL_p + \vec h) = NL_p + \vec h$. Hence $\Psi$ maps $O^+(NL_p+\vec h)$ onto $H$. 
	\end{proof}
	
	\begin{corollary_prop}\label{cor::grp_index_using_matrix_reduction}
		With the same conditions and notation of Proposition \ref{prop::matrix_reduction_of_isometries}, we have $$\displaystyle [O^+(L_p): O^+(NL_p+\vec h)] = [G_{L,q',q} : H_{L,\vec h, q',q}] = \frac{|G_{L,q',q}|}{|H_{L,\vec h, q',q}|}.$$
	\end{corollary_prop}
	\begin{proof}
		This follows from elementary group theory, noting that $\Psi(O^+(L_p)) = G$, $\Psi(O^+(NL_p+\vec h)) = H$, and $\ker \Psi\subset O^+(NL_p+\vec h)$.
	\end{proof}
	From a computational viewpoint Proposition \ref{prop::matrix_reduction_of_isometries} is useful. The orbit of a lattice coset $NL_p + \vec h$ ($\vec h = h_1 \vec e_1 + h_2 \vec e_2 + h_3\vec e_3$) under the infinite group $O^+(L_p)$ is the same as the orbit of vectors $(h_1, h_2, h_3) \in (\Z/q\Z)^3$ under the finite group $G$ (the latter group action is simply multiplication of a matrix with a vector modulo $q$). If $L$ is a class 1 lattice, then the classes in the orbit of $NL+\vec h$ under $O_{\A}(V)$ are the same as the classes in the orbit of $NL+\vec h$ under $O_{\A}(L) = \prod_p O^+(L_p)$, the latter of which can be computed by the action of $\prod_p G_{L,q',q}$ on $\vec h \modc{q}$. Since $G_{L,q',q}$ is finite, and since we only need to study the action of $G_{L,q',q}$ on $NL_p+\vec h$ for the finite number of primes $p|N$, it is possible for a computer to list the genus of $NL+\vec h$ by listing out all elements in $G_{L,q',q}$. Lemma \ref{lem::Mass_Formula} can then be used to prove that this list is complete, where $[O^+(L_p): O^+(NL_p+\vec h)]$ is computed using Corollary \ref{cor::grp_index_using_matrix_reduction} (the computer should have already listed out all elements of $G$). Of course, for extremely large values of $N$ and $dL$, this computation may not necessarily be feasible due to the large value of $q'$. For reasonable values on the other hand, we now have a simple algorithm to find the genus of $NL+\vec h$ whenever $L$ has class number 1. 
	
	This leaves the problem of listing out elements of $G_{L,q',q}=G$. In small cases, this is an elementary exercise in solving a system of congruences (an example of which is carried out in section \ref{sctn::classify_cond2}). In larger cases, a computer may be able to list out all matrices $X\in SL_3(\Z/q'\Z)$ satisfying $X^TAX \equiv A\modc{q'}$ and reducing modulo $q$. A simple straightforward method would be to first find all $\vec x_j = (x_{1j}, x_{2j}, x_{3j})^T\in (\Z/q'\Z)^3$ satisfying $a_1x_{1j}^2 + a_2x_{2j}^2 + a_3x_{3j}^2 \equiv a_j \modc{q'}$ for each $j=1,2,3$, and then go through every choice of columns (one from each list) and check whether $X = [\vec x_1 \,\, \vec x_2 \,\, \vec x_3]$ satisfies all conditions. This brute force approach can be quite slow for large values of $q'$, but smarter, more efficient approaches to listing out all elements of $G_{L,q',q}$ can reduce computation time.
	
	The next step in the classification of spinor genera would be to compute the number of spinor genera, using equation (\ref{eqn::num_spn_in_gen}). The following two lemmas do this in some special cases.
	%
	%
	\begin{lemma}\label{lem::numspn_theta_triv_p_ne_q}
		Fix a prime number $q$. Suppose $NL+\vec h$ is a lattice coset such that the following conditions hold
		\begin{itemize}
			\item $\theta(O^+(NL_p+\vec h)) = \Z_p^\times (\Q_p^\times )^2$ for all primes $p\ne q$
			\item Either $\theta(O^+(NL_q+\vec h)) \subseteq \Z_q^\times (\Q_q^\times)^2$, or $q(\Q_q^\times)^2 \in \theta(O^+(NL_q+\vec h)) $.
		\end{itemize}
		Then the number of spinor genera in $\gen^+(NL+\vec h)$ is $[\Z_q^\times (\Q_q^\times )^2:\theta(O^+(NL_q+\vec h)) \cap \Z_q^\times (\Q_q^\times )^2]$.
	\end{lemma}
	\begin{proof}
		Let $n = [\Z_q^\times (\Q_q^\times )^2:\theta(O^+(NL_q+\vec h)) \cap \Z_q^\times (\Q_q^\times )^2]$ and $P = \prod_p \theta(O^+(NL_p + \vec h))$. By equation (\ref{eqn::num_spn_in_gen}), we just need to prove that $[J_\Q : \Q^\times P] = n$. Let $\mf i^{(k)}$ ($1\le k\le n$) be \idele s in $J_\Q$ such that $\mf i_p^{(k)} = 1$ for and $p\ne q$ all $1\le k\le n$, and such that $\{\mf i_q^{(k)} : 1\le k\le n\}$ is a complete set of coset-representatives of the quotient group $\Z_q^\times (\Q_q^\times )^2/ \left(\theta(O^+(NL_q+\vec h)) \cap \Z_q^\times (\Q_q^\times )^2\right)$, where (without loss of generality), each of the $\mf i_q^{(k)}$ are $q$-adic units, and $\mf i_q^{(1)} = 1$. We will be done if we show that $\{\mf i^{(k)} : k=1,...,n\}$ is a complete set of coset-representatives of $J_\Q / (\Q^\times P)$.
		
		To check that $\{\mf i^{(k)}:1\le k\le n\}$ is a set of $n$ distinct coset representatives of $J_\Q$ modulo $\Q^\times P$, notice that $\mf i^{(k)}(\mf i^{(k')})^{-1} \in \Q^\times P$ is equivalent to $\mf i^{(m)}\in \Q^\times P$ where $m=1,2,...,n$ satisfies $(\mf i^{(m)}_q) \in \mf i_q^{(k)}(\mf i_q^{(k')})^{-1} (\Q^\times P)$ (such an $m$ exists as the $\mf i_q^{(k)}$ are a a complete of set of representatives modulo $\Q^\times P$). Thus it suffices to check that each of the $\mf i^{(k)}$ are non-trivial for $k\ge 2$ modulo $\Q^\times P$. So suppose that $\mf i^{(k)}$ ($k\ge 2$) is trivial, i.e. there exists $x\in \Q^\times$ such that $x\mf i^{(k)}\in P$. Then $x = x\mf i_p^{(k)}\in \Z_p^\times (\Q_p^\times )^2$ for all $p\ne q$, and $x\mf i_q^{(k)}\in \theta(O^+(NL_q+\vec h))$. The former implies that $\ord_p x$ is even for all primes $p\ne q$, and so $x \in \Q^2 \cup q\Q^2$. If $x = y^2$ for some $y\in \Q^\times $, then $x\in (\Q_q^\times )^2$ and so $\mf i_q^{(k)} \in \theta(O^+(NL_q+\vec h))$, contradicting the assumption that $\mf i_q^{(k)}$ is a $q$-adic unit not in $\theta(O^+(NL_q + \vec h))$. Now suppose $x=qy^2$ for some $y\in \Q^\times$; then $q\mf i_q^{(k)} \in \theta(O^+(NL_q+\vec h))$. By the second condition, this implies that $q\in \theta(O^+(NL_q+\vec h))$, and so $\mf i_q^{(k)} \in \theta(O^+(NL_q+\vec h))$, which is again a contradiction. 
		
		Now, consider any \idele{} $\mf j\in J_\Q$. Since $\mf j_p$ is a unit for almost all $p$, by multiplying by some suitable non-zero rational number $x$ whose sign is the same as that of $\mf j_\infty$, we may suppose that $x\mf j_p$ is a unit for all primes $p$. By construction $x\mf j_q$ is in $\mf i^{(k)}_q \theta(O^+( NL_q +\vec h))$ for some $k$, and thus $x\mf j$ is in $\mf i^{(k)} P$, or equivalently $\mf j \in \mf i^{(k)} (\Q^\times P)$. Hence $\{\mf i^{(k)} : 1\le k\le n\}$ is a complete set of coset representatives.
	\end{proof}
	
	In light of Corollary \ref{cor::spnmap_img_of_modular_lattice}, the above lemma is useful in the special case where $N$ and $dL$ are both powers of a prime. However, if this does not hold, the requirement that $\theta(O^+(NL_p+\vec h)) = \Z_p^\times (\Q_p^\times)^2$ for all primes $p\ne q$ may be too strict, in which case the below lemma is useful as well.
	
	\begin{lemma}\label{lem::numspn=2_ind2subgrp_p=2_nonon_trivial_unit_squares_p=q}
		Fix two prime numbers $q$ and $q'$, where $q'\ne 2$, and $q$ is a quadratic nonresidue modulo $q'$. Suppose $NL+\vec h$ is a lattice coset such that 
		\begin{itemize}
			\item $\theta(O^+(NL_p+\vec h)) = \Z_p^\times (\Q_p^\times)^2$ for all $p\ne q,q'$.
			\item $\theta(O^+(NL_{q'} + \vec h)) = \{1, q'\}(\Q_{q'}^ \times)^2$.
			\item $\theta(O^+(NL_q+ \vec h))$ is a subgroup of $\Z_q (\Q_q^\times)^2$ with group index at most 2.
			\item If $\theta(O^+(NL_q+ \vec h))\ne \Z_q(\Q_q^\times)^2$, then $q' (\Q_{q}^\times)^2 \not \in \theta(O^+(NL_q+\vec h))$.
		\end{itemize}
		Then, there are two spinor genera in the genus of $NL+\vec h$.
	\end{lemma}
	\begin{proof}
		If $\theta(O^+(NL_q + \vec h))$ has group index 1 in $\Z_q(\Q_q^\times)^2$, then $\theta(O^+(NL_q + \vec h)) = \Z_q(\Q_q^\times)^2$, and so the previous lemma is applicable. Noting that $\Z_{q'}(\Q_{q'}^\times)^2 \cap \theta(O^+(NL_{q'}+\vec h))$ is the trivial subgroup of the order 2 group $\Z_{q'}(\Q_{q'}^\times)^2$, it follows that there are two spinor genera in the genus. From now onwards, suppose that $\theta(O^+(NL_q+ \vec h))$ has group index exactly 2 in $\Z_q (\Q_q^\times)^2$, which also implies that $q' (\Q_{q}^\times)^2 \not \in \theta(O^+(NL_q+\vec h))$.
		
		Let $P = \prod_p \theta(O^+(NL_p + \vec h))$; we want to show that $\Q^\times P $ is an index 2 subgroup of $J_\Q$. Consider $\mf i \in J_\Q$ such that $\mf i_{q'} = q$, and $\mf i_p = 1$ for all $p\ne q'$. Note that $\mf i_{q'}$ is a coset representative of the only non-trivial coset of $\Q_{q'}^\times/(\Q_{q'}^\times)^2$ modulo $\theta(O^+(NL_{q'}+ \vec h))$, since $q$ is a quadratic non-residue modulo $q'$. Suppose, to the contrary, that $\mf i \in \Q^\times P$. There must then exist $x\in \Q^\times$ such that $x\mf i\in P$. From $x = x\mf i_p \in \Z_p^\times (\Q_p^\times)^2$ for $p\ne q,q'$, it follows that $\ord_p x$ is even for these primes. Hence, there exists $y\in \Q^\times$ such that $x=ty^2$, where $t=1,q,q'$, or $qq'$. From $qx = x\mf i_{q'} \in \theta(O^+(NL_{q'} + \vec h))$ it follows that $qt \in \{1,q'\} (\Q_{q'}^\times)^2$; as $q$ is not in $(\Q_{q'}^\times)^2$ and $t\in \{1,q,q',qq'\}$, we must have $t = q$ or $t=qq'$. Either way, $x = ty^2 \not \in \Z_q^\times (\Q_q^\times)^2$, contradicting the fact that $x = x\mf i_q \in \theta(O^+(NL_q + \vec h)) \subset \Z_q (\Q_q^\times)^2$. Hence $\mf i$ is a non-trivial coset representative of $J_\Q$ modulo $\Q^\times P$. 
		
		Next consider any $\mf j\in J_\Q$. Since $\mf j_p$ is a unit for almost all $p$, there exists $x\in \Q$ whose sign is the same as that of $\mf j_\infty$ such that $x\mf j_p \in \Z_p^\times (\Q_p^\times)^2$ for all $p\ne q'$. By the fourth assumption, $q'(\Q_q^\times)^2$ is a coset representative of the non-trivial coset of $\Z_q^\times (\Q_q^\times)^2$ modulo $\theta(O^+(NL_{q}+ \vec h))$; hence $x\mf j_q$ is either in $\theta(O^+(NL_{q}+ \vec h))$ or in $q'\theta(O^+(NL_{q}+ \vec h))$. If the latter holds, then we may replace $x$ by $q'x$ so that, without loss of generality, we can assume that $x\mf j_q\in \theta(O^+(NL_{q}+ \vec h))$ (while $x\mf j_p \in \Z_p^\times (\Q_p^\times)^2$ for all $p\ne q'$ continues to hold). It is now easy to see from $\theta(O^+(NL_{q'}+\vec h)) = \{1,q'\}(\Q_{q'}^\times)^2$ that either $x\mf j\in P$, or $x\mf j\in \mf iP$. Hence $\mf j$ is in either $\Q^\times P$, or in $\mf i \cdot \Q^\times P$, and so the group index given in equation (\ref{eqn::num_spn_in_gen}) is exactly 2. Therefore, there are two spinor genera in the genus.
	\end{proof}
	We will use Lemma \ref{lem::numspn=2_ind2subgrp_p=2_nonon_trivial_unit_squares_p=q} with $q=2$ and $q'=3$.
	
	\subsection{Classification of Lattice Cosets Found}\label{sctn::classify_found_cosets}
	In this sub-section, we shall use the convention that $L_{\vec a} \cong \diag\;\vec a$ in the standard basis $\vec e_1, \vec e_2,\vec e_3$ for $\Q^3$. Usually we shall abbreviate $L_{\vec a}$ to simply $L$ when $\vec a$ is understood from context.
	
	Let us first classify the genus of most of the lattice cosets listed in Appendix \ref{sctn::list_spn1_identities}. 
	\begin{proposition}\label{prop::classify_genus_all_listed}
		For each coset $NL_{\vec a}+\vec h_0$ given in Table \ref{table::genus_info} in Appendix \ref{sctn::gen_info}, the set $\mc G = \{NL_{\vec a} +\vec h_0, NL_{\vec a} +\vec h_1, ...\}$ given in the table is a complete set of distinct class representatives of $\gen^+(NL_{\vec a} + \vec h_0)$.
	\end{proposition}
	\begin{proof}
		Notice that, for all $p|N$, the matrix $X_{i,p} \in G_{L, q', q}$ given in Table \ref{table::genus_info} satisfies $X_{i,p}\vec h_0 \equiv \vec h_i \modc{q}$. It then follows that any rotation $\Sigma_p$ in $\Psi^{-1}X_{i,p}$ (where $\Psi$ is from Proposition \ref{prop::matrix_reduction_of_isometries}) sends $NL+\vec h_0$ to $NL+\vec h_i$. By setting $\Sigma_p$ to be the identity on $V_p$ for all other $p$, and noticing that $NL_p + \vec h_0 = L_p = NL_p + \vec h_i$ for such $p$, we get an element $\Sigma \in O^+_\A(V)$ such that $\Sigma(NL+\vec h_0) = NL+\vec h_i$. Thus every coset in $\mc G$ belongs to $\gen^+(NL+\vec h_0)$. 
		
		Next, notice that any rotation $\sigma$ sending $NL+\vec h_i$ to $NL+\vec h_j$ must preserve $L$, and thus $\sigma$ belongs to $O^+(L)$. Lemma \ref{lem::characterize_O(L)} tells us that this is a finite group, and that $\sigma$ essentially permutes $h_{i1}, h_{i2}, h_{i3}$ and changes the sign of some of them. From this, it is easy to check that no such $\sigma$ can exist, i.e. no element of $O^+(L)$ sends $NL+\vec h_i$ to $NL+\vec h_j$ for any pair $\vec h_i, \vec h_j$ given in Table \ref{table::genus_info}. Thus $\mc G$ is a set of distinct class representatives. 
		
		To check that the classes of the cosets in $\mc G$ exhausts the genus, we use Lemma \ref{lem::Mass_Formula}. Each of the values required for equation (\ref{eqn::mass_formula}) is listed in Table \ref{table::genus_info}, where we use the fact that $[O^+(L_p) : O^+(NL_p + \vec h_0)] = |G_{L, q',q}| / |H_{L, \vec h_0, q', q}|$ by Corollary \ref{cor::grp_index_using_matrix_reduction}. Using these values, it can be checked that the equation \[\sum_{NL+\vec h_i \in \mc G} \frac{1}{|O^+(NL+\vec h_i)|} = \frac{1}{|O^+(L)|} \prod_{p|N} \frac{|G_{L, q',q}|}{|H_{L, \vec h_0, q', q}|}\]
		holds. Lemma \ref{lem::Mass_Formula} then implies that $\mc G$ exhausts the genus.
	\end{proof}
	
	Notice that Table \ref{table::genus_info} in Appendix \ref{sctn::gen_info} does not classify the genus of all of the lattice cosets given in Appendix \ref{sctn::list_spn1_identities}, since in these cases the group $G_{L,q',q}$ was too large for the computer to compute (using the brute force algorithm) in any reasonable amount of time. In particular, the following cosets are missing:
	\begin{multline*}
		(\vec a, \vec h, N)\in \Big\{  \left((1,4,8)^T, (0,4,1)^T, 16 \right), \left((1,4,8)^T, (0,4,3)^T, 16 \right), \left((1,4,8)^T, (0,4,5)^T, 16 \right), \\ 
		\left((1,4,8)^T, (0,4,7)^T, 16 \right), \left((1,8,16)^T, (4,2,1)^T, 8 \right), \left((1,8,16)^T, (4,2,3)^T, 8 \right) \Big\}.
	\end{multline*}
	
	It remains to classify the spinor genera of those lattice cosets considered in Proposition \ref{prop::classify_genus_all_listed}. We dispose of the simplest case first: when there are only two classes in the genus.
	
	\begin{proposition}\label{prop::2cls_in_genus}
		For all 20 values of $\vec a$, $N$, $\vec h_0$, and $\vec h_1$ given below, both the lattice cosets $NL_{\vec a} + \vec h_0$ and $NL_{\vec a} + \vec h_1$ are in the same genus and have spinor class number 1.
	\end{proposition}
	\begin{small}
	\begin{center}
		\begin{tabular}{|c||c|c|c|c|c|c|c|}
			\hline 
			$\vec a$ & $(1,1,1)^T$ & $(1,1,2)^T$ & $(1,1,4)^T$ & $(1,1,4)^T$ & $(1,1,4)^T$ & $(1,1,8)^T$ & $(1,2,8)^T$ \\ 
			$N$ & 4 & 4 & 2 & 8 & 8 & 2 & 4 \\
			$\vec h_0$ & $(0,0,1)^T$ & $(1,1,0)^T$ & $(0,1,0)^T$ & $(0,0,1)^T$ & $(0,0,3)^T$ & $(1,1,0)^T$ & $(0,1,0)^T$ \\
			$\vec h_1$ & $(2,2,1)^T$ & $(1,1,2)^T$ & $(0,1,1)^T$ & $(4,4,3)^T$ & $(4,4,1)^T$ & $(1,1,1)^T$ & $(0,1,2)^T$ \\
			\hline
			$\vec a$& $(1,2,16)^T$ & $(1,3,3)^T$ & $(1,3,9)^T$ & $(1,3,9)^T$ & $(1,3,12)^T$ & $(1,3,12)^T$ & $(1,4,4)^T$ \\
			$N$& 2 & 6 & 4 & 6 & 3 & 6 & 2 \\%
			$\vec h_0$& $(1,0,0)^T$ & $(0,0,1)^T$ & $(1,1,0)^T$ & $(0,0,1)^T$ & $(0,1,0)^T$ & $(3,1,0)^T$ & $(1,0,0)^T$ \\%
			$\vec h_1$& $(1,0,1)^T$ & $(0,2,3)^T$ & $(0,1,1)^T$ & $(3,0,2)^T$ & $(0,0,1)^T$ & $(3,3,2)^T$ & $(1,1,1)^T$ \\%
			\hline 
			$\vec a$ & $(1,4,4)^T$ & $(1,4,8)^T$ & $(3,4,12)^T$ & $(3,4,12)^T$ & $(3,4,12)^T$ & $(3,4,36)^T$ \\
			$N$ & 4 & 4 & 3 & 6 & 12 & 4\\
			$\vec h_0$ & $(0,0,1)^T$ & $(2,1,0)^T$ & $(1,0,0)^T$ & $(0,0,1)^T$ & $(2,3,0)^T$ & $(2,1,0)^T$\\
			$\vec h_1$ & $(0,2,1)^T$ & $(2,1,2)^T$ & $(0,0,1)^T$ & $(2,0,3)^T$ & $(6,3,4)^T$ & $(2,0,1)^T$ \\
			\cline{1-7}
		\end{tabular}
	\end{center}
	\end{small}
	\begin{proof}
		From Proposition \ref{prop::classify_genus_all_listed} and Appendix \ref{sctn::gen_info}, it is easy to see that the classes of $NL+\vec h_0$ and $NL+\vec h_1$ are distinct and exhaust the genus. Suppose we show that there are exactly two spinor genera in the genus. Since each spinor genus must have at least one class, it follows that $NL+\vec h_0$ and $NL+\vec h_1$ are in different spinor genera, and thus are both spinor class number 1 lattice cosets.
		
		It now remains to check that there are exactly two spinor genera in the genus. In light of Lemmas \ref{lem::numspn_theta_triv_p_ne_q} and \ref{lem::numspn=2_ind2subgrp_p=2_nonon_trivial_unit_squares_p=q}, we need to compute $\theta(O^+(NL_p+\vec h_0))$. By Corollary \ref{cor::spnmap_img_of_modular_lattice}, we have $\theta(O^+(NL_p+\vec h_0)) = \Z_p^\times (\Q_p^\times)^2$ for all odd primes $p$ satisfying $p\nmid N$ and $p\nmid a_1a_2a_3$. For primes $p$ (possibly 2) satisfying $p|a_1a_2a_3$ but $p\nmid N$, notice that $NL_p + \vec h_0 = L_p$, and so we can use the algorithm given in Lemma \ref{lem::conway_algorithm} to compute $\theta(O^+(NL_p+\vec h)) = \theta(O^+(L_p))$. This has been done in Table \ref{table::spn_info_upperbnd} in Appendix \ref{sctn::upperbnd_img_spn_norm}. It remains to check for those primes $p$ dividing $N$.
		 
		Consider the lattice $M = \Z \vec h_0 + NL$. Notice that $M_p = \Z_p \vec h_0 + NL_p$, and any rotation in $O^+(NL_p+\vec h_0)$ preserves $NL_p$ and $\vec h_0$ modulo $NL_p$, and so also preserves $M_p$. This implies that $O^+(NL_p+\vec h_0) \subseteq O^+(M_p) \cap O^+(L_p)$, and so $\theta(O^+(NL_p + \vec h_0)) \subseteq \theta(O^+(M_p)) \cap \theta(O^+(L_p))$. The latter two groups are easily calculated using the algorithm in Lemma \ref{lem::conway_algorithm}, and the results are given in Table \ref{table::spn_info_upperbnd} in Appendix \ref{sctn::upperbnd_img_spn_norm}.
		
		From Table \ref{table::spn_info_upperbnd}, it is seen that for the lattices $L$ under consideration in this proposition, whenever both $N$ and $dL$ are a power of 2, the subgroup $\Z_2^\times(\Q_2^\times)^2 \cap \theta(O^+(M_2)) \cap \theta(O^+(L_p))$ has index exactly 2 in $\Z_2^\times (\Q_2^\times)^2$. It then follows that $\Z_2^\times(\Q_2^\times) \cap \theta(O^+(NL_2 + \vec h_0))$ has index at least two in $\Z_2^\times (\Q_2^\times)^2$ (which itself has order 4). By Lemma \ref{lem::numspn_theta_triv_p_ne_q} with $q=2$, it follows that there are either two or four spinor genera in the genus. Since there are only two classes, it then follows that there are exactly two spinor genera in the genus. 
		Similarly, for $N=3$ and \[(\vec a, \vec h_0)\in \left\{ \left( (1,3,12)^T, (0,1,0)^T \right), \left( (3,4,12)^T, (1,0,0)^T \right) \right\},\]
		notice that $3L_2 + \vec h_0 = L_2$, and so $\theta(O^+(3L_2 + \vec h_0)) = \theta(O^+(L_2)) = \Z_2^\times (\Q_2^\times)^2$. Moreover, in each case we also have $\Z_3^\times (\Q_3^\times)^2 \cap \theta(O^+(3L_3+\vec h_0)) \subseteq \Z_3^\times(\Q_3^\times)^2\cap \theta(O^+(M_3))\cap \theta(O^+(L_3)) = (\Q_3^\times)^2$, and so $\Z_3^\times (\Q_3^\times)^2 \cap \theta(O^+(3L_3+\vec h_0)) = (\Q_3^\times)^2$. Lemma \ref{lem::numspn_theta_triv_p_ne_q} with $q=3$ implies that there are exactly two spinor genera in the genus. 
		
		This takes care of all but 7 of the genera under consideration. For each of these lattices, notice that $\theta(O^+(M_2))\cap \theta(O^+(L_2)) = \Z_2^\times (\Q_2^\times)^2$ and $\theta(O^+(M_3))\cap \theta(O^+(L_3)) = \{1,3\}(\Q_3^\times)^2$. Thus $\theta(O^+(NL_2+\vec h_0)) \subseteq \Z_2^\times (\Q_2^\times)^2$ and $\theta(O^+(NL_3+\vec h_0)) \subseteq \{1,3\}(\Q_3^\times)^2$. If we can show that $5(\Q_2^\times)^2 \in \theta(O^+(NL_2+\vec h_0))$ and $3(\Q_3^\times)^2 \in \theta(O^+(NL_3+\vec h_0))$, then all conditions of Lemma \ref{lem::numspn=2_ind2subgrp_p=2_nonon_trivial_unit_squares_p=q} with $q=2$ and $q'=3$ are satisfied, and so there would be two spinor genera in the genus. 
		
		\begin{table}[H]
			\centering
			\begin{tabular}[c]{|c|c|c|c|c|} 
				\hline 
				\multirow{2}{*}{$\vec a$} & \multirow{2}{*}{$N$} & \multirow{2}{*}{$\vec h_0$} & Rotation in $O^+(NL_2+\vec h_0)$ & Rotation in $O^+(NL_3+\vec h_0)$\\
				& & & with spinor norm $5(\Q_2^\times)^2$ & with spinor norm $3(\Q_3^\times)^2$\\\hline
				
				$(1,3,3)^T$ & $6$ & $(0,0,1)^T$ & $\tau_{\vec e_1} \tau_{\vec e_1+2\vec e_2} $ & $\tau_{\vec e_1}\tau_{\vec e_2}$\\\hline 
				
				$(1,3,9)^T$ & $4$ & $(1,1,0)^T$ & $\tau_{\vec e_3} \tau_{3\vec e_1 - \vec e_2-\vec e_3}$ & $\tau_{\vec e_1}\tau_{\vec e_2}$\\\hline 
				
				$(1,3,9)^T$ & $6$ & $(0,0,1)^T$ & $\tau_{\vec e_1} \tau_{\vec e_1+2\vec e_3} $ & $\tau_{\vec e_1}\tau_{\vec e_2}$ \\\hline 
				
				$(1,3,12)^T$ & $6$ & $(3,1,0)^T$ & $\tau_{\vec e_1} \tau_{\vec e_1+ 2\vec e_2 + 2\vec e_3} $ & $\tau_{\vec e_1}\tau_{\vec e_3}$ \\\hline 
				
				$(3,4,12)^T$ & $6$ & $(0,0,1)^T$ & $\tau_{\vec e_1} \tau_{\vec e_1+\vec e_2} $ & $\tau_{\vec e_1}\tau_{\vec e_2}$\\\hline 
				
				$(3,4,12)^T$ & $12$ & $(2,3,0)^T$ & $\tau_{\vec e_3} \tau_{\vec e_1+\vec e_3} $ & $\tau_{\vec e_2}\tau_{\vec e_3}$\\\hline 
				
				$(3,4,36)^T$ & $4$ & $(2,1,0)^T$ & $\tau_{\vec e_1} \tau_{\vec e_1+\vec e_2} $ & $\tau_{\vec e_1}\tau_{\vec e_2}$\\\hline 
			\end{tabular}
			
			\caption{Rotations in $O^+(NL_p + \vec h_0)$ with desired spinor norm.}
			\label{table::2cls_in_genus__spn_info_lowerbnd}	
		\end{table}
		
		Table \ref{table::2cls_in_genus__spn_info_lowerbnd} shows that $5(\Q_2^\times)^2 \in \theta(O^+(NL_2+\vec h_0))$ and $3(\Q_3^\times)^2 \in \theta(O^+(NL_3+\vec h_0))$ for all 7 of the remaining genera under consideration, and so we are done.
	\end{proof}
	
	Thus, we have shown that 40 of the 85 identities given in Appendix \ref{sctn::list_spn1_identities} indeed come from spinor class number 1 lattice cosets, as expected. Let us now consider the remaining identities. While finding the number of spinor genera follows in a similar vein as the proof of Proposition \ref{prop::2cls_in_genus}, a little more work needs to be done in order to classify the spinor genera.
	
	\begin{proposition}\label{prop::3+cls_in_genus_spnnum=1}
		For all 16 values of $\vec a$, $N$, and $\vec h_0$ given below, the lattice coset $NL_{\vec a}+\vec h_0$ has spinor class number 1. 
	\end{proposition}
	\begin{footnotesize}
		\begin{center}
			\begin{tabular}{|c||c|c|c|c|c|c|c|c|}
				\hline 
				$\vec a$ & $(1,1,1)^T$ & $ (1,1,1)^T $ & $(1,1,2)^T$ & $(1,1,2)^T$ & $(1,1,4)^T$ & $(1,1,8)^T$ & $(1,2,2)^T$ & $(1,2,2)^T$\\
				$N$ & 8 & 8 & 8 & 8 & 4 & 4 & 4 & 8 \\
				$\vec h_0$ & $(1,2,2)^T$ & $(3,2,2)^T$ & $(0,4,1)^T$ & $(0,4,3)^T$ & $(1,2,1)^T$ & $(1,1,1)^T$ & $(1,0,2)^T$ & $(4,1,3)^T$\\\hline \hline 
				$\vec a$ & $(1,3,12)^T$ & $(1,4,4)^T$ & $(1,4,4)^T$ & $(1,4,4)^T$ & $(1,6,6)^T$ & $(1,8,8)^T$ & $(1,8,8)^T$ & $(1,12,12)^T$ \\
				$N$ & 12 & 4 & 8 & 8 & 12 & 2 & 8 & 12\\
				$\vec h_0$ & $(0,4,3)^T$ & $(1,1,1)^T$ & $(4,1,2)^T$ & $(4,3,2)^T$ & $(0,1,5)^T$ & $(1,0,1)^T$ & $(0,1,3)^T$ & $(6,3,4)^T$\\\hline
			\end{tabular}
		\end{center}
	\end{footnotesize}
	\begin{proof}
		Proposition \ref{prop::classify_genus_all_listed} and Table \ref{table::genus_info} in Appendix \ref{sctn::gen_info} gives a complete genus classification of the genus of $NL+\vec h_0$. Let $\mc G = \{\vec h_0, ..., \vec h_k\}$ be the complete set of distinct class representatives of $\gen^+(NL_{\vec a} + \vec h_0)$ established by Proposition \ref{prop::classify_genus_all_listed}. In Table \ref{table::spnorbit} given in Appendix \ref{sctn::grouping_spn_gen}, for all 16 values of $\vec a, N, \vec h_0$ given above, an element $\Sigma \in O_\A'(V)$ is given which takes some $NL+\vec h_i$ to some $NL+\vec h_j$, where $1\le i, j\le k$. It can easily be checked that sufficiently many $\Sigma$ have been listed to establish that $NL+\vec h_i$ are all in the same $O_\A'(V)$-orbit. This implies that $NL+\vec h_i$ for $i=1,2,...,k$ all belong to the same spinor genus. In particular, it follows that there is exactly one spinor genus in the genus if $NL+\vec h_0$ belongs to the same spinor genus as $NL+\vec h_1$, and otherwise there are exactly two spinor genera with (the class of) $NL+\vec h_0$ lying alone in its spinor genus. Thus, if we can show there are (at least) two spinor genera in the genus, then $NL+\vec h_0$ lies alone in its spinor genus, and we are done.
		
		For all $\vec a, \vec h_0, N$ given in the above table for which $dL = a_1a_2a_3$ and $N$ are both powers of 2, it follows from Corollary \ref{cor::spnmap_img_of_modular_lattice} that $\theta(O^+(NL_p+\vec h_0)) = \theta(O^+(L_p)) = \Z_p^\times (\Q_p^\times)^2$ for all odd primes $p$. Thus, for these values of $\vec a, \vec h_0$, and $N$, there are exactly \[[\Z_2^\times (\Q_2^\times)^2 : \theta(O^+(NL_2+\vec h_0)) \cap \Z_2^\times(\Q_2^\times)^2]\]
		number of spinor genera in the genus, by Lemma \ref{lem::numspn_theta_triv_p_ne_q}. However, $\theta(O^+(NL_2+\vec h_0)\cap \Z_2^\times(\Q_2^\times)^2$ is contained in $\theta(O^+(L_2)) \cap \theta(O^+(M_2))$ (where $M:=NL+\Z \vec h_0$), the latter of which is a subgroup of $\Z_2^\times (\Q_2^\times)^2$ of index 2 by Table \ref{table::spn_info_upperbnd} in Appendix \ref{sctn::upperbnd_img_spn_norm}. Thus $\theta(O^+(NL_2+\vec h_0)) \cap \Z_2^\times(\Q_2^\times)^2$ has group index at least 2 in $\Z_2^\times (\Q_2^\times)^2$. Hence there are at least two spinor genera in the genus, from which the result follows.
		
		As usual $\theta(O^+(12L_p+\vec h_0)) = \Z_p^\times (\Q_p^\times)^2$ for all $p\ne 2,3$. It remains to check the remaining three values of $\vec a, N, \vec h_0$. We have $N=12$, and  \[(\vec a, \vec h_0)\in \left\{ \left( (1,3,12)^T, (0,4,3)^T \right), \left( (1,6,6)^T, (0,1,5)^T \right), \left((1,12,12)^T, (6,3,4)^T\right) \right\}.\]
		In each of the three cases, we see from Table \ref{table::spn_info_upperbnd} in Appendix \ref{sctn::upperbnd_img_spn_norm} that $\theta(O^+(12L_2+ \vec h_0)) \subseteq \Z_2^\times (\Q_2^\times)^2$, and $\theta(O^+(12L_3+ \vec h)) \subseteq \{1,3\}(\Q_3^\times)^2$. In particular, the latter implies that $\Z_3^\times(\Q_3^\times)^2 \cap \theta(O^+(12L_3+ \vec h))$ is trivial, and so has index 2 in $\Z_3^\times(\Q_3^\times)^2$. Thus, if we can show that $\Z_2^\times (\Q_2^\times)^2\subseteq \theta(O^+(12L_2+ \vec h_0))$, then we would have  $\theta(O^+(12L_2+ \vec h_0)) = \Z_2^\times (\Q_2^\times)^2$, and so Lemma \ref{lem::numspn_theta_triv_p_ne_q} would imply that there are exactly two spinor genera in the genus. 
		
		\begin{table}[H]
			\centering
			\caption{Rotations in $O^+(12L_2+\vec h_0)$ with desired spinor norm.} 
			\label{table::3+cls_in_genus__spn_info_lowerbnd}
			\begin{tabular}{|c|c|c|c|}
				\hline 
				\multirow{2}{*}{$\vec a$} & \multirow{2}{*}{$\vec h_0$} & Rotation in $O^+(12L_2+\vec h_0)$ & Rotation in $O^+(12L_3+\vec h_0)$\\
				& & with spinor norm $3(\Q_2^\times)^2$ & with spinor norm $5(\Q_2^\times)^2$\\\hline
				$(1,3,12)^T$ & $(0,4,3)^T$ & $\tau_{\vec e_1}\tau_{\vec e_2}$ & $\tau_{\vec e_2}\tau_{2\vec e_1+\vec e_2}$ \\
				$(1,6,6)^T$ & $(0,1,5)^T$ & $\tau_{\vec e_1 + \vec e_2} \tau_{\vec e_1+\vec e_2+\vec e_3}$ & $\tau_{\vec e_1} \tau_{\vec e_1+ \vec e_2+ \vec e_3}$\\
				$(1,12,12)^T$ & $(6,3,4)^T$ & $\tau_{\vec e_1}\tau_{\vec e_3}$ & $\tau_{\vec e_1} \tau_{2\vec e_2+ \vec e_3} $\\\hline 
			\end{tabular}
		\end{table}
		Notice that Table \ref{table::3+cls_in_genus__spn_info_lowerbnd} implies that $3(\Q_2^\times)^2, 5(\Q_2^\times)^2 \in \theta(O^+(12L_2+ \vec h))$. Since $3(\Q_2^\times)^2$ and $5(\Q_2^\times)^2$ generate $\Z_2^\times (\Q_2^\times)^2$, it follows that $\Z_2^\times (\Q_2^\times)^2\subseteq \theta(O^+(12L_2+ \vec h_0))$, as required. 
	\end{proof}

	We have thus explained 56 of the 85 pseudo-spinor class 1 identities as coming from lattice cosets with spinor class number 1. However, this need not be the case in general, as is shown by the next two propositions.

	\begin{proposition}\label{prop::spnnum=2}
		For $\vec a = (1,1,2)^T$, $N=8$, $\vec h_0= (1,3,2)^T$, and for $\vec a = (1,3,3)^T$, $N=12$, $\vec h_0 = (3,3,4)^T$, the lattice coset $NL_{\vec a} + \vec h_0$ has spinor class number 2.
	\end{proposition}
	\begin{proof}
		As before, Proposition \ref{prop::classify_genus_all_listed} gives the genus classification of $NL+\vec h_0$, and from Table \ref{table::genus_info} in Appendix \ref{sctn::gen_info}, we know that $\mc G = \{NL+\vec h_0, ..., NL+\vec h_k\}$ (where $k=3$ in the case of $\vec a = (1,3,3)^T$, and $k=5$ for $\vec a = (1,1,2)^T$) is a complete set of distinct class representatives of $\gen^+(NL+\vec h_0)$. Moreover, notice that $\vec h_1 \equiv -\vec h_0 \modc{N}$. As in the previous proof, Table \ref{table::spnorbit} in Appendix \ref{sctn::grouping_spn_gen} gives explicit values of $\sigma \Sigma \in O^+(V)O_\A'(V)$ mapping $NL+\vec h_0$ to $NL+\vec h_1$, and mapping $NL+\vec h_i$ to $NL+\vec h_j$ where $2\le i, j\le k$. It follows that the (proper) classes of $NL+\vec h_0$ and $NL+\vec h_1$ belong to the same spinor genus, and the classes of $NL+\vec h_2$, .., $NL+\vec h_k$ belong to the same spinor genus. Thus, as before, we only need to show that there are at least two spinor genera in the genus. In the case of $\vec a = (1,1,2)^T$ and $N=8$, notice that $\theta(O^+(8L_p + \vec h_0)) = \Z_p^\times(\Q_p^\times)^2$ for all $p\ne 2$ by Corollary \ref{cor::spnmap_img_of_modular_lattice}, while $\theta(O^+(8L_2+\vec h_0)) \subseteq \theta(O^+(8L_2+\Z_2 \vec h_0)) = \{1,5\}(\Q_2^\times)^2$ (by Table \ref{table::spn_info_upperbnd} in Appendix \ref{sctn::upperbnd_img_spn_norm}). Lemma \ref{lem::numspn_theta_triv_p_ne_q} then implies that there are at least 2 spinor genera, and so we are done in this case. On the other hand, for $\vec a = (1,3,3)^T$, Table \ref{table::spn_info_upperbnd} again tells us that $$\theta(O^+(12L_2+\vec h_0))\subseteq \Z_2^\times (\Q_2^\times)^2 \quad\text{ and }\quad \theta(O^+(12L_3+\vec h_0)) \subseteq \{1,3\}(\Q_3^\times)^2$$
		However, notice that $\tau_{\vec e_2} \tau_{3\vec e_1+\vec e_2+ \vec e_3} \in O(12L_2 + \vec h_2)$ has 2-adic spinor norm $5(\Q_2^\times)^2$, while $\tau_{\vec e_1 + \vec e_3}\tau_{\vec e_3} \in O^+(12L_3 + \vec h_2)$ has 3-adic spinor norm $3(\Q_3^\times)^2$. It follows that $\theta(O^+(12L_3 +  \vec h_2)) = \{1,3\}(\Q_3^\times)^2$, while $\theta(O^+(12L_2 + \vec h_2))$ is a non-trivial subgroup of $\Z_2^\times (\Q_2^\times)^2$ containing $5 (\Q_2^\times)^2$. Since $\theta(O^+(12L_p+\vec h_0)) = \Z_p^\times (\Q_p^\times )^2$ for all $p\ne 2,3$ anyway, all conditions of Lemma \ref{lem::numspn=2_ind2subgrp_p=2_nonon_trivial_unit_squares_p=q} hold, and so there are exactly two spinor genera in this genus as well. 
	\end{proof}

	Notice that the two lattice cosets $NL+\vec h_0$ given above have proper spinor class number 2 and yet satisfy a pseudo-spinor class 1 identity (listed in Appendix \ref{sctn::list_spn1_identities}). This is because the other (proper) class in its spinor genus is $NL+\vec h_1 = NL - \vec h_0$ (since $\vec h_1 \equiv -\vec h_0\modc{N}$), and so the theta series of $NL+\vec h_0$ and $NL+\vec h_1$ are the same. Since the average over the (proper) spinor genus is expected to satisfy a pseudo-spinor class 1 identity, it follows that $\Theta_{\vec a, \vec h_0, N}$ would also satisfy such an identity. Thus, it is perhaps not surprising that these lattice cosets satisfy pseudo-spinor class 1 identities while having spinor class number at least 2. The next proposition gives three, more interesting, examples.

	\begin{proposition}\label{prop::spnnum=3}
		The lattice coset $NL_{\vec a} + \vec h_0$ has spinor class number 3 for 
		\[(\vec a, \vec h, N) \in \left\{\left((1,2,4)^T, (0,1,2)^T, 8\right), \left((1,2,4)^T, (0,3,2)^T, 8\right), \left((1,2,8)^T, (1,0,1)^T, 4\right) \right\}.\]
	\end{proposition}
	\begin{proof}
		By Proposition \ref{prop::classify_genus_all_listed}, the genus of $NL+\vec h_0$ contains six classes of lattice cosets. A complete set of distinct class representatives $\mc G = \{NL+\vec h_0, ..., NL+\vec h_5\}$ is given in Table \ref{table::genus_info} in Appendix \ref{sctn::gen_info}. Moreover, Table \ref{table::spnorbit} in Appendix \ref{sctn::grouping_spn_gen} gives elements of $O_\A'(V)$ that map $\{NL+\vec h_0, NL+\vec h_1, NL+\vec h_2\}$ onto itself, and $\{NL+\vec h_3, NL+\vec h_4, NL+\vec h_5\}$ onto itself. This implies $NL+\vec h_1, NL+\vec h_2 \in \spn^+(NL+\vec h_0)$, and $NL+\vec h_4, NL+\vec h_5\in \spn^+(NL+\vec h_3)$. Hence, if there are (at least) two spinor genera in the genus, it would follow that \[\spn^+(NL+\vec h_0) = \cls^+(NL+\vec h_0) \sqcup \cls^+(NL+\vec h_1) \sqcup \cls^+(NL+\vec h_2)\]
		and so $NL+\vec h_0$ would have spinor class number 3. However, notice that $dL=a_1a_2a_3$ and $N$ are both powers of two, and so Corollary \ref{cor::spnmap_img_of_modular_lattice} implies that $\theta(O^+(NL_p+\vec h_0)) = \Z_p^\times (\Q_p^\times)^2$ for odd $p$. Moreover, from Table \ref{table::spn_info_upperbnd} in Appendix \ref{sctn::upperbnd_img_spn_norm} we have \[\theta(O^+(NL_2 + \vec h_0)) \subseteq \theta(O^+(NL_2+ \Z_2\vec h_0)) = \{1,2,5,10\}(\Q_2^\times)^2\]
		and so $\Z_2^\times (\Q_2^\times)^2 \cap \theta(O^+(NL_2 + \vec h_0))$ has group index at least two in $\Z_2^\times (\Q_2^\times)^2$. Lemma \ref{lem::numspn_theta_triv_p_ne_q} then implies that there are at least two spinor genera in the genus, and so we are done. 
	\end{proof}

	Even though the three lattice cosets given above have spinor class number 3, they satisfy pseudo-spinor class 1 identities, given in Table \ref{table::list_spn1_identities}. However, unlike the two lattice cosets in Proposition \ref{prop::spnnum=2}, the pseudo-spinor class 1 identities cannot be explained as easily. By checking sufficiently many coefficients (upto the valence bound; see Lemma \ref{lem::valence_formula}), the following identity can be established
	\begin{equation}\label{eqn::spnnum=3_avg_of_other_two}
		\Theta_{\vec a, \vec h_0, N} = \frac12 \left(\Theta_{\vec a, \vec h_1, N} + \Theta_{\vec a, \vec h_2, N} \right) \qquad \left( \vec a, \vec h_0, N \text{ as in Proposition \ref{prop::spnnum=3}} \right).
	\end{equation}
	This identity then implies that the average over the spinor genus of $NL+\vec h_0$ is also equal to $\Theta_{\vec a, \vec h_0, N}$, and so by the conjecture given in equation (\ref{eqn::conjecture_avg_over_spn}), $\Theta_{\vec a, \vec h_0, N}$ should satisfy a pseudo-spinor class 1 identity. While equation (\ref{eqn::spnnum=3_avg_of_other_two}) does explain why these cosets satisfy pseudo-spinor class 1 identities, there is no obvious explanation as to why equation (\ref{eqn::spnnum=3_avg_of_other_two}) must hold in the first place. It is possible that there is some (easily generalizable) elementary correspondence between solutions $\vec x\in \Z^3$ of $Q_{\vec a}(\vec x) = n$ satisfying $x\equiv \vec h_0\modc{N}$, and the solutions $\vec x\in \Z^3$ of $Q_{\vec a}(\vec x) = n$ satisfying $x\equiv \vec h_1\modc{N}$ or $x\equiv \vec h_2 \modc{N}$. It is also possible that this phenomenon hints to some as yet unknown deeper structure in the genus of a lattice coset---perhaps, there is some (as yet unknown) classification of lattice cosets intermediate to the spinor genus and the class.
	
	Notice that Propositions \ref{prop::2cls_in_genus} to \ref{prop::spnnum=3} classify 61 of the 85 lattice cosets listed in Table \ref{table::list_spn1_identities} in Appendix \ref{sctn::list_spn1_identities}. Proposition \ref{prop::classify_cond2_(1,3,9)} in the next section classifies the spinor genus for two other lattice cosets, namely $\vec a = (1,3,9)^T$, $N=2$, and $\vec h = (0,0,1)^T, (1,0,0)^T$. Thus the genus and the spinor genus of 63 of the found 85 lattice cosets have been classified. For 16 of the remaining 22 unclassified lattice cosets, a determination of the representatives of the genus was still possible, and was carried out in Proposition \ref{prop::classify_genus_all_listed} and Appendix \ref{sctn::gen_info}. On the other hand, our method to compute the representatives in the spinor genera requires computing $\theta(O^+(N(L_{\vec a})_2 + \vec h))$, and the group $\theta(O^+(N(L_{\vec a})_2 + \Z_2 \vec h)) \cap \theta(O^+((L_{\vec a})_2))$ for these 16 lattices was too large to effectively determine $\theta(O^+(N(L_{\vec a})_2 + \vec h))$. Currently, there does not seem to be any general method known to compute $\theta(O^+(N(L_{\vec a})_2 + \vec h))$ directly.
	
	\subsection{Classification of All 2-adically Unimodular Conductor 2 Lattice Cosets} \label{sctn::classify_cond2}
	
	As another example illustrating the use of Proposition \ref{prop::matrix_reduction_of_isometries}, we classify all lattice cosets of the form $2L_{\vec a}+\vec h$, where $\vec h$ is one of the 7 distinct non-zero vectors modulo $2$, and $L = L_{\vec a}$ is a 2-adically unimodular lattice (i.e. $2\nmid a_1a_2a_3$) with class number 1. In this case, in the notation of Proposition \ref{prop::matrix_reduction_of_isometries}, we will always have $q'=8$, $q=2$ and $\delta = 2$. Also, Lemma \ref{lem::Mass_Formula} along with Corollary \ref{cor::grp_index_using_matrix_reduction} reduces to simply 
	\begin{equation}\label{eqn::mass_formula_cond2}
		\sum_{2L+\vec h' \in \mc G} \frac{1}{|O^+(2L+\vec h')|} = \frac{[O^+(L_2) : O^+(2L_2+\vec h) ]}{|O^+(L)|} = \frac{|G_{L,q',2}|}{|H_{L,\vec h, q',2}| \cdot |O^+(L)|}
	\end{equation}
	where $\mc G$ is a set of class representatives in the genus of $2L+\vec h$. 
	
	The next lemma guarantees a symmetry between different, not necessarily $p$-adically equivalent, lattice cosets. The lemma is stated and proved in more generality with the 2-adic unimodularity condition removed.
	
	\begin{lemma}\label{lem::cond2_symmetry_of_100_011}
		Let $\vec h = \vec e_1+\vec e_2+\vec e_3$, and consider $1\le i< j\le 3$. If $L$ and $2L+\vec h$ have class number 1, then $2L+\vec e_j$ is in the same class (resp. genus, spinor genus) of $2L+\vec e_i$ if and only if $2L+(\vec h - \vec e_j)$ is in the same class (resp. genus, spinor genus) of $2L+(\vec h - \vec e_i)$.
	\end{lemma}
	\begin{proof}
		First let us show that, for any $\Sigma \in O_\A(V)$, we have $\Sigma(2L+\vec e_i) = 2L+\vec e_j$ if and only if $\Sigma (2L+(\vec h - \vec e_i)) = 2L+ \vec (\vec h - \vec e_j)$. In either case however, $\Sigma$ must satisfy $\Sigma L = L$. Now, if $\Sigma L = L$, then notice that $\Sigma (2L+\vec h) = 2L+\vec v$ for some $\vec v\in L$. By definition, we have $2L+\vec v\in \gen^+(2L+\vec h)$. 
		Since $2L+\vec h$ is assumed to be of class number 1, it follows that $2L+\vec v \in \cls^+(2L+\vec h)$. It can be checked, using Lemma \ref{lem::characterize_O(L)}, that $2L+\vec h$ is fixed by $O^+(L)$, which then implies that $\Sigma(2L + \vec h) = 2L + \vec h$. 
		Thus, for all $p$ there exists $\vec u_p \in L_p$ such that $\Sigma_p \vec h = \vec h + 2\vec u_p$.
		
		Now, suppose that $\Sigma (2L+\vec e_i) = 2L+\vec e_j$; then $\Sigma_p \vec e_i = \vec e_j + 2\vec v_p$ for some $\vec v_p\in L_p$, for all $p$. This implies that $\Sigma_p(\vec h - \vec e_i) = \vec h - \vec e_j + 2(\vec u_p - \vec v_p) \equiv \vec h - \vec e_j \modc{2L_p}$ by linearity, and so from $\Sigma L = L$ it follows that $\Sigma (2L+(\vec h - \vec e_i)) = 2L+ \vec (\vec h - \vec e_j)$. The backward direction follows in exactly the same way.
		
		Finally, notice that $2L+\vec e_j$ is in the same class (resp. genus, spinor genus) of $2L+\vec e_i$ if and only if there exists $\Sigma$ in $O^+(V)$ (resp. $O_\A(V)$, $O^+(V)O'_\A(V)$) such that $\Sigma(2L+\vec e_i) = 2L+\vec e_j$. By the above result, this is equivalent to $\Sigma (2L+(\vec h - \vec e_i)) = 2L+ \vec (\vec h - \vec e_j)$ for some $\Sigma$ in $O^+(V)$ (resp. $O_\A(V)$, $O^+(V)O'_\A(V)$), which is itself equivalent to $2L+(\vec h - \vec e_j)$ being in the same class (resp. genus, spinor genus) of $2L+(\vec h - \vec e_i)$. The result follows.
	\end{proof}
	
	To make use of Proposition \ref{prop::matrix_reduction_of_isometries}, let us completely characterize $G_{L, 8, 2}$ for 2-adically unimodular lattices. 
	
	\begin{lemma}\label{lem::characterize_G_L82}
		For a permutation $s$ in the symmetric group $S_3$, let $E_s$ denote the matrix whose $i$'th column is the $s(i)$'th standard basis vector of $(\Z/2\Z)^3$, i.e. the only non-zero entries in $E_s$ are 1s in the $(i,s(i))$'th entries for $i=1,2,3$. Then, for $L_{\vec a}$ 2-adically unimodular, \[G_{L,8,2} = \{E_s : s\in S_3 \text{ such that } a_{s(i)} \equiv a_i \modc{4} \text{ for } i=1,2,3\}.\]
	\end{lemma}
	\begin{proof}
		Consider any $X\in G_{L,8,2}$; then there exists $Y\in SL_3(\Z/8\Z)$ such that $Y\equiv X\modc{2}$ and $Y^TAY\equiv A\modc{8}$. Let $Y=(y_{ij})$. Since $\sum_{i=1}^3 a_iy_{ij}^2 \equiv a_j \modc{8}$ for $j=1,2,3$, where each of the $a_i$ are odd, it follows that an odd number of the $y_{ij}$ are odd for each $j$. On the other hand, from $8|\sum_{i=1}^3 a_i y_{ij}y_{ik}$ for $1\le j<k\le 3$, it follows that at least one of $y_{1j},y_{2j},y_{3j}$ must be even for each $j$. Hence, exactly one of $y_{1j},y_{2j},y_{3j}$ must be odd (and the others even) for each $j$, and so $X = E_s$ for some permutation $s\in S_3$. From $\sum_{i=1}^3 a_iy_{ij}^2 \equiv a_j \modc{8}$ (where $y_{ij}^2\equiv 1\modc{8}$ if $y_{ij}$ is odd) it then follows that $a_{s(i)} \equiv a_i\modc{4}$ as well. Thus $G_{L,8,2}$ is contained in the set on the right hand side.
		
		For a rotation $\sigma$, let $[\sigma]$ denote its coordinate matrix in $M_3(\Q_2)$. If $ a_i \equiv a_j \equiv \pm 1\modc{4}$, a simple computation yields $\tau_{\vec e_1} \tau_{\vec e_i - \vec e_j} \in O^+(L)$ and
		\[[\tau_{\vec e_1}\tau_{\vec e_i- \vec e_j}] \equiv E_{(i,j)} \modc{2} \]
		for all $1\le i<j\le 3$, where $(i,j)$ is the transposition in $S_3$ that swaps $i$ and $j$. Since $S_3$ is generated by these transpositions, it follows that each of the $E_s$ lies in the image of the homomorphism $\Psi$ whenever $s$ fixes $\vec a\modc{4}$. The proposition follows.
	\end{proof}
	Note that this lemma implies $G_{L_{\vec a},8,2} = G_{L_{\vec a'},8,2}$ if and only if $\vec a' \equiv \vec a \modc{4}$.
	
	We now classify all 2-adically unimodular lattice cosets $2L+\vec h$ in the following six propositions. 
	\begin{proposition}\label{prop::classify_cond2_h=111}
		For $L$ 2-adically unimodular with class number 1, and for $\vec h = \vec e_1+\vec e_2+\vec e_3$, we have \[\gen^+(2L+\vec h) = \spn^+(2L+\vec h) = \cls^+(2L+\vec h).\]
	\end{proposition}
	\begin{proof}
		It is easy to see that any element of $E_s$ of $G_{L,8,2}$ satisfies $E_s\vec h =\vec h$, and so $H_{L, \vec h,8,2} = G_{L,8,2}$. As $O^+(2L+\vec h) = O^+(L)$ as well, by equation (\ref{eqn::mass_formula_cond2}) 
		\[\frac{1}{|O^+(2L+\vec h)|} \;+\! \sum_{2L+\vec h' \in \mc G \backslash \{2L+\vec h\}} \frac{1}{|O^+(2L+\vec h')|} =  \frac{|G_{L,8,2}|}{|H_{L,\vec h, 8,2}| \cdot |O^+(L)|} = \frac{1}{|O^+(2L+\vec h)|}.\]
		It follows that the class of $2L+\vec h$ is the only class in its genus.
	\end{proof}
	
	Thus $2L+(\vec e_1+\vec e_2+\vec e_3)$ has class number 1, and so it follows from Lemma \ref{lem::cond2_symmetry_of_100_011} that the structure of the genus (or genera) of $2L+(\vec e_1+\vec e_2), 2L+(\vec e_1+\vec e_3), 2L+(\vec e_2+\vec e_3)$ mirrors that of the genus (or genera) of $2L+\vec e_3$, $2L+\vec e_2$, and $2L+\vec e_1$ respectively. That $2L+(\vec e_i + \vec e_j)$ and $2L+\vec e_k$ (for $1\le i<j\le 3$, and $1\le k\le 3$) are never in the same genus can be seen by noticing that $2L+(\vec e_i + \vec e_j)$ and $2L+\vec e_k$ can never be 2-adically equivalent, which in turn follows from Lemma \ref{lem::characterize_G_L82}, since $E_s(\vec e_i + \vec e_j) \not \equiv \vec e_k \modc{2}$ for all $s\in S_3$. Hence, we need only classify $2L+\vec e_1$, $2L+\vec e_2$, and $2L+\vec e_3$.
	
	\begin{proposition}\label{prop::classify_cond2_(1,1,1)}
		For $\vec a = (1,1,1)^T$, we have \[\gen^+(2L+\vec e_1) = \spn^+(2L+\vec e_1) = \cls^+(2L+\vec e_1) \] 
		where $\cls^+(2L+\vec e_1) = \cls^+(2L+\vec e_2) = \cls^+(2L+\vec e_3) $.
	\end{proposition}
	\begin{proof}
		Note that $|O^+(L)|=24$ by Lemma \ref{lem::characterize_O(L)}, while $|G_{L,8,2}| = 6$ by the above lemma. Moreover, a simple computation shows that \[(\tau_{\vec e_1} \tau_{\vec e_1-\vec e_2}) (2L+\vec e_1) = 2L+\vec e_2 \quad\text{ and }\quad  (\tau_{\vec e_1} \tau_{\vec e_1-\vec e_3}) (2L+\vec e_1) = 2L+\vec e_3 \]
		where $\tau_{\vec e_1} \tau_{\vec e_1-\vec e_2}, \tau_{\vec e_1} \tau_{\vec e_1-\vec e_3} \in O^+(L)$, and so $2L+\vec e_2,2L+\vec e_3\in \cls^+(2L+\vec e_1)$. That this class exhausts the genus can be seen by equation (\ref{eqn::mass_formula_cond2}), noting that $|O^+(2L+\vec e_1)| = 8$ and $H_{L,\vec e_1,8,2} = \{E_{\mathrm{id}}, E_{(2,3)}\}$. The proposition follows.
	\end{proof}
	
	\begin{proposition}\label{prop::classify_cond2_2uni_(a,a,b)_a=b(4)}
		Suppose $\vec a= (a,a,b)^T$ where $a\ne b$ and $a\equiv b\equiv \pm 1\modc{4}$. Then, we have the following classification \[\gen^+(2L+\vec e_1) = \spn^+(2L+ \vec e_1) = \cls^+(2L+\vec e_1) \sqcup \cls^+(2L+\vec e_3)\]
		where $\cls^+(2L+\vec e_2) = \cls^+(2L+\vec e_1)$.
	\end{proposition}
	\begin{remark}
		In particular, this holds for $\vec a$ in $$ \{(1,1,5)^T, (1,1,9)^T, (1,1,21)^T, (5,5,1)^T, (9,9,1)^T, (21,21,1)^T, (3,3,7)^T, (7,7,3)^T\}.$$
	\end{remark}
	\begin{proof}
		Note that $|G_{L,8,2}|=6$, and $|O^+(L)|=8$ (since $ \vec e_3$ must now be mapped to $\pm \vec e_3$). That $2L+\vec e_1$ and $2L+\vec e_2$ are in the same class follows by checking that $(\tau_{\vec e_3}\tau_{\vec e_1- \vec e_2})(2L+\vec e_1) = 2L+\vec e_2$ where $\tau_{\vec e_3}\tau_{\vec e_1 - \vec e_2}\in O^+(L)$. By going through all 8 elements of $O^+(L)$, it is easy to see that $2L+\vec e_1$ and $2L+\vec e_3$ are in distinct classes. Consider $\Sigma \in O_\A'(V)$ given by $\Sigma_2 = \tau_{\vec e_2 + \vec e_3}\tau_{\vec e_1 + \vec e_3}$, and $\Sigma_p$ is the identity transformation for all other $p$. Note that \[\Sigma_2 \vec e_1 = \frac{b-a}{b+a}\vec e_1 + \frac{4ab}{(a+b)^2} \vec e_2 + \frac{2a(b-a)}{(a+b)^2} \vec e_3 , \qquad
		\Sigma_2 \vec e_2 = \frac{b-a}{b+a}\vec e_2 - \frac{2a}{a+b} \vec e_3, \]
		\[\Sigma_2 \vec e_3 = \frac{-2b}{b+a}\vec e_1 + \frac{2b(b-a)}{(a+b)^2} \vec e_2 + \frac{(a-b)^2}{(a+b)^2} \vec e_3. \]
		Since $a\equiv b\equiv \pm1 \modc{4}$, it follows that $4|(b-a)$ and $\frac12 (a+b) \in \Z_2^\times$, and so $\Sigma_2\in O^+(L_2)$ and $\Sigma_2 \vec e_3 \equiv \vec e_1 \modc{2L_2}$. Hence $\Sigma_2(2L_2 + \vec e_3) = 2L_2 +\vec e_1$, and since $2L_p + \vec e_3 = L_p = 2L_p + \vec e_1$ for all other $p$, it follows that $\Sigma(2L+\vec e_3) = 2L+\vec e_1$. Therefore, $2L+\vec e_1$ and $2L+\vec e_3$ lie in the same spinor genus. 
		
		Finally, that these two classes exhaust the genus can be seen by noticing that $|H_{L,\vec e_1, 8,2}| = 2$, $|O^+(2L+\vec e_1)| = 4$, $O^+(2L+\vec e_3) = O^+(L)$, and so \[\sum_{2L+\vec h' \in \mc G} \frac{1}{|O^+(2L+\vec h')|} = \frac{|G_{L,8,2}|}{|H_{L,\vec h, 8,2}| \cdot |O^+(L)|} = \frac{3}{8} = \frac{1}{|O^+(2L+\vec e_1)|} + \frac{1}{|O^+(2L+\vec e_3)|}\]
		by equation (\ref{eqn::mass_formula_cond2}).
	\end{proof}
	
	\begin{proposition}\label{prop::classify_cond2_(1,3,9)}
		For $\vec a = (1,3,9)^T$, we have 
		\begin{align*}
			\gen^+(2L+\vec e_1) &= \spn^+(2L+ \vec e_1) \sqcup \spn^+(2L+\vec e_3)\\
			\gen^+(2L+\vec e_2) &= \spn^+(2L+ \vec e_2) = \cls^+(2L+\vec e_2)
		\end{align*}
		where $\spn^+(2L+\vec e_1) = \cls^+(2L+\vec e_1)$ and $\spn^+(2L+\vec e_3) = \cls^+(2L+\vec e_3)$.
	\end{proposition}
	\begin{remark}
		Hence $2L+\vec e_1$ and $2L+\vec e_3$ have spinor class number 1 (but not class number 1), which also explains the identities found for $\vec a = (1,3,9)^T$ and $N=2$ in Appendix \ref{sctn::list_spn1_identities}.
	\end{remark}
	\begin{proof}
		Note that $G_{L,8,2} = \{E_{\mathrm{id}}, E_{(1,3)}\}$ and $|O^+(L)| = 4$ (since now $\sigma \vec e_i = \pm \vec e_i$ for any $\sigma \in O^+(L)$). In particular, it is seen that $2L+\vec e_1$, $2L+\vec e_2$ and $2L+\vec e_3$ are all in distinct classes. Also, as both matrices in $G_{L,8,2}$ fix $(0,1,0)^T$, it follows that $2L+\vec e_1$ and $2L+\vec e_2$ are in distinct genera. On the other hand, since $E_{(1,3)}$ sends $(1,0,0)^T$ to $(0,0,1)^T$, it follows that $2L+\vec e_1$ and $2L+\vec e_3$ are in the same genus. Since $H_{L,\vec e_2, 8,2} = G_{L,8,2}$, $H_{L,\vec e_1, 8,2}=\{E_{\mathrm{id}}\}$ and $O^+(L) = O^+(2L+ \vec e_i)$ for $i=1,2,3$, equation (\ref{eqn::mass_formula_cond2}) implies that $2L+\vec e_2$ is the only class in its genus, while $2L+\vec e_1$ and $2L+\vec e_3$ exhaust all classes in their genus. This gives the classification into genera and classes. Let us now consider the spinor genera.
		
		For all $p\ne 2,3$, notice that $\theta(O^+(2L_p + \vec e_1)) = \theta(O^+(L_p)) = \Z_p^\times (\Q_p^\times)^2$. For $p=3$, the algorithm in Lemma \ref{lem::conway_algorithm} gives $\theta(O^+(2L_3 + \vec e_1)) = \theta(O^+(L_3)) = \{1,3\}(\Q_3^\times)^2$. For $p=2$, consider $M_2= \Z_2 \vec e_1 + 2L_2 \cong \diag\{1,12,36\}$ in the basis $\{\vec e_1, 2\vec e_2, 2\vec e_3\}$. By Lemma \ref{lem::conway_algorithm} again, we have $\theta(O^+(2L_2 + \vec e_1)) \subseteq \theta(O^+(M_2)) = \Z_2^\times (\Q_2^\times )^2$. Since $\tau_{\vec e_1}, \tau_{\vec e_2}, \tau_{2\vec e_2+\vec e_3} \in O(2L_2+\vec e_1)$ have spinor norms $(\Q_2^\times )^2$, $3(\Q_2^\times )^2$, and $5(\Q_2^\times )^2$, it follows that $\theta(O^+(2L_2 + \vec e_1)) = \Z_2^\times \Q_2^\times$. From Lemma \ref{lem::numspn_theta_triv_p_ne_q} it then follows that there are exactly two spinor genera in the genus of $2L+\vec e_1$. Since there are only two classes in this genus as well, the classification given above follows.
	\end{proof}
	
	From this result, along with Proposition \ref{prop::classify_cond2_(1,1,1)} and Lemma \ref{lem::cond2_symmetry_of_100_011}, it also follows that both $2L+ (\vec e_1+\vec e_2)$ and $2L+(\vec e_2 + \vec e_3)$ have spinor class number 1; however no identity is given in Appendix \ref{sctn::list_spn1_identities} for these two lattice cosets. This is because their theta series satisfy the following identities 
	\[\Theta_{\vec a, (1,1,0)^T, 2} = \frac12 \Theta_{\vec a} - \frac12 \Theta_{\vec a, \vec 0, 2} + 2\theta_{\chi_{-12},4} \qquad \Theta_{\vec a, (0,1,1)^T, 2} = \frac12 \Theta_{\vec a} - \frac12 \Theta_{\vec a, \vec 0, 2}- 2\theta_{\chi_{-12},4}\]
	where $\Theta_{\vec a, \vec 0, 2}$ is the generating function for $r_{\vec a, \vec 0, 2}(n)$, which is just equal to $r_{\vec a}(n/4)$ if $4|n$, and is 0 otherwise. In other words, there is a contribution corresponding to imprimitive representations of $n$ (i.e. where $\gcd(h_i, N)\ne 1$ for all $i$) by $Q_{\vec a}$ whenever $4|n$. While this identity is a valid pseudo spinor class number 1 identity, it was out of the scope of this computer search.
	
	\begin{proposition}
		For $\vec a = (1,25,5)^T, (1,9,21)^T, (7,63,3)^T$, we have \[\gen^+(2L+\vec e_1) = \spn^+(2L+ \vec e_1) = \cls^+(2L+\vec e_1) \sqcup \cls^+(2L+\vec e_2) \sqcup \cls^+(2L+\vec e_3)\]
	\end{proposition}
	\begin{proof}
		Note that $a_1,a_2,a_3$ are distinct, $a_1\equiv a_2 \not\equiv a_3\modc{8}$ but $a_1\equiv a_2\equiv a_3\modc{4}$. Thus, $|O^+(L)|=4$ and $|G_{L,8,2}| = 6$. It is easily seen that $2L+\vec e_i$ for $i=1,2,3$ are in distinct classes. Now, consider $\Sigma \in O_\A(V)$ for which $\Sigma_p$ is the identity for all $p\ne 2$, and \[\Sigma_2 = \begin{cases}
			\tau_{\vec e_1 - \vec e_2} \tau_{\vec e_1 - \vec e_3} \tau_{\vec e_2 - \vec e_3} \tau_{2\vec e_1 + \vec e_2 + \vec e_3} & \text{if } \vec a = (1,25,5)^T, (1,9,21)^T,\\
			\tau_{\vec e_1 + \vec e_3} \tau_{2\vec e_1 + \vec e_2 + \vec e_3} & \text{if } \vec a = (7, 63, 3)^T.
		\end{cases}\]
		It can be checked that $\Sigma_2 \in O'(L_2)$, and is $E_{(1,3,2)}$ under the homomorphism $\Psi$ from Proposition \ref{prop::matrix_reduction_of_isometries}. Thus $\Sigma_2$ maps $2L_2 +\vec e_1$ to $2L_2 + \vec e_3$ and $2L_2+\vec e_3$ to $2L_2+\vec e_2$, and so $\Sigma \in O'_\A(V)$ maps $2L+\vec e_1$ and $2L+\vec e_3$ to $2L+\vec e_3$ and $2L+\vec e_2$ respectively. Hence $2L+\vec e_i$ for $i=1,2,3$ all lie in the same spinor genus. Finally, since $H_{L, \vec e_1, 8,2} = \{E_{\mathrm{id}}, E_{(2,3)}\}$, and $O^+(2L+\vec e_i) = O^+(L)$ for $i=1,2,3$, equation (\ref{eqn::mass_formula_cond2}) implies that these three classes exhaust the classes in the genus.
	\end{proof}
	
	\begin{proposition}
		For $\vec a = (1,1,3)^T, (3,3,1)^T$, we have 
		\begin{align*}
			\gen^+(2L+\vec e_1) &= \spn^+(2L+\vec e_1) = \cls^+(2L+\vec e_1) = \cls^+(2L+\vec e_2)\\
			\gen^+(2L+\vec e_3) &= \spn^+(2L+\vec e_3) = \cls^+(2L+\vec e_3) 
		\end{align*}
		and these two genera are distinct.
	\end{proposition}
	\begin{proof}
		Notice that, in either case, $\vec a = (a,a,b)^T$ where $a\not\equiv b\modc{4}$. This implies that $|O^+(L)| = 8$ and $G_{L,8,2} = \{E_{\mathrm{id}}, E_{(1,2)}\}$. It follows that $2L+\vec e_1$ and $2L+\vec e_3$ are in distinct genera, since both matrices in $G_{L,8,2}$ fix $(0,0,1)^T$ modulo 2. As $\tau_{\vec e_3}\tau_{\vec e_1 - \vec e_2} \in O^+(L)$ maps $\vec e_1$ to $\vec e_2$, it follows that $2L+\vec e_1$ and $2L+\vec e_2$ lie in the same class. Finally, since $H_{L,\vec e_1, 8,2} = \{E_{\mathrm{id}}\}$, $H_{L, \vec e_3,8,2} = G_{L,8,2}$, $|O^+(2L+\vec e_1)| = 4$ and $O^+(2L+\vec e_3) = O^+(L)$, equation (\ref{eqn::mass_formula_cond2}) implies that both $2L+\vec e_1$ and $2L+\vec e_3$ have only one class in their respective genera.
	\end{proof}
	
	\appendix 
	\section{List of All Pseudo-Spinor Class 1 Identities} \label{sctn::list_spn1_identities}
	Table \ref{table::list_spn1_identities} below gives a list of 85 lattice cosets satisfying identities of the form described in Section \ref{sctn::comp_search}. The sub-columns under the column ``$\mc E$'' give the terms $\Theta_{\vec a}\big| S_{M,m}$ and their corresponding coefficients in the linear combination comprising $\mc E$. Multiple rows for the same value of $\vec a, N, \vec h$ imply multiple terms in the linear combination giving $\mc E$. If there is a congruence class $m \modc{M}$ that is missing from the table, the coefficient of $\Theta_{\vec a} \big| S_{M,m}$ is assumed to be zero. Similarly, the linear combination $\mc U$ of unary theta series is given in the column headed by $\mc U$. Here, if a value of $M'$ and $m'$ is given, then the corresponding term is $\theta_{\chi, t} \big| S_{M', m'}$. As an example, the identity for $\vec a = (1,1,3)^T, N = 12, \vec h = (0,3,1)^T$ is \[ \Theta_{\vec a,\vec h,12} = \frac1{20} \Theta_{\vec a}\Big |S_{144, 12} + \frac1{32} \Theta_{\vec a}\Big| S_{288, 84} + \frac1{24} \Theta_{\vec a}\Big| S_{288, 228} + \frac12 \theta_{\chi_{-12}, 12} \]
	while that for $\vec a = (1,1,2)^T, N = 8, \vec h = (0,4,1)^T$ is \[ \Theta_{\vec a, \vec h, 8} = \frac1{24} \Theta_{\vec a} \Big| S_{32, 18} - \frac14 \theta_{\chi_{-4}, 2} \Big| S_{32, 18}. \]
	The right most column gives the number of initial coefficients that need to be checked to prove the identity.
	
	For those rows whose $\vec h$ cell is marked with a \dag, a proof that the corresponding lattice coset indeed has proper spinor class number 1 is given in Section \ref{sctn::classify_found_cosets}, in Propositions \ref{prop::2cls_in_genus} or \ref{prop::3+cls_in_genus_spnnum=1}. The proof for the two lattice cosets corresponding to $\vec a = (1,3,9)^T$ and $N=2$ is given in Section \ref{sctn::classify_cond2}, Proposition \ref{prop::classify_cond2_(1,3,9)}. For those cells marked with a \ddag, the corresponding lattice coset has proper spinor class number greater than 1, a proof of which is again given in Section \ref{sctn::classify_found_cosets}, in Propositions \ref{prop::spnnum=2} or \ref{prop::spnnum=3}.
	
	The code used to check the identities may be found on the author's GitHub page.
	\newpage 
	{\setlength{\extrarowheight}{2pt}
	\begin{footnotesize}
	\begin{longtable}[c]{|c|c|c|c|c|c|c|c|c|c|c|c|c|} 
		\caption{\normalsize{List of 85 Pseudo-Spinor Class 1 Identities}}\label{table::list_spn1_identities}\\
		\hline
		\multirow[c]{2}{*}{$\vec a$} & \multirow{2}{*}{$N$} & \multirow{2}{*}{$\vec h$} & \multicolumn{3}{|c|}{$\mc E$} & \multicolumn{5}{|c|}{$\mc U$} & Congruence & Valence \\ \cline{4-11}
		& & & $M$ & $m$ & coeff. & $\chi$ & $t$ & coeff. & $M'$ & $m'$ & Subgroup $\Gamma$ & Bound\\
		\hline 
		\endfirsthead 
		
		\hline
		\multirow[c]{2}{*}{$\vec a$} & \multirow{2}{*}{$N$} & \multirow{2}{*}{$\vec h$} & \multicolumn{3}{|c|}{$\mc E$} & \multicolumn{5}{|c|}{$\mc U$} & Congruence & Valence \\\nopagebreak \cline{4-11}
		& & & $M$ & $m$ & coeff. & $\chi$ & $t$ & coeff. & $M'$ & $m'$ & Subgroup & Bound\\\nopagebreak
		\hline 
		\endhead 
		
		\hline
		\multirow{4}{*}{$(1,1,1)^T$} & \multirow{2}{*}{4} & $(1,0,0)^T$\dag & 8 & 1 & $1/12$ & $\chi_{-4}$ & 1 & $1/2$ & - & - & $\Gamma_{64, 4}$ & 25  \\
		\cline{3-13}
		&  & $(1,2,2)^T$\dag & 8 & 1 & $1/12$ & $\chi_{-4}$ & 1 & $-1/2$ &- &- &  $\Gamma_{64, 4}$ & 25  \\
		\cline{2-13}
		& \multirow{2}{*}{8} & $(1,2,2)^T$\dag & 16 & 9 & $1/48$ & $\chi_{-4}$ & 1& $-1/8$ & 16& 9 & $\Gamma_{256,16}$ & 384 \\
		\cline{3-13}
		& & $(3,2,2)^T$\dag & 16 & 1 & $1/48$ & $\chi_{-4}$ & 1& $-1/8$ & 16& 1 & $\Gamma_{256,16}$ & 384 \\
		\hline
		
		\multirow{5}{*}{$(1,1,2)^T$} & \multirow{2}{*}{4} & $(1,1,0)^T$\dag & 8 & 2 & $1/12$ & $\chi_{-4}$ & $2$ & $1/2$ & - & - & $\Gamma_{128,4}$ & 48\\\cline{3-13} 
		& & $(1,1,2)^T$\dag & 8 & 2 & $1/12$ & $\chi_{-4}$ & $2$ & $-1/2$ & - & - & $\Gamma_{128,4}$ & 48\\\cline{2-13}
		& \multirow{3}{*}{8} & $(0,4,1)^T$\dag & 32 & 18 & $1/24$ & $\chi_{-4}$ & $2$ & $-1/4$ & 32 & 18 & $\Gamma_{1024,32}$ & 3072\\\cline{3-13} 
		& & $(0,4,3)^T$\dag & 32 & 2 & $1/24$ & $\chi_{-4}$ & $2$ & $-1/4$ & 32 & 2 & $\Gamma_{1024,32}$ & 3072\\\cline{3-13}
		& & $(1,3,2)^T$\ddag & 16 & 2 & $1/48$ & $\chi_{-4}$ & $2$ & $-1/8$ & - & - & $\Gamma_{512,16}$ & 768\\\hline
		
		\multirow{6}{*}{$(1,1,3)^T$} & \multirow{6}{*}{12} & \multirow{3}{*}{$(0,3,1)^T$} & 144 & 12 & $1/20$ & \multirow{3}{*}{$\chi_{-12}$} & \multirow{3}{*}{12} & \multirow{3}{*}{$1/2$} & \multirow{3}{*}{-} & \multirow{3}{*}{-} & \multirow{3}{*}{$\Gamma_{82944,228}$} & \multirow{3}{*}{1990656 }\\\cline{4-6} 
		& &  & 288 & 84 & $1/32$ & & & & & & & \\\cline{4-6}
		& &  & 288 & 228 & $1/24$ & & & & & & & \\\cline{3-13} 
		& & \multirow{3}{*}{$(0,3,5)^T$} & 144 & 12 & $1/20$ & \multirow{3}{*}{$\chi_{-12}$} & \multirow{3}{*}{12} & \multirow{3}{*}{$-1/2$} & \multirow{3}{*}{-} & \multirow{3}{*}{-} & \multirow{3}{*}{$\Gamma_{82944,228}$} & \multirow{3}{*}{1990656 }\\\cline{4-6} 
		& &  & 288 & 84 & $1/32$ & & & & & & & \\\cline{4-6}
		& &  & 288 & 228 & $1/24$ & & & & & & & \\\hline
		
		\multirow{7}{*}{$(1,1,4)^T$} & \multirow{2}{*}{2} & $(0,1,0)^T$\dag & 4 & 1 & $1/4$ & $\chi_{-4}$ & 1 & 1 & - & - & $\Gamma_{64,2}$ & 25 \\\cline{3-13}
		& & $(0,1,1)^T$\dag & 4 & 1 & $1/4$ & $\chi_{-4}$ & 1 & $-1$ & - & - & $\Gamma_{64,2}$ & 25 \\\cline{2-13}
		& 4 & $(1,2,1)^T$\dag & 8 & 1 & $1/16$ & $\chi_{-4}$ & 1 & $-1/4$ & - & - & $\Gamma_{256,4}$ & 96 \\\cline{2-13}
		& \multirow{4}{*}{8} & $(0,0,1)^T$\dag & 64 & 4 & $1/12$ & $\chi_{-4}$ & 4 & $1/2$ & 64 & 4 & $\Gamma_{4096,64}$ & 24576 \\\cline{3-13}
		& & $(0,0,3)^T$\dag & 64 & 36 & $1/12$ & $\chi_{-4}$ & 4 & $1/2$ & 64 & 36 & $\Gamma_{4096,64}$ & 24576 \\\cline{3-13}
		& & $(4,4,1)^T$\dag & 64 & 36 & $1/12$ & $\chi_{-4}$ & 4 & $-1/2$ & 64 & 36 & $\Gamma_{4096,64}$ & 24576 \\\cline{3-13}
		& & $(4,4,3)^T$\dag & 64 & 4 & $1/12$ & $\chi_{-4}$ & 4 & $-1/2$ & 64 & 4 & $\Gamma_{4096,64}$ & 24576 \\\hline
		
		\multirow{5}{*}{$(1,1,8)^T$} & \multirow{2}{*}{2} & $(1,1,0)^T$\dag & 8 & 2 & $1/2$ & $\chi_{-4}$ & 2 & 2 & - & - & $\Gamma_{128, 2}$ & 25 \\\cline{3-13}
		& & $(1,1,1)^T$\dag & 8 & 2 & $1/2$ & $\chi_{-4}$ & 2 & $-2$ & - & - & $\Gamma_{128, 2}$ & 25 \\\cline{2-13}
		& 4 & $(1,1,1)^T$\dag & 8 & 2 & $1/16$ & $\chi_{-4}$ & 2 & $-1/4$ & - & - & $\Gamma_{512, 4}$ & 192 \\\cline{2-13} 
		& \multirow{2}{*}{8} & $(0,0,1)^T$ & 64 & 8 & $1/12$ & $\chi_{-2}$ & 8 & $1/2$ & - & - & $\Gamma_{4096, 64}$ & 24576 \\\cline{3-13}
		& & $(0,0,3)^T$ & 64 & 8 & $1/12$ & $\chi_{-2}$ & 8 & $-1/2$ & - & - & $\Gamma_{4096, 64}$ & 24576 \\\hline
		%
		%
		\multirow{4}{*}{$(1,2,2)^T$} & 4 & $(1,0,2)^T$\dag & 8 & 1 & $1/8$ & $\chi_{-4}$ & 1 & $-1/4$ & - & - & $\Gamma_{128, 4}$ & 48 \\\cline{2-13} 
		& \multirow{3}{*}{8} & $(2,1,1)^T$ & 32 & 8 & $1/24$ & $\chi_{-4}$ & 8 & $1/2$ & - & - & $\Gamma_{1024, 32}$ & 3072 \\\cline{3-13} 
		& & $(2,3,3)^T$ & 32 & 8 & $1/24$ & $\chi_{-4}$ & 8 & $-1/2$ & - & - & $\Gamma_{1024, 32}$ & 3072 \\\cline{3-13} 
		& & $(4,1,3)^T$\dag & 32 & 4 & $1/24$ & $\chi_{-4}$ & 4 & $-1/4$ & - & - & $\Gamma_{1024, 32}$ & 3072 \\\hline
		
		\multirow{4}{*}{$(1,2,4)$} & \multirow{4}{*}{8} & $(0,1,2)^T$\ddag & 32 & 18 & $1/16$ & $\chi_{-4}$ & 2 & $-1/8$ & 32 & 18 & $ \Gamma_{1024, 32} $ & 3072 \\\cline{3-13}
		& & $(0,3,2)^T$\ddag & 32 & 2 & $1/16$ & $\chi_{-4}$ & 2 & $-1/8$ & 32 & 18 & $ \Gamma_{1024, 32} $ & 3072 \\\cline{3-13}
		&  & $(2,2,1)^T$ & 32 & 16 & $1/24$ & $\chi_{-4}$ & 16 & $1/2$ & - & - & $ \Gamma_{1024, 32} $ & 3072 \\\cline{3-13}
		&  & $(2,2,3)^T$ & 32 & 16 & $1/24$ & $\chi_{-4}$ & 16 & $-1/2$ & - & - & $ \Gamma_{1024, 32} $ & 3072 \\\hline
		
		\multirow{5}{*}{$(1,2,8)^T$} & \multirow{3}{*}{4} & $(0,1,0)^T$ \dag & 16 & 2 & $1/4$ & $\chi_{-2}$ & 2 & $1/2$ & - & - & $\Gamma_{512, 16}$ & 768 \\\cline{3-13} 
		& & $(0,1,2)^T$ \dag & 16 & 2 & $1/4$ & $\chi_{-2}$ & 2 & $-1/2$ & - & - & $\Gamma_{512, 16}$ & 768 \\\cline{3-13} 
		& & $(1,0,1)^T$\ddag & 8 & 1 & $1/16$ & $\chi_{-4}$ & 1 & $-1/8$ & - & - & $\Gamma_{512, 4}$ & 192 \\\cline{2-13} 
		& \multirow{2}{*}{8} & $(4,2,1)^T$ & 64 & 32 & $1/12$ & $\chi_{-4}$ & 32 & 1 & - & - & $\Gamma_{4096, 64}$ & 24576 \\\cline{3-13} 
		& & $(4,2,3)^T$ & 64 & 32 & $1/12$ & $\chi_{-4}$ & 32 & $-1$ & - & - & $\Gamma_{4096, 64}$ & 24576 \\\hline
		
		\multirow{4}{*}{$(1,2,16)^T$} & \multirow{2}{*}{2} & $(1,0,0)^T$\dag & 8 & 1 & $1/2$ & $\chi_{-2}$ & 1 & 1 & - & - & $\Gamma_{256, 2}$ & 128 \\\cline{3-13}
		& & $(1,0,1)^T$\dag & 8 & 1 & $1/2$ & $\chi_{-2}$ & 1 & $-1$ & - & - & $\Gamma_{256, 2}$ & 128 \\\cline{2-13}
		& \multirow{2}{*}{8} & $(4,4,1)^T$ & 128 & 64 & $1/6$ & $\chi_{-4}$ & 64 & 2 & - & - & $\Gamma_{16384, 128}$ & 196608 \\\cline{3-13}
		& & $(4,4,3)^T$ & 128 & 64 & $1/6$ & $\chi_{-4}$ & 64 & $-2$ & - & - & $\Gamma_{16384, 128}$ & 196608 \\\hline 
		
		\multirow{7}{*}{$(1,3,3)^T$} & \multirow{4}{*}{6} & \multirow{2}{*}{$(0,0,1)^T$\dag} & 72 & 3 & $1/8$ & \multirow{2}{*}{$\chi_{-12}$} & \multirow{2}{*}{3} & \multirow{2}{*}{$1/2$} & \multirow{2}{*}{-} & \multirow{2}{*}{-} & \multirow{2}{*}{$\Gamma_{5184,72}$} & \multirow{2}{*}{31104} \\\cline{4-6}
		& & & 72 & 39 & $1/16$ & & & & & & & \\\cline{3-13}
		& & \multirow{2}{*}{$(0,2,3)^T$\dag} & 72 & 3 & $1/8$ & \multirow{2}{*}{$\chi_{-12}$} & \multirow{2}{*}{3} & \multirow{2}{*}{$-1/2$} & \multirow{2}{*}{-} & \multirow{2}{*}{-} & \multirow{2}{*}{$\Gamma_{5184,72}$} & \multirow{2}{*}{31104} \\\cline{4-6}
		& & & 72 & 39 & $1/16$ & & & & & & & \\\cline{2-13}
		& \multirow{3}{*}{12} & \multirow{3}{*}{$(3,3,4)^T$ \ddag} & 144 & 84 & $1/40$ & \multirow{3}{*}{$\chi_{-12}$} & \multirow{3}{*}{12} & \multirow{3}{*}{$-1/4$} & \multirow{3}{*}{-} & \multirow{3}{*}{-} & \multirow{3}{*}{$\Gamma_{82944,288}$} & \multirow{3}{*}{1990656} \\\cline{4-6}
		& & & 288 & 12 & $1/48$ & & & & & & & \\\cline{4-6}
		& & & 288 & 156 & $1/64$ & & & & & & & \\\hline
		\pagebreak
		
		\multirow{18}{*}{$(1,3,9)^T$} & \multirow{4}{*}{2} & \multirow{2}{*}{$(0,0,1)^T$\dag } & 8 & 1 & $1/2$ & \multirow{2}{*}{$\chi_{-12}$} & \multirow{2}{*}{1} & \multirow{2}{*}{$-1$} & \multirow{2}{*}{-} & \multirow{2}{*}{-} & \multirow{2}{*}{$\Gamma_{576,2}$} & \multirow{2}{*}{144} \\\cline{4-6}
		& & & 8 & 5 & $1/4$ & & & & & & & \\\cline{3-13}
		& & \multirow{2}{*}{$(1,0,0)^T$\dag} & 8 & 1 & $1/2$ & \multirow{2}{*}{$\chi_{-12}$} & \multirow{2}{*}{1} & \multirow{2}{*}{$1$} & \multirow{2}{*}{-} & \multirow{2}{*}{-} & \multirow{2}{*}{$\Gamma_{576,2}$} & \multirow{2}{*}{144} \\\cline{4-6}
		& & & 8 & 5 & $1/4$ & & & & & & & \\\cline{2-13}
		& \multirow{6}{*}{4} & \multirow{3}{*}{$(0,1,1)^T$\dag} & 16 & 12 & $1/10$ & \multirow{3}{*}{$\chi_{-12}$} & \multirow{3}{*}{4} & \multirow{3}{*}{$-1/2$} & \multirow{3}{*}{-} & \multirow{3}{*}{-} & \multirow{3}{*}{$\Gamma_{9216,32}$} & \multirow{3}{*}{36864} \\\cline{4-6}	
		& & & 32 & 4 & $1/12$ & & & & & & & \\\cline{4-6}
		& & & 32 & 20 & $1/16$ & & & & & & & \\\cline{3-13}
		& & \multirow{3}{*}{$(1,1,0)^T$\dag} & 16 & 12 & $1/10$ & \multirow{3}{*}{$\chi_{-12}$} & \multirow{3}{*}{4} & \multirow{3}{*}{$1/2$} & \multirow{3}{*}{-} & \multirow{3}{*}{-} & \multirow{3}{*}{$\Gamma_{9216,32}$} & \multirow{3}{*}{36864} \\\cline{4-6}	
		& & & 32 & 4 & $1/12$ & & & & & & & \\\cline{4-6}
		& & & 32 & 20 & $1/16$ & & & & & & & \\\cline{2-13}
		& \multirow{8}{*}{6} & \multirow{4}{*}{$(0,0,1)$ \dag} & 216 & 9 & $1/8$ & \multirow{4}{*}{$\chi_{-12}$} & \multirow{4}{*}{9} & \multirow{4}{*}{$1/2$} & \multirow{4}{*}{-} & \multirow{4}{*}{-} & \multirow{4}{*}{$\Gamma_{46656, 216}$} & \multirow{4}{*}{839808} \\\cline{4-6}
		& & & 216 & 45 & $1/8$ & & & & & & & \\\cline{4-6}
		& & & 216 & 117 & $1/16$ & & & & & & & \\\cline{4-6}
		& & & 216 & 153 & $1/4$ & & & & & & & \\\cline{3-13}
		& & \multirow{4}{*}{$(3,0,2)$ \dag} & 216 & 9 & $1/8$ & \multirow{4}{*}{$\chi_{-12}$} & \multirow{4}{*}{9} & \multirow{4}{*}{$-1/2$} & \multirow{4}{*}{-} & \multirow{4}{*}{-} & \multirow{4}{*}{$\Gamma_{46656, 216}$} & \multirow{4}{*}{839808} \\\cline{4-6}
		& & & 216 & 45 & $1/8$ & & & & & & & \\\cline{4-6}
		& & & 216 & 117 & $1/16$ & & & & & & & \\\cline{4-6}
		& & & 216 & 153 & $1/4$ & & & & & & & \\\hline
		
		\multirow{12}{*}{$(1,3,12)^T$} & \multirow{4}{*}{3} & \multirow{2}{*}{$(0,0,1)^T$\dag} & \multirow{2}{*}{9} & \multirow{2}{*}{3} & \multirow{2}{*}{$1/4$} & $\chi_{-12}$ & 3 & $-1/2$ & \multirow{2}{*}{-} & \multirow{2}{*}{-} & \multirow{2}{*}{$\Gamma_{1296,9}$} & \multirow{2}{*}{1944} \\\cline{7-9}
		& & & & & & $\chi_{-3}$ & 12 & $-1$ & & & & \\\cline{3-13}
		& & \multirow{2}{*}{$(0,1,0)^T$ \dag} & \multirow{2}{*}{9} & \multirow{2}{*}{3} & \multirow{2}{*}{$1/4$} & $\chi_{-12}$ & 3 & $1/2$ & \multirow{2}{*}{-} & \multirow{2}{*}{-} & \multirow{2}{*}{$\Gamma_{1296,9}$} & \multirow{2}{*}{1944} \\\cline{7-9}
		& & & & & & $\chi_{-3}$ & 12 & $1$ & & & & \\\cline{2-13}
		& \multirow{6}{*}{6} & \multirow{3}{*}{$(3,1,0)^T$ \dag} & 144 & 84 & $1/6$ & \multirow{3}{*}{$\chi_{-12}$} & \multirow{3}{*}{12} & \multirow{3}{*}{1} & \multirow{3}{*}{-} & \multirow{3}{*}{-} & \multirow{3}{*}{$\Gamma_{82944,288}$} & \multirow{3}{*}{1990656} \\\cline{4-6}
		& & & 288 & 12 & $1/8$ & & & & & & & \\\cline{4-6}
		& & & 288 & 156 & $1/12$ & & & & & & & \\\cline{3-13}
		& & \multirow{3}{*}{$(3,3,2)^T$ \dag} & 144 & 84 & $1/6$ & \multirow{3}{*}{$\chi_{-12}$} & \multirow{3}{*}{12} & \multirow{3}{*}{$-1$} & \multirow{3}{*}{-} & \multirow{3}{*}{-} & \multirow{3}{*}{$\Gamma_{82944,288}$} & \multirow{3}{*}{1990656} \\\cline{4-6}
		& & & 288 & 12 & $1/8$ & & & & & & & \\\cline{4-6}
		& & & 288 & 156 & $1/12$ & & & & & & & \\\cline{2-13}
		& \multirow{2}{*}{12} & \multirow{2}{*}{$(0,4,3)^T$\dag } & 288 & 12 & $1/32$ & \multirow{2}{*}{$\chi_{-12}$} & \multirow{2}{*}{12} & \multirow{2}{*}{$-1/4$} & \multirow{2}{*}{-} & \multirow{2}{*}{-} & \multirow{2}{*}{$\Gamma_{82944,288}$} & \multirow{2}{*}{1990656} \\\cline{4-6}
		& & & 288 & 156 & $1/48$ & & & & & & & \\\hline
		
		\multirow{7}{*}{$(1,4,4)^T$} & \multirow{2}{*}{2} & $(1,0,0)^T$\dag & 8 & 1 & $1/2$ & $\chi_{-4}$ & 1 & 1 & - & - & $\Gamma_{64,2}$ & 12 \\\cline{3-13}\nopagebreak
		& & $(1,1,1)^T$\dag & 8 & 1 & $1/2$ & $\chi_{-4}$ & 1 & $-1$ & - & - & $\Gamma_{64,2}$ & 12 \\\cline{2-13}\nopagebreak
		& \multirow{3}{*}{4} & $(0,0,1)^T$\dag & 16 & 4 & $1/12$ & $\chi_{-4}$ & 4 & $1/2$ & - & - & $\Gamma_{256,16}$ & 384 \\\cline{3-13}\nopagebreak
		& & $(0,2,1)^T$\dag & 16 & 4 & $1/12$ & $\chi_{-4}$ & 4 & $-1/2$ & - & - & $\Gamma_{256,16}$ & 384 \\\cline{3-13}\nopagebreak
		& & $(1,1,1)^T$\dag & 8 & 1 & $1/16$ & $\chi_{-4}$ & 1 & $-1/8$ & - & - & $\Gamma_{256,4}$ & 96 \\\cline{2-13} \nopagebreak
		& \multirow{2}{*}{8} & $(4,1,2)^T$\dag & 64 & 36 & $1/24$ & $\chi_{-4}$ & 4 & $-1/4$ & 64 & 36 & $\Gamma_{4096, 64}$ & 24576 \\\cline{3-13}\nopagebreak
		& & $(4,3,2)^T$\dag & 64 & 4 & $1/24$ & $\chi_{-4}$ & 4 & $-1/4$ & 64 & 36 & $\Gamma_{4096, 64}$ & 24576 \\\hline 
		
		\multirow{6}{*}{$(1,4,8)^T$} & \multirow{2}{*}{4} & $(2,1,0)^T$\dag & 32 & 8 & $1/6$ & $\chi_{-4}$ & 8 & 1 & - & - & $\Gamma_{1024, 32}$ & 3072 \\\cline{3-13}
		& & $(2,1,2)^T$\dag & 32 & 8 & $1/6$ & $\chi_{-4}$ & 8 & $-1$ & - & - & $\Gamma_{1024, 32}$ & 3072 \\\cline{2-13}
		& \multirow{4}{*}{16} & $(0,4,1)$ & 256 & 72 & $1/48$ & $\chi_{-4}$ & 8 & $-1/8$ & - & - & $\Gamma_{65536, 256}$ & 1572864 \\\cline{3-13}
		& & $(0,4,3)$ & 256 & 136 & $1/48$ & $\chi_{-4}$ & 8 & $-1/8$ & - & - & $\Gamma_{65536, 256}$ & 1572864 \\\cline{3-13}
		& & $(0,4,5)$ & 256 & 8 & $1/48$ & $\chi_{-4}$ & 8 & $-1/8$ & - & - & $\Gamma_{65536, 256}$ & 1572864 \\\cline{3-13}
		& & $(0,4,7)$ & 256 & 200 & $1/48$ & $\chi_{-4}$ & 8 & $-1/8$ & - & - & $\Gamma_{65536, 256}$ & 1572864 \\\hline
		
		\multirow{6}{*}{$(1,4,12)^T$} & \multirow{6}{*}{12} & \multirow{3}{*}{$(6,0,1)^T$} & 576 & 48 & $1/10$ & \multirow{3}{*}{$ \chi_{-12} $} & \multirow{3}{*}{48} & \multirow{3}{*}{1} & \multirow{3}{*}{-} & \multirow{3}{*}{-} & \multirow{3}{*}{$\Gamma_{1327104,1152}$} & \multirow{3}{*}{127401984} \\\cline{4-6}
		& & & 1152 & 336 & $1/16$ & & & & & & & \\\cline{4-6}
		& & & 1152 & 912 & $1/12$ & & & & & & & \\\cline{3-13}
		& & \multirow{3}{*}{$(6,0,5)^T$} & 576 & 48 & $1/10$ & \multirow{3}{*}{$ \chi_{-12} $} & \multirow{3}{*}{48} & \multirow{3}{*}{$-1$} & \multirow{3}{*}{-} & \multirow{3}{*}{-} & \multirow{3}{*}{$\Gamma_{1327104,1152}$} & \multirow{3}{*}{127401984} \\\cline{4-6}
		& & & 1152 & 336 & $1/16$ & & & & & & & \\\cline{4-6}
		& & & 1152 & 912 & $1/12$ & & & & & & & \\\hline
		
		\multirow{2}{*}{$(1,6,6)^T$} & \multirow{2}{*}{12} & \multirow{2}{*}{$(0,1,5)^T$\dag} & 288 & 12 & $1/16$ & \multirow{2}{*}{$\chi_{-12}$} & \multirow{2}{*}{12} & \multirow{2}{*}{$-1/4$} & \multirow{2}{*}{-} & \multirow{2}{*}{-} & \multirow{2}{*}{$\Gamma_{82944,288}$} & \multirow{2}{*}{1990656} \\\cline{4-6}
		& & & 288 & 156 & $1/32$ & & & & & & & \\\hline
		\pagebreak
		
		\multirow{4}{*}{$(1,8,8)^T$} & 2 & $(1,0,1)^T$\dag & 8 & 1 & $1/4$ & $\chi_{-4}$ & 1 & $-1/2$ & - & - & $\Gamma_{128,2}$ & 24 \\\cline{2-13} 
		& \multirow{3}{*}{8} & $(0,1,3)^T$\dag & 64 & 16 & $1/24$ & $\chi_{-4}$ & 16 & $-1/4$ & - & - & $\Gamma_{4096, 64}$ & 24576 \\\cline{3-13}
		& & $(4,1,1)^T$ & 128 & 32 & $1/12$ & $\chi_{-4}$ & 32 & $1$ & - & - & $\Gamma_{16384, 128}$ & 196608 \\\cline{3-13}
		& & $(4,3,3)^T$ & 128 & 32 & $1/12$ & $\chi_{-4}$ & 32 & $-1$ & - & - & $\Gamma_{16384, 128}$ & 196608 \\\hline 
		
		\multirow{2}{*}{$(1,8,16)$} & \multirow{2}{*}{8} & $(4,2,1)^T$ & 128 & 64 & $1/12$ & $\chi_{-4}$ & 64 & 1 & - & - & $\Gamma_{16384, 128}$ & 196608\\\cline{3-13}
		& & $(4,2,3)^T$ & 128 & 64 & $1/12$ & $\chi_{-4}$ & 64 & $-1$ & - & - & $\Gamma_{16384, 128}$ & 196608\\\hline
		
		\multirow{3}{*}{$(1,12,12)^T$} & \multirow{3}{*}{12} & \multirow{3}{*}{$(6,3,4)^T$\dag} & 576 & 336 & $1/20$ & \multirow{3}{*}{$\chi_{-12}$} & \multirow{3}{*}{48} & \multirow{3}{*}{$-1/2$} & \multirow{3}{*}{-} & \multirow{3}{*}{-} & \multirow{3}{*}{$\Gamma_{1327104, 1152}$} & \multirow{3}{*}{127401984} \\\cline{4-6}
		& & & 1152 & 48 & $1/24$ & & & & & & & \\\cline{4-6}
		& & & 1152 & 624 & $1/32$ & & & & & & & \\\hline
		
		\multirow{14}{*}{$(3,4,12)^T$} & \multirow{2}{*}{3} & $(0,0,1)^T$\dag & 9 & 3 & $1/4$ & $\chi_{-12}$ & 3 & $-1/2$ & - & - & $\Gamma_{1296,9}$ & 1944\\\cline{3-13}
		&  & $(1,0,0)^T $\dag & 9 & 3 & $1/4$ & $\chi_{-12}$ & 3 & $1/2$ & - & - & $\Gamma_{1296,9}$ & 1944\\\cline{2-13}
		& \multirow{6}{*}{6} & \multirow{3}{*}{$(0,0,1)^T$\dag} & 144 & 120 & $1/4$ & \multirow{3}{*}{$\chi_{-12}$} & \multirow{3}{*}{12} & \multirow{3}{*}{$1/2$} & \multirow{3}{*}{-} & \multirow{3}{*}{-} & \multirow{3}{*}{$\Gamma_{82944, 288}$} & \multirow{3}{*}{1990656} \\\cline{4-6}
		& & & 288 & 12 & $1/8$ & & & & & & & \\\cline{4-6}
		& & & 288 & 156 & $1/16$ & & & & & & & \\\cline{3-13}
		& & \multirow{3}{*}{$(2,0,3)^T$\dag} & 144 & 120 & $1/4$ & \multirow{3}{*}{$\chi_{-12}$} & \multirow{3}{*}{12} & \multirow{3}{*}{$-1/2$} & \multirow{3}{*}{-} & \multirow{3}{*}{-} & \multirow{3}{*}{$\Gamma_{82944, 288}$} & \multirow{3}{*}{1990656} \\\cline{4-6}
		& & & 288 & 12 & $1/8$ & & & & & & & \\\cline{4-6}
		& & & 288 & 156 & $1/16$ & & & & & & & \\\cline{2-13}
		& \multirow{6}{*}{12} & \multirow{3}{*}{$(2,3,0)^T$\dag} & 576 & 336 & $1/20$ & \multirow{3}{*}{$\chi_{-12}$} & \multirow{3}{*}{48} & \multirow{3}{*}{$1/2$} & \multirow{3}{*}{-} & \multirow{3}{*}{-} & \multirow{3}{*}{$\Gamma_{1327104, 576}$} & \multirow{3}{*}{127401984} \\\cline{4-6}
		& & & 1152 & 48 & $1/24$ & & & & & & & \\\cline{4-6}
		& & & 1152 & 624 & $1/32$ & & & & & & & \\\cline{3-13}
		& & \multirow{3}{*}{$(6,3,4)^T$\dag} & 576 & 336 & $1/20$ & \multirow{3}{*}{$\chi_{-12}$} & \multirow{3}{*}{48} & \multirow{3}{*}{$-1/2$} & \multirow{3}{*}{-} & \multirow{3}{*}{-} & \multirow{3}{*}{$\Gamma_{1327104, 576}$} & \multirow{3}{*}{127401984} \\\cline{4-6}
		& & & 1152 & 48 & $1/24$ & & & & & & & \\\cline{4-6}
		& & & 1152 & 624 & $1/32$ & & & & & & & \\\hline
		
		\multirow{6}{*}{$(3,4,36)^T$} & \multirow{6}{*}{4} & \multirow{3}{*}{$(2,0,1)^T$\dag} & 64 & 48 & $1/5$ & \multirow{3}{*}{$\chi_{-12}$} & \multirow{3}{*}{16} & \multirow{3}{*}{$-1$} & \multirow{3}{*}{-} & \multirow{3}{*}{-} & \multirow{3}{*}{$\Gamma_{147456,128}$} & \multirow{3}{*}{2359296} \\\cline{4-6}
		& & & 128 & 16 & $1/6$ & & & & & & & \\\cline{4-6}
		& & & 128 & 80 & $1/8$ & & & & & & & \\\cline{3-13}
		& & \multirow{3}{*}{$(2,1,0)^T$ \dag} & 64 & 48 & $1/5$ & \multirow{3}{*}{$\chi_{-12}$} & \multirow{3}{*}{16} & \multirow{3}{*}{$1$} & \multirow{3}{*}{-} & \multirow{3}{*}{-} & \multirow{3}{*}{$\Gamma_{147456,128}$} & \multirow{3}{*}{2359296} \\\cline{4-6}
		& & & 128 & 16 & $1/6$ & & & & & & & \\\cline{4-6}
		& & & 128 & 80 & $1/8$ & & & & & & & \\\hline
	\end{longtable}
	\end{footnotesize}
}
	\section{Genus Information for Lattice Cosets Found}\label{sctn::gen_info}
	Table \ref{table::genus_info} describes the genus of the lattice cosets listed previously in Table \ref{table::list_spn1_identities}. All of the values occurring in equation (\ref{eqn::mass_formula}) are given in this table, which helps to prove Proposition \ref{prop::classify_genus_all_listed}.
	
	As usual, the coordinates of $\vec h_1$ and $\vec h_2$ are given in the standard basis $\{\vec e_1, \vec e_2, \vec e_3\}$ for $V = \Q^3$. Here, the matrix $X_{i,p }$ sends $\vec h_0$ to $\vec h_i$ modulo $q$, where $q$ is as in Proposition \ref{prop::matrix_reduction_of_isometries}. Here, $G_p$ and $H_p$ is shorthand for the groups $G_{L,q',q}$ and $H_{L, \vec h_0, q', q}$ from Proposition \ref{prop::matrix_reduction_of_isometries}. The sizes of the two groups $|G_{L,q',q}|$ and $|H_{L,\vec h_1, q',q'}|$ are computed by a computer. Since some of these computations were too large for the computer, this table does not exhaust all of the lattice cosets given in Table \ref{table::list_spn1_identities} in Appendix \ref{sctn::list_spn1_identities}. In particular, the following cosets are missing:
	\begin{multline*}
		(\vec a, \vec h, N)\in \Big\{  \left((1,4,8)^T, (0,4,1)^T, 16 \right), \left((1,4,8)^T, (0,4,3)^T, 16 \right), \left((1,4,8)^T, (0,4,5)^T, 16 \right), \\ 
		\left((1,4,8)^T, (0,4,7)^T, 16 \right), \left((1,8,16)^T, (4,2,1)^T, 8 \right), \left((1,8,16)^T, (4,2,3)^T, 8 \right) \Big\}.
	\end{multline*}

	The code used to evaluate $G_p$ and $H_p$ may be found on the author's GitHub page.
	\newpage 
	{
		\renewcommand{\arraystretch}{1.3}
	\begin{center}
	\begin{footnotesize}
	\begin{longtable}[c]{|c|c|c|c|c|c|c|c|c|}
		\caption[]{\normalsize{Genus information for $\gen^+(NL_{\vec a}+\vec h_0)$}} \label{table::genus_info}\\
		\hline 
		$\vec a$ & $N$ & $\vec h_0$ & $\vec h_1$ & $\vec h_2$ & $\vec h_3$ & $\vec h_4$ & $\vec h_5$ & Misc. Genus Info\\\hline \hline 
		\endfirsthead 
		\hline 
		$\vec a$ & $N$ & $\vec h_0$ & $\vec h_1$ & $\vec h_2$ & $\vec h_3$ & $\vec h_4$ & $\vec h_5$ & Misc. Genus Info\\\hline \hline 
		\endhead 
		
		$(1,1,1)^T$ & 4 & $(0,0,1)^T$ & $(2,2,1)^T$ & - & - & - & - & \\\cline{1-8}
		\multicolumn{2}{|c|}{$|O^+(NL+\vec h_i)|$} & 4 & 4 & - & - & - & - & $|O^+(L)|=24$ \\ \cline{1-8}
		\multicolumn{3}{|c|}{$ X_{i,2} \in G_{L,8,4}, \;X_{i,2} \vec h_0 \equiv \vec h_i$} & \rule{0pt}{14pt} $\left(\begin{smallmatrix} 2&2&1 \\1&2&2 \\2&1&2 \end{smallmatrix} \right)$ & - & - & - & - & \small{$|G_2|=48, |H_2| = 4$}\\[3pt]\hline 
		
		$(1,1,1)^T$ & 8 & $(1,2,2)^T$ & $(3,0,0)^T$ & $(3,4,0)^T$ & $(3,4,4)^T$ & - & - & \\\cline{1-8}
		\multicolumn{2}{|c|}{$|O^+(NL+\vec h_i)|$} & 1 & 4 & 2 & 4 & - & - & $|O^+(L)|=24$ \\ \cline{1-8}
		\multicolumn{3}{|c|}{$ X_{i,2} \in G_{L,16,8}, \;X_{i,2} \vec h_0 \equiv \vec h_i$} & \rule{0pt}{14pt} $ \left(\begin{smallmatrix} 3&6&6\\ 2&2&5 \\ 2&5&2 \end{smallmatrix} \right)$ & $ \left(\begin{smallmatrix} 3&2&6\\ 2&3&2 \\ 2&6&5 \end{smallmatrix} \right)$ & $ \left(\begin{smallmatrix} 3&2&2\\ 2&2&3 \\ 2&3&2 \end{smallmatrix} \right)$ & - & - & \small{$|G_2|=384, |H_2| = 8$}\\[3pt] \hline\hline  
		
		$(1,1,1)^T$ & 8 & $(3,2,2)^T$ & $(1,0,0)^T$ & $(1,4,0)^T$ & $(1,4,4)^T$ & - & - & \\\cline{1-8}
		\multicolumn{2}{|c|}{$|O^+(NL+\vec h_i)|$} & 1 & 4 & 2 & 4 & - & - & $|O^+(L)|=24$ \\ \cline{1-8}
		\multicolumn{3}{|c|}{$ X_{i,2} \in G_{L,16,8}, \;X_{i,2} \vec h_0 \equiv \vec h_i$} & \rule{0pt}{14pt}$ \left(\begin{smallmatrix} 3&6&6\\ 2&2&5 \\ 2&5&2 \end{smallmatrix} \right)$ & $ \left(\begin{smallmatrix} 3&2&6\\ 2&3&2 \\ 2&6&5 \end{smallmatrix} \right)$ & $ \left(\begin{smallmatrix} 3&2&2\\ 2&2&3 \\ 2&3&2 \end{smallmatrix} \right)$ & - & - & \small{$|G_2|=384, |H_2| = 8$} \\[3pt] \hline\hline 
		
		$(1,1,2)^T$ & 4 & $(1,1,0)^T$ & $(1,1,2)^T$ & - & - & - & - & \\\cline{1-8}
		\multicolumn{2}{|c|}{$|O^+(NL+\vec h_i)|$} & 2 & 2 & - & - & - & - & $|O^+(L)|=8$ \\ \cline{1-8}
		\multicolumn{3}{|c|}{$ X_{i,2} \in G_{L,16,4}, \;X_{i,2} \vec h_0 \equiv \vec h_i$} & \rule{0pt}{14pt} $\left(\begin{smallmatrix} 1&0&0 \\0&1&0 \\0&2&1 \end{smallmatrix} \right)$ & - & - & - & - & \small{$|G_2|=32, |H_2| = 4$}\\[3pt] \hline\hline  
		
		$(1,1,2)^T$ & 8 & $(0,4,1)^T$ & $(0,0,3)^T$ & $(4,4,3)^T$ & - & - & - & \\\cline{1-8}
		\multicolumn{2}{|c|}{$|O^+(NL+\vec h_i)|$} & 2 & 4 & 4 & - & - & - & $|O^+(L)|=8$ \\ \cline{1-8}
		\multicolumn{3}{|c|}{$ X_{i,2} \in G_{L,32,8}, \;X_{i,2} \vec h_0 \equiv \vec h_i$} & \rule{0pt}{14pt} $ \left(\begin{smallmatrix} 1&0&0\\ 0&3&4 \\ 4&6&3 \end{smallmatrix} \right)$ & $ \left(\begin{smallmatrix} 3&0&4\\ 0&1&0 \\ 2&0&3 \end{smallmatrix} \right)$ & - & - & - & \small{$|G_2|=256, |H_2| = 32$} \\[3pt] \hline\hline 
		
		$(1,1,2)^T$ & 8 & $(0,4,3)^T$ & $(0,0,1)^T$ & $(4,4,1)^T$ & - & - & - & \\\cline{1-8}
		\multicolumn{2}{|c|}{$|O^+(NL+\vec h_i)|$} & 2 & 4 & 4 & - & - & - & $|O^+(L)|=8$ \\ \cline{1-8}
		\multicolumn{3}{|c|}{$ X_{i,2} \in G_{L,32,8}, \;X_{i,2} \vec h_0 \equiv \vec h_i$} & \rule{0pt}{14pt} $ \left(\begin{smallmatrix} 1&0&0\\ 0&3&4 \\ 4&6&3 \end{smallmatrix} \right)$ & $ \left(\begin{smallmatrix} 3&0&4\\ 0&1&0 \\ 2&0&3 \end{smallmatrix} \right)$ & - & - & - & \small{$|G_2|=256, |H_2| = 32$} \\[3pt] \hline\hline 
		
		$(1,1,2)^T$ & 8 & $(1,3,2)^T$ & $(1,5,2)^T$ & $(1,1,0)^T$ & $(1,1,4)^T$ & $(3,3,0)^T$ & $(3,3,4)^T$ & \\\cline{1-8}
		\multicolumn{2}{|c|}{$|O^+(NL+\vec h_i)|$} & 1 & 1 & 2 & 2 & 2 & 2 & $|O^+(L)|=8$ \\ \cline{1-8}
		\multicolumn{3}{|c|}{$ X_{i,2} \in G_{L,32,8}, \;X_{i,2} \vec h_0 \equiv \vec h_i$} & \rule{0pt}{14pt} $ \left(\begin{smallmatrix} 1&0&0\\ 0&7&0 \\ 0&4&7 \end{smallmatrix} \right)$ & $ \left(\begin{smallmatrix} 1&0&0\\ 0&3&4 \\ 0&6&3 \end{smallmatrix} \right)$ & $ \left(\begin{smallmatrix} 1&0&0\\ 0&3&4 \\ 0&2&3 \end{smallmatrix} \right)$ & $ \left(\begin{smallmatrix} 3&0&4\\ 0&1&0 \\ 2&0&3 \end{smallmatrix} \right)$ & $ \left(\begin{smallmatrix} 3&0&4\\ 0&1&0 \\ 2&4&3 \end{smallmatrix} \right)$ & \small{$|G_2|=256, |H_2| = 8$} \\[3pt] \hline\hline 
		
		$(1,1,3)^T$ & 12 & $(0,3,1)^T$ & $(0,3,5)^T$ & - & - & - & - & \\\cline{1-8}
		\multicolumn{2}{|c|}{$|O^+(NL+\vec h_i)|$} & 1 & 1 & - & - & - & - & $|O^+(L)|=8$ \\ \cline{1-8}
		\multicolumn{3}{|c|}{$ X_{i,2} \in G_{L,8,4}, \;X_{i,2} \vec h_0 \equiv \vec h_i$} & \rule{0pt}{14pt} $ \left(\begin{smallmatrix} 1&0&0\\ 0&1&0 \\ 0&0&1 \end{smallmatrix} \right)$ & - & - & - & - & \small{$|G_2|=16, |H_2| = 2$}\\[3pt] \cline{1-8}  
		\multicolumn{3}{|c|}{$ X_{i,3} \in G_{L,9,3}, \;X_{i,3} \vec h_0 \equiv \vec h_i$} & \rule{0pt}{14pt} $ \left(\begin{smallmatrix} 1&0&0\\ 0&2&0 \\ 0&0&2 \end{smallmatrix} \right)$ & - & - & - & - & \small{$|G_3|=72, |H_3| = 36$}\\[3pt] \hline\hline  
		
		$(1,1,4)^T$ & 2 & $(0,1,0)^T$ & $(0,1,1)^T$ & - & - & - & - & \\\cline{1-8}
		\multicolumn{2}{|c|}{$|O^+(NL+\vec h_i)|$} & 4 & 4 & - & - & - & - & $|O^+(L)|=8$ \\ \cline{1-8}
		\multicolumn{3}{|c|}{$ X_{i,2} \in G_{L,32,2}, \;X_{i,2} \vec h_0 \equiv \vec h_i$} & \rule{0pt}{14pt} $\left(\begin{smallmatrix} 1&0&0 \\0&1&0 \\1&1&1 \end{smallmatrix} \right)$ & - & - & - & - & \small{$|G_2|=4, |H_2| = 1$}\\[3pt] \hline\hline  
		
		$(1,1,4)^T$ & 4 & $(1,2,1)^T$ & $(1,0,0)^T$ & $(1,0,2)^T$ & - & - & - & \\\cline{1-8}
		\multicolumn{2}{|c|}{$|O^+(NL+\vec h_i)|$} & 1 & 2 & 2 & - & - & - & $|O^+(L)|=8$ \\ \cline{1-8}
		\multicolumn{3}{|c|}{$ X_{i,2} \in G_{L,32,4}, \;X_{i,2} \vec h_0 \equiv \vec h_i$} & \rule{0pt}{14pt} $ \left(\begin{smallmatrix} 1&2&0\\ 2&3&0 \\ 3&3&3 \end{smallmatrix} \right)$ & $ \left(\begin{smallmatrix} 1&2&0\\ 2&3&0 \\ 1&1&3 \end{smallmatrix} \right)$ & - & - & - & \small{$|G_2|=64, |H_2| = 4$} \\[3pt] \hline\hline 
		
		$(1,1,4)^T$ & 8 & $(0,0,1)^T$ & $(4,4,3)^T$ & - & - & - & - & \\\cline{1-8}
		\multicolumn{2}{|c|}{$|O^+(NL+\vec h_i)|$} & 4 & 4 & - & - & - & - & $|O^+(L)|=8$ \\ \cline{1-8}
		\multicolumn{3}{|c|}{$ X_{i,2} \in G_{L,64,8}, \;X_{i,2} \vec h_0 \equiv \vec h_i$} & \rule{0pt}{14pt} $\left(\begin{smallmatrix} 3&6&4 \\2&5&4 \\5&5&3 \end{smallmatrix} \right)$ & - & - & - & - & \small{$|G_2|=512, |H_2| = 128$}\\[3pt] \hline\hline  
		
		$(1,1,4)^T$ & 8 & $(0,0,3)^T$ & $(4,4,1)^T$ & - & - & - & - & \\\cline{1-8}
		\multicolumn{2}{|c|}{$|O^+(NL+\vec h_i)|$} & 4 & 4 & - & - & - & - & $|O^+(L)|=8$ \\ \cline{1-8}
		\multicolumn{3}{|c|}{$ X_{i,2} \in G_{L,64,8}, \;X_{i,2} \vec h_0 \equiv \vec h_i$} & \rule{0pt}{14pt} $\left(\begin{smallmatrix} 3&6&4 \\2&5&4 \\5&5&3 \end{smallmatrix} \right)$ & - & - & - & - & \small{$|G_2|=512, |H_2| = 128$}\\[3pt] \hline\pagebreak   
		
		$(1,1,8)^T$ & 2 & $(1,1,0)^T$ & $(1,1,1)^T$ & - & - & - & - & \\\cline{1-8}
		\multicolumn{2}{|c|}{$|O^+(NL+\vec h_i)|$} & 8 & 8 & - & - & - & - & $|O^+(L)|=8$ \\ \cline{1-8}
		\multicolumn{3}{|c|}{$ X_{i,2} \in G_{L,64,2}, \;X_{i,2} \vec h_0 \equiv \vec h_i$} & \rule{0pt}{14pt} $\left(\begin{smallmatrix} 1&0&0 \\0&1&0 \\1&1&1 \end{smallmatrix} \right)$ & - & - & - & - & \small{$|G_2|=8, |H_2| = 4$}\\[3pt] \hline\hline  
		
		$(1,1,8)^T$ & 4 & $(1,1,1)^T$ & $(1,1,0)^T$ & $(1,1,2)^T$ & - & - & - & \\\cline{1-8}
		\multicolumn{2}{|c|}{$|O^+(NL+\vec h_i)|$} & 1 & 2 & 2 & - & - & - & $|O^+(L)|=8$ \\ \cline{1-8}
		\multicolumn{3}{|c|}{$ X_{i,2} \in G_{L,64,4}, \;X_{i,2} \vec h_0 \equiv \vec h_i$} & \rule{0pt}{14pt} $ \left(\begin{smallmatrix} 1&0&0\\ 0&1&0 \\ 0&3&1 \end{smallmatrix} \right)$ & $ \left(\begin{smallmatrix} 1&0&0\\ 0&1&0 \\ 1&0&1 \end{smallmatrix} \right)$ & - & - & - & \small{$|G_2|=128, |H_2| = 8$} \\[3pt] \hline
		\hline 
		
		$(1,1,8)^T$ & 8 & $(0,0,1)^T$ & $(0,0,3)^T$ & - & - & - & - & \\\cline{1-8}
		\multicolumn{2}{|c|}{$|O^+(NL+\vec h_i)|$} & 4 & 4 & - & - & - & - & $|O^+(L)|=8$ \\ \cline{1-8}
		\multicolumn{3}{|c|}{$ X_{i,2} \in G_{L,128,8}, \;X_{i,2} \vec h_0 \equiv \vec h_i$} &\rule{0pt}{14pt}  $ \left(\begin{smallmatrix} 1&0&0\\ 0&3&0 \\ 0&1&3 \end{smallmatrix} \right)$ & - & - & - & - & \small{$|G_2|=1024, |H_2| = 256$} \\[3pt] \hline \hline   
		
		$(1,2,2)^T$ & 4 & $(1,0,2)^T$ & $(1,0,0)^T$ & $(1,2,2)^T$ & - & - & - & \\\cline{1-8}
		\multicolumn{2}{|c|}{$|O^+(NL+\vec h_i)|$} & 2 & 4 & 4 & - & - & - & $|O^+(L)|=8$ \\ \cline{1-8}
		\multicolumn{3}{|c|}{$ X_{i,2} \in G_{L,16,4}, \;X_{i,2} \vec h_0 \equiv \vec h_i$} &\rule{0pt}{14pt}  $ \left(\begin{smallmatrix} 1&0&0\\ 2&1&0 \\ 2&0&1 \end{smallmatrix} \right)$ & $ \left(\begin{smallmatrix} 1&0&0\\ 2&1&0 \\ 0&0&1 \end{smallmatrix} \right)$ & - & - & - & \small{$|G_2|=32, |H_2| = 4$} \\[3pt] \hline\hline   
		
		$(1,2,2)^T$ & 8 & $(2,1,1)^T$ & $(2,3,3)^T$ & - & - & - & - & \\\cline{1-8}
		\multicolumn{2}{|c|}{$|O^+(NL+\vec h_i)|$} & 1 & 1 & - & - & - & - & $|O^+(L)|=8$ \\ \cline{1-8}
		\multicolumn{3}{|c|}{$ X_{i,2} \in G_{L,32,8}, \;X_{i,2} \vec h_0 \equiv \vec h_i$} & \rule{0pt}{14pt} $ \left( \begin{smallmatrix} 1&0&0\\ 0&7&4\\ 0&4&7 \end{smallmatrix} \right)$ & - & - & - & - & \small{$|G_2|=256, |H_2| = 16$} \\[3pt] \hline\hline 	
		
		$(1,2,2)^T$ & 8 & $(4,1,3)^T$ & $(0,1,1)^T$ & $(0,3,3)^T$ & - & - & - & \\\cline{1-8}
		\multicolumn{2}{|c|}{$|O^+(NL+\vec h_i)|$} & 1 & 2 & 2 & - & - & - & $|O^+(L)|=8$ \\ \cline{1-8}
		\multicolumn{3}{|c|}{$ X_{i,2} \in G_{L,32,8}, \;X_{i,2} \vec h_0 \equiv \vec h_i$} & \rule{0pt}{14pt} $ \left(\begin{smallmatrix} 3&0&4\\ 0&1&0 \\ 2&0&3 \end{smallmatrix} \right)$ & $ \left(\begin{smallmatrix} 3&0&4\\ 0&7&4 \\ 2&4&5 \end{smallmatrix} \right)$ & - & - & - & \small{$|G_2|=256, |H_2| = 16$} \\[3pt] \hline\hline 
		
		$(1,2,4)^T$ & 8 & $(0,1,2)^T$ & $(4,1,0)^T$ & $(4,1,4)^T$ & $(0,3,0)^T$ & $(0,3,4)^T$ & $(4,3,2)^T$ & \\\cline{1-8}
		\multicolumn{2}{|c|}{$|O^+(NL+\vec h_i)|$} & 1 & 2 & 2 & 2 & 2 & 1 & $|O^+(L)|=4$ \\ \cline{1-8}
		\multicolumn{3}{|c|}{$ X_{i,2} \in G_{L,64,8}, \;X_{i,2} \vec h_0 \equiv \vec h_i$} & \rule{0pt}{14pt} $ \left(\begin{smallmatrix} 3&4&0\\ 2&1&4 \\ 0&2&3 \end{smallmatrix} \right)$ & $ \left(\begin{smallmatrix} 3&4&0\\ 2&1&4 \\ 0&6&3 \end{smallmatrix} \right)$ & $ \left(\begin{smallmatrix} 1&0&0\\ 0&3&4 \\ 0&2&3 \end{smallmatrix} \right)$ & $ \left(\begin{smallmatrix} 1&0&0\\ 0&3&4 \\ 0&6&3 \end{smallmatrix} \right)$ & $ \left(\begin{smallmatrix} 3&4&0\\ 6&3&0 \\ 0&0&1 \end{smallmatrix} \right)$ & \small{$|G_2|=256, |H_2| = 16$} \\[3pt] \hline\hline 
		
		$(1,2,4)^T$ & 8 & $(0,3,2)^T$ & $(4,3,0)^T$ & $(4,3,4)^T$ & $(0,1,0)^T$ & $(0,1,4)^T$ & $(4,1,2)^T$ & \\\cline{1-8}
		\multicolumn{2}{|c|}{$|O^+(NL+\vec h_i)|$} & 1 & 2 & 2 & 2 & 2 & 1 & $|O^+(L)|=4$ \\ \cline{1-8}
		\multicolumn{3}{|c|}{$ X_{i,2} \in G_{L,64,8}, \;X_{i,2} \vec h_0 \equiv \vec h_i$} & \rule{0pt}{14pt} $ \left(\begin{smallmatrix} 3&4&0\\ 2&1&4 \\ 0&2&3 \end{smallmatrix} \right)$ & $ \left(\begin{smallmatrix} 3&4&0\\ 2&1&4 \\ 0&6&3 \end{smallmatrix} \right)$ & $ \left(\begin{smallmatrix} 1&0&0\\ 0&3&4 \\ 0&2&3 \end{smallmatrix} \right)$ & $ \left(\begin{smallmatrix} 1&0&0\\ 0&3&4 \\ 0&6&3 \end{smallmatrix} \right)$ & $ \left(\begin{smallmatrix} 3&4&0\\ 6&3&0 \\ 0&0&1 \end{smallmatrix} \right)$ & \small{$|G_2|=256, |H_2| = 16$} \\[3pt] \hline\hline 
		
		$(1,2,4)^T$ & 8 & $(2,2,1)^T$ & $(2,2,3)^T$ & $(2,2,5)^T$ & $(2,2,7)^T$ & - & - & \\\cline{1-8}
		\multicolumn{2}{|c|}{$|O^+(NL+\vec h_i)|$} & 1 & 1 & 1 & 1 & - & - & $|O^+(L)|=4$ \\ \cline{1-8}
		\multicolumn{3}{|c|}{$ X_{i,2} \in G_{L,64,8}, \;X_{i,2} \vec h_0 \equiv \vec h_i$} & \rule{0pt}{14pt} $ \left(\begin{smallmatrix} 1&0&0\\ 0&3&4 \\ 2&2&3 \end{smallmatrix} \right)$ & $ \left(\begin{smallmatrix} 1&0&0\\ 0&1&0 \\ 2&4&1 \end{smallmatrix} \right)$ & $ \left(\begin{smallmatrix} 1&0&0\\ 0&3&4 \\ 4&2&3 \end{smallmatrix} \right)$ & - & - & \small{$|G_2|=256, |H_2| = 16$} \\[3pt] \hline\hline 
		
		$(1,2,8)^T$ & 4 & $(0,1,0)^T$ & $(0,1,2)^T$ & - & - & - & - & \\\cline{1-8}
		\multicolumn{2}{|c|}{$|O^+(NL+\vec h_i)|$} & 2 & 2 & - & - & - & - & $|O^+(L)|=4$ \\ \cline{1-8}
		\multicolumn{3}{|c|}{$ X_{i,2} \in G_{L,64,4}, \;X_{i,2} \vec h_0 \equiv \vec h_i$} & \rule{0pt}{14pt} $\left(\begin{smallmatrix} 1&0&0 \\0&1&0 \\0&2&1 \end{smallmatrix} \right)$ & - & - & - & - & \small{$|G_2|=64, |H_2| = 16$}\\[3pt] \hline\hline  
		
		$(1,2,8)^T$ & 4 & $(1,0,1)^T$ & $(1,2,0)^T$ & $(1,2,2)^T$ & $(1,0,0)^T$ & $(1,0,2)^T$ & $(1,2,1)^T$ & \\\cline{1-8}
		\multicolumn{2}{|c|}{$|O^+(NL+\vec h_i)|$} & 1 & 2 & 2 & 2 & 2 & 1 & $|O^+(L)|=4$ \\ \cline{1-8}
		\multicolumn{3}{|c|}{$ X_{i,2} \in G_{L,64,4}, \;X_{i,2} \vec h_0 \equiv \vec h_i$} & \rule{0pt}{14pt} $ \left(\begin{smallmatrix} 1&0&0\\ 2&3&0 \\ 1&2&3 \end{smallmatrix} \right)$ & $ \left(\begin{smallmatrix} 1&0&0\\ 2&3&0 \\ 3&0&3 \end{smallmatrix} \right)$ & $ \left(\begin{smallmatrix} 1&0&0\\ 0&3&0 \\ 1&0&3 \end{smallmatrix} \right)$ & $ \left(\begin{smallmatrix} 1&0&0\\ 0&1&0 \\ 1&2&1 \end{smallmatrix} \right)$ & $ \left(\begin{smallmatrix} 1&0&0\\ 2&3&0 \\ 2&2&3 \end{smallmatrix} \right)$ & \small{$|G_2|=64, |H_2| = 4$} \\[3pt] \hline\hline 
		
		$(1,2,8)^T$ & 8 & $(4,2,1)^T$ & $(4,2,3)^T$ & - & - & - & - & \\\cline{1-8}
		\multicolumn{2}{|c|}{$|O^+(NL+\vec h_i)|$} & 1 & 1 & - & - & - & - & $|O^+(L)|=4$ \\ \cline{1-8}
		\multicolumn{3}{|c|}{$ X_{i,2} \in G_{L,128,8}, \;X_{i,2} \vec h_0 \equiv \vec h_i$} & \rule{0pt}{14pt} $\left(\begin{smallmatrix} 3&0&0 \\0&1&0 \\1&2&3 \end{smallmatrix} \right)$ & - & - & - & - & \small{$|G_2|=512, |H_2| = 64$}\\[3pt] \hline\hline  
		
		$(1,2,16)^T$ & 2 & $(1,0,0)^T$ & $(1,0,1)^T$ & - & - & - & - & \\\cline{1-8}
		\multicolumn{2}{|c|}{$|O^+(NL+\vec h_i)|$} & 4 & 4 & - & - & - & - & $|O^+(L)|=4$ \\ \cline{1-8}
		\multicolumn{3}{|c|}{$ X_{i,2} \in G_{L,128,2}, \;X_{i,2} \vec h_0 \equiv \vec h_i$} & \rule{0pt}{14pt} $\left(\begin{smallmatrix} 1&0&0 \\0&1&0 \\1&0&1 \end{smallmatrix} \right)$ & - & - & - & - & \small{$|G_2|=4, |H_2| = 2$}\\[3pt] \hline\pagebreak   
		
		$(1,3,3)^T$ & 6 & $(0,0,1)^T$ & $(0,2,3)^T$ & - & - & - & - & \\\cline{1-8}
		\multicolumn{2}{|c|}{$|O^+(NL+\vec h_i)|$} & 2 & 2 & - & - & - & - & $|O^+(L)|=8$ \\ \cline{1-8}
		\multicolumn{3}{|c|}{$ X_{i,2} \in G_{L,8,2}, \;X_{i,2} \vec h_0 \equiv \vec h_i$} & \rule{0pt}{14pt}$ \left(\begin{smallmatrix} 1&0&0\\ 0&1&0 \\ 0&0&1 \end{smallmatrix} \right)$ & - & - & - & - & \small{$|G_2|=2, |H_2| = 1$}\\[3pt] \cline{1-8}  
		\multicolumn{3}{|c|}{$ X_{i,3} \in G_{L,9,3}, \;X_{i,3} \vec h_0 \equiv \vec h_i$} & \rule{0pt}{14pt} $ \left(\begin{smallmatrix} 1&0&0\\ 1&0&2 \\ 0&1&0 \end{smallmatrix} \right)$ & - & - & - & - & \small{$|G_3|=72, |H_3| = 18$}\\[3pt] \hline\hline  
		
		$(1,3,3)^T$ & 12 & $(3,3,4)^T$ & $(3,3,8)^T$ & $(3,0,1)^T$ & $(3,0,5)^T$ & - & - & \\\cline{1-8}
		\multicolumn{2}{|c|}{$|O^+(NL+\vec h_i)|$} & 1 & 1 & 1 & 1 & - & - & $|O^+(L)|=8$ \\ \cline{1-8}
		\multicolumn{3}{|c|}{$ X_{i,2} \in G_{L,8,4}, \;X_{i,2} \vec h_0 \equiv \vec h_i$} &\rule{0pt}{14pt} $ \left(\begin{smallmatrix} 1&0&0\\ 0&1&0 \\ 0&0&1 \end{smallmatrix} \right)$ & $ \left(\begin{smallmatrix} 1&0&0\\ 0&0&1 \\ 0&3&0 \end{smallmatrix} \right)$ & $ \left(\begin{smallmatrix} 1&0&0\\ 0&0&1 \\ 0&3&0 \end{smallmatrix} \right)$ & - & - & \small{$|G_2|=16, |H_2| = 2$}\\[3pt] \cline{1-8}  
		\multicolumn{3}{|c|}{$ X_{i,3} \in G_{L,9,3}, \;X_{i,3} \vec h_0 \equiv \vec h_i$} &\rule{0pt}{14pt} $ \left(\begin{smallmatrix} 1&0&0\\ 0&2&0 \\ 0&0&2 \end{smallmatrix} \right)$ & $ \left(\begin{smallmatrix} 1&0&0\\ 0&1&0 \\ 0&0&1 \end{smallmatrix} \right)$ & $ \left(\begin{smallmatrix} 1&0&0\\ 0&2&0 \\ 0&0&2 \end{smallmatrix} \right)$ & - & - & \small{$|G_3|=72, |H_3| = 18$}\\[3pt] \hline\hline  
		
		$(1,3,9)^T$ & 4 & $(1,1,0)^T$ & $(0,1,1)^T$ & - & - & - & - & \\\cline{1-8}
		\multicolumn{2}{|c|}{$|O^+(NL+\vec h_i)|$} & 2 & 2 & - & - & - & - & $|O^+(L)|=4$ \\ \cline{1-8}
		\multicolumn{3}{|c|}{$ X_{i,2} \in G_{L,8,2}, \;X_{i,2} \vec h_0 \equiv \vec h_i$} & \rule{0pt}{14pt} $\left(\begin{smallmatrix} 0&0&1 \\0&1&0 \\3&0&0 \end{smallmatrix} \right)$ & - & - & - & - & \small{$|G_2|=16, |H_2| = 2$}\\[3pt] \hline   \hline  
		
		$(1,3,9)^T$ & 6 & $(0,0,1)^T$ & $(3,0,2)^T$ & - & - & - & - & \\\cline{1-8}
		\multicolumn{2}{|c|}{$|O^+(NL+\vec h_i)|$} & 2 & 2 & - & - & - & - & $|O^+(L)|=4$ \\ \cline{1-8}
		\multicolumn{3}{|c|}{$ X_{i,2} \in G_{L,8,2}, \;X_{i,2} \vec h_0 \equiv \vec h_i$} & \rule{0pt}{14pt} $ \left(\begin{smallmatrix} 0&0&1\\ 0&1&0 \\ 1&0&0 \end{smallmatrix} \right)$ & - & - & - & - & \small{$|G_2|=2, |H_2| = 1$}\\[3pt] \cline{1-8}  
		\multicolumn{3}{|c|}{$ X_{i,3} \in G_{L,27,3}, \;X_{i,3} \vec h_0 \equiv \vec h_i$} & \rule{0pt}{14pt} $ \left(\begin{smallmatrix} 1&0&0\\ 0&2&0 \\ 0&0&2 \end{smallmatrix} \right)$ & - & - & - & - & \small{$|G_3|=108, |H_3| = 54$}\\[3pt] \hline\hline   
		
		$(1,3,12)^T$ & 3 & $(0,1,0)^T$ & $(0,0,1)^T$ & - & - & - & - & \\\cline{1-8}
		\multicolumn{2}{|c|}{$|O^+(NL+\vec h_i)|$} & 2 & 2 & - & - & - & - & $|O^+(L)|=4$ \\ \cline{1-8}
		\multicolumn{3}{|c|}{$ X_{i,3} \in G_{L,9,3}, \;X_{i,3} \vec h_0 \equiv \vec h_i$} & \rule{0pt}{14pt} $\left(\begin{smallmatrix} 1&0&0 \\0&0&1 \\0&2&0 \end{smallmatrix} \right)$ & - & - & - & - & \small{$|G_3|=72, |H_3| = 18$}\\[3pt] \hline\hline  
		
		$(1,3,12)^T$ & 6 & $(3,1,0)^T$ & $(3,3,2)^T$ & - & - & - & - & \\\cline{1-8}
		\multicolumn{2}{|c|}{$|O^+(NL+\vec h_i)|$} & 2 & 2 & - & - & - & - & $|O^+(L)|=4$ \\ \cline{1-8}
		\multicolumn{3}{|c|}{$ X_{i,2} \in G_{L,32,2}, \;X_{i,2} \vec h_0 \equiv \vec h_i$} & \rule{0pt}{14pt} $ \left(\begin{smallmatrix} 1&0&0\\ 0&1&0 \\ 0&0&1 \end{smallmatrix} \right)$ & - & - & - & - & \small{$|G_2|=2, |H_2| = 2$}\\[3pt] \cline{1-8}  
		\multicolumn{3}{|c|}{$ X_{i,3} \in G_{L,9,3}, \;X_{i,3} \vec h_0 \equiv \vec h_i$} & \rule{0pt}{14pt} $ \left(\begin{smallmatrix} 1&0&0\\ 0&0&1 \\ 0&2&0 \end{smallmatrix} \right)$ & - & - & - & - & \small{$|G_3|=72, |H_3| = 18$}\\[3pt] \hline\hline 
		
		$(1,3,12)^T$ & 12 & $(0,4,3)^T$ & $(0,0,1)^T$ & $(0,0,5)^T$ & - & - & - & \\\cline{1-8}
		\multicolumn{2}{|c|}{$|O^+(NL+\vec h_i)|$} & 1 & 2 & 2 & - & - & - & $|O^+(L)|=4$ \\ \cline{1-8}
		\multicolumn{3}{|c|}{$ X_{i,2} \in G_{L,32,4}, \;X_{i,2} \vec h_0 \equiv \vec h_i$} & \rule{0pt}{14pt} $ \left(\begin{smallmatrix} 1&0&0\\ 0&3&0 \\ 0&0&3 \end{smallmatrix} \right)$ & $ \left(\begin{smallmatrix} 1&0&0\\ 0&3&0 \\ 0&0&3 \end{smallmatrix} \right)$ & - & - & - & \small{$|G_2|=32, |H_2| = 16$}\\[3pt] \cline{1-8}  
		\multicolumn{3}{|c|}{$ X_{i,3} \in G_{L,9,3}, \;X_{i,3} \vec h_0 \equiv \vec h_i$} & \rule{0pt}{14pt} $ \left(\begin{smallmatrix} 1&0&0\\ 0&0&2 \\ 0&1&0 \end{smallmatrix} \right)$ & $ \left(\begin{smallmatrix} 1&0&0\\ 0&0&1 \\ 0&2&0 \end{smallmatrix} \right)$ & - & - & - & \small{$|G_3|=72, |H_3| = 18$}\\[3pt] \hline\hline 
		
		$(1,4,4)^T$ & 2 & $(1,0,0)^T$ & $(1,1,1)^T$ & - & - & - & - & \\\cline{1-8}
		\multicolumn{2}{|c|}{$|O^+(NL+\vec h_i)|$} & 8 & 8 & - & - & - & - & $|O^+(L)|=8$ \\ \cline{1-8}
		\multicolumn{3}{|c|}{$ X_{i,2} \in G_{L,32,2}, \;X_{i,2} \vec h_0 \equiv \vec h_i$} & \rule{0pt}{14pt} $ \left(\begin{smallmatrix} 1&0&0\\ 1&1&0 \\ 1&0&1 \end{smallmatrix} \right)$ & - & - & - & - & \small{$|G_2|=4, |H_2| = 2$} \\[3pt] \hline\hline 
		
		$(1,4,4)^T$ & 4 & $(0,0,1)^T$ & $(0,2,1)^T$ & - & - & - & - & \\\cline{1-8}
		\multicolumn{2}{|c|}{$|O^+(NL+\vec h_i)|$} & 2 & 2 & - & - & - & - & $|O^+(L)|=8$ \\ \cline{1-8}
		\multicolumn{3}{|c|}{$ X_{i,2} \in G_{L,32,4}, \;X_{i,2} \vec h_0 \equiv \vec h_i$} &\rule{0pt}{14pt} $ \left(\begin{smallmatrix} 1&0&0\\ 1&1&2 \\ 1&2&1 \end{smallmatrix} \right)$ & - & - & - & - & \small{$|G_2|=64, |H_2| = 8$} \\[3pt] \hline\hline 
		
		$(1,4,4)^T$ & 4 & $(1,1,1)^T$ & $(1,0,0)^T$ & $(1,2,0)^T$ & $(1,2,2)^T$ & - & - & \\\cline{1-8}
		\multicolumn{2}{|c|}{$|O^+(NL+\vec h_i)|$} & 1 & 4 & 2 & 4 & - & - & $|O^+(L)|=8$ \\ \cline{1-8}
		\multicolumn{3}{|c|}{$ X_{i,2} \in G_{L,32,4}, \;X_{i,2} \vec h_0 \equiv \vec h_i$} & \rule{0pt}{14pt} $ \left(\begin{smallmatrix} 1&0&0\\ 1&1&2 \\ 1&2&1 \end{smallmatrix} \right)$ & $ \left(\begin{smallmatrix} 1&0&0\\ 1&3&2 \\ 3&2&3 \end{smallmatrix} \right)$ & $ \left(\begin{smallmatrix} 1&0&0\\ 1&2&3 \\ 3&1&2 \end{smallmatrix} \right)$ & - & - & \small{$|G_2|=64, |H_2| = 4$} \\[3pt] \hline\hline 
		
		$(1,4,4)^T$ & 8 & $(4,1,2)^T$ & $(0,3,0)^T$ & $(0,3,4)^T$ & - & - & - & \\\cline{1-8}
		\multicolumn{2}{|c|}{$|O^+(NL+\vec h_i)|$} & 1 & 2 & 2 & - & - & - & $|O^+(L)|=8$ \\ \cline{1-8}
		\multicolumn{3}{|c|}{$ X_{i,2} \in G_{L,64,8}, \;X_{i,2} \vec h_0 \equiv \vec h_i$} & \rule{0pt}{14pt} $ \left(\begin{smallmatrix} 3&4&4\\ 1&3&6 \\ 1&2&5 \end{smallmatrix} \right)$ & $ \left(\begin{smallmatrix} 3&4&4\\ 1&3&2 \\ 1&6&5 \end{smallmatrix} \right)$ & - & - & - & \small{$|G_2|=512, |H_2| = 32$} \\[3pt] \hline\pagebreak  
		
		$(1,4,4)^T$ & 8 & $(4,3,2)^T$ & $(0,1,0)^T$ & $(0,1,4)^T$ & - & - & - & \\\cline{1-8}
		\multicolumn{2}{|c|}{$|O^+(NL+\vec h_i)|$} & 1 & 2 & 2 & - & - & - & $|O^+(L)|=8$ \\ \cline{1-8}
		\multicolumn{3}{|c|}{$ X_{i,2} \in G_{L,64,8}, \;X_{i,2} \vec h_0 \equiv \vec h_i$} & \rule{0pt}{14pt} $ \left(\begin{smallmatrix} 3&4&4\\ 1&3&6 \\ 1&2&5 \end{smallmatrix} \right)$ & $ \left(\begin{smallmatrix} 3&4&4\\ 1&3&2 \\ 1&6&5 \end{smallmatrix} \right)$ & - & - & - & \small{$|G_2|=512, |H_2| = 32$} \\[3pt] \hline\hline 
		
		$(1,4,8)^T$ & 4 & $(2,1,0)^T$ & $(2,1,2)^T$ & - & - & - & - & \\\cline{1-8}
		\multicolumn{2}{|c|}{$|O^+(NL+\vec h_i)|$} & 2 & 2 & - & - & - & - & $|O^+(L)|=4$ \\ \cline{1-8}
		\multicolumn{3}{|c|}{$ X_{i,2} \in G_{L,64,4}, \;X_{i,2} \vec h_0 \equiv \vec h_i$} &\rule{0pt}{14pt} $ \left(\begin{smallmatrix} 1&0&0\\ 0&1&0 \\ 0&2&1 \end{smallmatrix} \right)$ & - & - & - & - & \small{$|G_2|=64, |H_2| = 16$} \\[3pt] \hline\hline 
		
		$(1,4,12)^T$ & 12 & $(6,0,1)^T$ & $(6,0,5)^T$ & - & - & - & - & \\\cline{1-8}
		\multicolumn{2}{|c|}{$|O^+(NL+\vec h_i)|$} & 2 & 2 & - & - & - & - & $|O^+(L)|=4$ \\ \cline{1-8}
		\multicolumn{3}{|c|}{$ X_{i,2} \in G_{L,32,4}, \;X_{i,2} \vec h_0 \equiv \vec h_i$} &\rule{0pt}{14pt} $ \left(\begin{smallmatrix} 1&0&0\\ 0&1&0 \\ 0&0&1 \end{smallmatrix} \right)$ & - & - & - & - & \small{$|G_2|=32, |H_2| = 16$}\\[3pt] \cline{1-8}  
		\multicolumn{3}{|c|}{$ X_{i,3} \in G_{L,9,3}, \;X_{i,3} \vec h_0 \equiv \vec h_i$} &\rule{0pt}{14pt} $ \left(\begin{smallmatrix} 1&0&0\\ 0&2&0 \\ 0&0&2 \end{smallmatrix} \right)$ & - & - & - & - & \small{$|G_3|=72, |H_3| = 36$}\\[3pt] \hline\hline 
		
		$(1,6,6)^T$ & 12 & $(0,1,5)^T$ & $(0,1,1)^T$ & $(0,5,5)^T$ & - & - & - & \\\cline{1-8}
		\multicolumn{2}{|c|}{$|O^+(NL+\vec h_i)|$} & 1 & 2 & 2 & - & - & - & $|O^+(L)|=8$ \\ \cline{1-8}
		\multicolumn{3}{|c|}{$ X_{i,2} \in G_{L,16,4}, \;X_{i,2} \vec h_0 \equiv \vec h_i$} &\rule{0pt}{14pt} $ \left(\begin{smallmatrix} 1&0&0\\ 0&1&0 \\ 0&0&1 \end{smallmatrix} \right)$ & $ \left(\begin{smallmatrix} 1&0&0\\ 0&1&0 \\ 0&0&1 \end{smallmatrix} \right)$ & - & - & - & \small{$|G_2|=32, |H_2| = 8$}\\[3pt] \cline{1-8}  
		\multicolumn{3}{|c|}{$ X_{i,3} \in G_{L,9,3}, \;X_{i,3} \vec h_0 \equiv \vec h_i$} &\rule{0pt}{14pt} $ \left(\begin{smallmatrix} 1&0&0\\ 0&0&2 \\ 0&1&0 \end{smallmatrix} \right)$ & $ \left(\begin{smallmatrix} 1&0&0\\ 0&0&1 \\ 0&2&0 \end{smallmatrix} \right)$ & - & - & - & \small{$|G_3|=72, |H_3| = 18$}\\[3pt] \hline \hline   
		
		$(1,8,8)^T$ & 2 & $(1,0,1)^T$ & $(1,0,0)^T$ & $(1,1,1)^T$ & - & - & - & \\\cline{1-8}
		\multicolumn{2}{|c|}{$|O^+(NL+\vec h_i)|$} & 4 & 8 & 8 & - & - & - & $|O^+(L)|=8$ \\ \cline{1-8}
		\multicolumn{3}{|c|}{$ X_{i,2} \in G_{L,64,2}, \;X_{i,2} \vec h_0 \equiv \vec h_i$} & \rule{0pt}{14pt}$ \left(\begin{smallmatrix} 1&0&0\\ 0&1&0 \\ 1&0&1 \end{smallmatrix} \right)$ & $ \left(\begin{smallmatrix} 1&0&0\\ 1&1&0 \\ 0&0&1 \end{smallmatrix} \right)$ & - & - & - & \small{$|G_2|=8, |H_2| = 2$} \\[3pt] \hline \hline 
		
		$(1,8,8)^T$ & 8 & $(0,1,3)^T$ & $(0,1,1)^T$ & $(0,3,3)^T$ & - & - & - & \\\cline{1-8}
		\multicolumn{2}{|c|}{$|O^+(NL+\vec h_i)|$} & 1 & 2 & 2 & - & - & - & $|O^+(L)|=8$ \\ \cline{1-8}
		\multicolumn{3}{|c|}{$ X_{i,2} \in G_{L,128,8}, \;X_{i,2} \vec h_0 \equiv \vec h_i$} &\rule{0pt}{14pt} $ \left(\begin{smallmatrix} 3&0&0\\ 0&1&0 \\ 1&0&3 \end{smallmatrix} \right)$ & $ \left(\begin{smallmatrix} 3&0&0\\ 0&7&4 \\ 1&4&5 \end{smallmatrix} \right)$ & - & - & - & \small{$|G_2|=1024, |H_2| = 64$} \\[3pt] \hline\hline  
		
		$(1,8,8)^T$ & 8 & $(4,1,1)^T$ & $(4,3,3)^T$ & - & - & - & - & \\\cline{1-8}
		\multicolumn{2}{|c|}{$|O^+(NL+\vec h_i)|$} & 2 & 2 & - & - & - & - & $|O^+(L)|=8$ \\ \cline{1-8}
		\multicolumn{3}{|c|}{$ X_{i,2} \in G_{L,128,8}, \;X_{i,2} \vec h_0 \equiv \vec h_i$} &\rule{0pt}{14pt} $ \left(\begin{smallmatrix} 1&0&0\\ 5&3&4 \\ 3&4&3 \end{smallmatrix} \right)$ & - & - & - & - & \small{$|G_2|=1024, |H_2| = 128$} \\[3pt] \hline\hline 
		
		$(1,12,12)^T$ & 12 & $(6,3,4)^T$ & $(6,0,1)^T$ & $(6,0,5)^T$ & - & - & - & \\\cline{1-8}
		\multicolumn{2}{|c|}{$|O^+(NL+\vec h_i)|$} & 1 & 2 & 2 & - & - & - & $|O^+(L)|=8$ \\ \cline{1-8}
		\multicolumn{3}{|c|}{$ X_{i,2} \in G_{L,32,4}, \;X_{i,2} \vec h_0 \equiv \vec h_i$} &\rule{0pt}{14pt} $ \left(\begin{smallmatrix} 1&0&0\\ 0&0&1 \\ 0&3&0 \end{smallmatrix} \right)$ & $ \left(\begin{smallmatrix} 1&0&0\\ 0&0&1 \\ 0&3&0 \end{smallmatrix} \right)$ & - & - & - & \small{$|G_2|=64, |H_2| = 16$}\\[3pt] \cline{1-8}  
		\multicolumn{3}{|c|}{$ X_{i,3} \in G_{L,9,3}, \;X_{i,3} \vec h_0 \equiv \vec h_i$} &\rule{0pt}{14pt} $ \left(\begin{smallmatrix} 1&0&0\\ 0&1&0 \\ 0&0&1 \end{smallmatrix} \right)$ & $ \left(\begin{smallmatrix} 1&0&0\\ 0&0&2 \\ 0&1&0 \end{smallmatrix} \right)$ & - & - & - & \small{$|G_3|=72, |H_3| = 18$}\\[3pt] \hline\hline
		
		$(3,4,12)^T$ & 3 & $(1,0,0)^T$ & $(0,0,1)^T$ & - & - & - & - & \\\cline{1-8}
		\multicolumn{2}{|c|}{$|O^+(NL+\vec h_i)|$} & 2 & 2 & - & - & - & - & $|O^+(L)|=4$ \\ \cline{1-8}
		\multicolumn{3}{|c|}{$ X_{i,3} \in G_{L,9,3}, \;X_{i,3} \vec h_0 \equiv \vec h_i$} &\rule{0pt}{14pt} $\left(\begin{smallmatrix} 0&0&1 \\0&2&0 \\1&0&0 \end{smallmatrix} \right)$ & - & - & - & - & \small{$|G_3|=72, |H_3| = 18$}\\[3pt] \hline\hline  
		
		$(3,4,12)^T$ & 6 & $(0,0,1)^T$ & $(2,0,3)^T$ & - & - & - & - & \\\cline{1-8}
		\multicolumn{2}{|c|}{$|O^+(NL+\vec h_i)|$} & 2 & 2 & - & - & - & - & $|O^+(L)|=4$ \\ \cline{1-8}
		\multicolumn{3}{|c|}{$ X_{i,2} \in G_{L,32,2}, \;X_{i,2} \vec h_0 \equiv \vec h_i$} &\rule{0pt}{14pt} $ \left(\begin{smallmatrix} 1&0&0\\ 0&1&0 \\ 0&0&1 \end{smallmatrix} \right)$ & - & - & - & - & \small{$|G_2|=2, |H_2| = 2$}\\[3pt] \cline{1-8}  
		\multicolumn{3}{|c|}{$ X_{i,3} \in G_{L,9,3}, \;X_{i,3} \vec h_0 \equiv \vec h_i$} & $ \left(\begin{smallmatrix} 0&0&2\\ 0&1&0 \\ 1&0&0 \end{smallmatrix} \right)$ & - & - & - & - & \small{$|G_3|=72, |H_3| = 18$}\\[3pt] \hline \hline
		
		$(3,4,12)^T$ & 12 & $(2,3,0)^T$ & $(6,3,4)^T$ & - & - & - & - & \\\cline{1-8}
		\multicolumn{2}{|c|}{$|O^+(NL+\vec h_i)|$} & 1 & 1 & - & - & - & - & $|O^+(L)|=4$ \\ \cline{1-8}
		\multicolumn{3}{|c|}{$ X_{i,2} \in G_{L,32,4}, \;X_{i,2} \vec h_0 \equiv \vec h_i$} & \rule{0pt}{14pt}$ \left(\begin{smallmatrix} 1&0&0\\ 0&1&0 \\ 0&0&1 \end{smallmatrix} \right)$ & - & - & - & - & \small{$|G_2|=32, |H_2| = 16$}\\[3pt] \cline{1-8}  
		\multicolumn{3}{|c|}{$ X_{i,3} \in G_{L,9,3}, \;X_{i,3} \vec h_0 \equiv \vec h_i$} & \rule{0pt}{14pt}$ \left(\begin{smallmatrix} 0&0&2\\ 0&1&0 \\ 1&0&0 \end{smallmatrix} \right)$ & - & - & - & - & \small{$|G_3|=72, |H_3| = 18$}\\[3pt] \hline \pagebreak
		
		$(3,4,36)^T$ & 4 & $(2,1,0)^T$ & $(2,0,1)^T$ & - & - & - & - & \\\cline{1-8}
		\multicolumn{2}{|c|}{$|O^+(NL+\vec h_i)|$} & 2 & 2 & - & - & - & - & $|O^+(L)|=4$ \\ \cline{1-8}
		\multicolumn{3}{|c|}{$ X_{i,2} \in G_{L,32,4}, \;X_{i,2} \vec h_0 \equiv \vec h_i$} &\rule{0pt}{14pt} $ \left(\begin{smallmatrix} 3&0&0\\ 0&0&1 \\ 0&1&0 \end{smallmatrix} \right)$ & - & - & - & - & \small{$|G_2|=64, |H_2| = 16$} \\[3pt] \hline 
	\end{longtable}
	\end{footnotesize}
	\end{center}
}
	\section{Upper Bound for Image of Spinor Norm Map}\label{sctn::upperbnd_img_spn_norm}
	With reference to Section \ref{sctn::classify_found_cosets}, the following table evaluates the group $\theta(O^+(M_p)) \cap \theta(O^+(L_p))$ for a lattice $L = L_{\vec a}$ and a prime $p$ dividing $N$ and $4a_1a_2a_3$; this then gives a sort of ``upper bound'' on the group $\theta(O^+(NL_p + \vec h))$. This latter group can be used to find the number of spinor genera in the genus (c.f. equation (\ref{eqn::num_spn_in_gen})). Here, the lattice $M$ is given by $M = \Z \vec h + NL$. The groups $\theta(O^+(M_p))$ and $\theta(O^+(L_p))$ are easily calculated using the algorithm given in Lemma \ref{lem::conway_algorithm}. 
	
	Since the number of spinor genera in the genus can be evaluated using any one coset in a genus, and since the genus has already been classified in Proposition \ref{prop::classify_genus_all_listed} and Appendix \ref{sctn::gen_info}, in the table below the calculation of $\theta(O^+(M_p))$ has been done for only one value of $\vec h$ in the genus. 
	Also, the Gram matrix of $M$ is given in an ordered $\Z$-basis $\{\vec h, *, *\}$ for $M$. For the image of the spinor norm map, only the square class representatives are given. For instance, in the first row, we actually have $\theta(O^+(M_2)) = (\Q_2^\times)^2 \cup 2(\Q_2^\times)^2 \cup 5(\Q_2^\times)^2 \cup 10(\Q_2^\times)^2$, and $\theta(O^+(L_2)) = \Q_2^\times$.
{	
	\renewcommand{\arraystretch}{1.6}
	\begin{small}
	\begin{longtable}[c]{|c|c|c|c|c|c|c|c|} 
		\caption[]{Evaluation of $\theta(O^+(M_p)) \cap \theta(O^+(L_p))$}\label{table::spn_info_upperbnd}\\
		\hline 
		$\vec a$ & $N$ & $\vec h$ & $M$ & $p$ & {$\theta(O^+(M_p))$} & {$\theta(O^+(L_p))$} & \small{$ \theta(O^+(M_p))\cap \theta(O^+(L_p))$}\\\hline 
		\endfirsthead		
		\hline 
		$\vec a$ & $N$ & $\vec h$ & $M$ & $p$ & {$\theta(O^+(M_p))$} & {$\theta(O^+(L_p))$} & \small{$ \theta(O^+(M_p))\cap \theta(O^+(L_p))$}\\\hline 
		\endhead
		
		$(1,1,1)^T$ & $4$ & $(0,0,1)^T$ & $\left(\begin{smallmatrix}1&0&0\\0&16&0\\0&0&16\end{smallmatrix}\right)$ & $2$ & $ 1,2,5,10 $ & {all} & $ 1,2,5,10 $ \\[2pt] \hline 
		
		$(1,1,1)^T$ & $8$ & $(1,2,2)^T$ & $\left(\begin{smallmatrix}9&4&0\\4&16&0\\0&0&32\end{smallmatrix}\right)$ & $2$ & $ 1,2,5,10 $ & {all} & $ 1,2,5,10 $ \\[2pt] \hline 
		
		$(1,1,1)^T$ & $8$ & $(3,2,2)^T$ & $\left(\begin{smallmatrix}17&12&0\\12&16&0\\0&0&32\end{smallmatrix}\right)$ & $2$ & $ 1,2,5,10 $ & {all} & $ 1,2,5,10 $ \\[2pt] \hline 
		
		$(1,1,2)^T$ & $4$ & $(1,1,0)^T$ & $\left(\begin{smallmatrix}2&-2&0\\-2&10&0\\0&0&2\end{smallmatrix}\right)$ & $2$  & {$1,5$} & all & $1,5$\\[2pt] \hline 
		
		$(1,1,2)^T$ & $8$ & $(0,4,1)^T$ & $\left(\begin{smallmatrix}18&0&4\\0&64&0\\4&0&8\end{smallmatrix}\right)$ & $2$  & {$1,2,5,10$} & all & $1,2,5,10$\\[2pt] \hline 
		
		$(1,1,2)^T$ & $8$ & $(0,4,3)^T$ & $\left(\begin{smallmatrix}34&0&12\\0&64&0\\12&0&8\end{smallmatrix}\right)$ & $2$  & {$1,2,5,10$} & all & $1,2,5,10$\\[2pt]\hline 
		
		$(1,1,2)^T$ & $8$ & $(1,3,2)^T$ & $\left(\begin{smallmatrix}18&40&24\\40&160&96\\24&96&64\end{smallmatrix}\right)$ & $2$  & {$1,5$} & all & $1,5$\\[2pt]\hline 
				
		$(1,1,4)^T$ & $2$ & $(0,1,0)^T$ & $\left(\begin{smallmatrix}1&0&0\\0&4&0\\0&0&16
		\end{smallmatrix}\right)$ & $2$ & $1,5$ & $1,2,5,10$& $1,5$\\[2pt]\hline 
		
		$(1,1,4)^T$ & $4$ & $(1,2,1)^T$ & $\left(\begin{smallmatrix}9&10&-6\\10&20&-12\\-6&-12&20
		\end{smallmatrix}\right)$ & $2$ & $1,5$ & $1,2,5,10$& $1,5$\\[2pt]\hline 
		
		$(1,1,4)^T$ & $8$ & $(0,0,1)^T$& $\left(\begin{smallmatrix}4&0&0\\0&64&0\\0&0&64
		\end{smallmatrix}\right)$ & $2$ & $1,2,5,10$ & $1,2,5,10$ & $1,2,5,10$\\[2pt]\hline 
		
		$(1,1,4)^T$ & $8$ & $(0,0,3)^T$& $\left(\begin{smallmatrix}36&0&0\\0&64&0\\0&0&64\end{smallmatrix}\right)$ & $2$  & $1,2,5,10$ & $1,2,5,10$ &  $1,2,5,10$\\[2pt]\hline 
		
		$(1,1,8)^T$ & $2$ & $(1,1,0)^T$& $\left(\begin{smallmatrix}2&0&0\\0&2&0\\0&0&32\end{smallmatrix}\right)$ & $2$ & $ 1,2,5,10 $ & all & $1,2,5,10$ \\[2pt]\hline 
		
		$(1,1,8)^T$ & $4$ & $(1,1,1)^T$& $\left(\begin{smallmatrix}10&6&-22\\6&18&-26\\-22&-26&74\end{smallmatrix}\right)$ & $2$ & $ 1,5 $ & all & $1,5$ \\[2pt]\hline 
		
		$(1,2,2)^T$ & $4$ & $(1,0,2)^T$ & $\left(\begin{smallmatrix}9&0&2\\0&32&0\\2&0&4\end{smallmatrix}\right)$ & $2$ & $1,2,5,10$ & all & $1,2,5,10$ \\[2pt]\hline 
%
%
		$(1,2,2)^T$ & $8$ & $(4,1,3)^T$ & $\left(\begin{smallmatrix}36&40&-8\\40&80&-16\\-8&-16&16\end{smallmatrix}\right)$ & $2$ & $1,2,5,10$ & all & $1,2,5,10$ \\[2pt]\hline
		
		$(1,2,4)^T$ & $8$ & $(0,1,2)^T$ & $\left(\begin{smallmatrix}18&0&8\\0&64&0\\8&0&32\end{smallmatrix}\right)$ & $2$ & $1,2,5,10$ & all & $1,2,5,10$ \\[2pt]\hline 
		
		$(1,2,4)^T$ & $8$ & $(0,3,2)^T$ & $\left(\begin{smallmatrix}34&0&24\\0&64&0\\24&0&32\end{smallmatrix}\right)$ & $2$ & $1,2,5,10$ & all & $1,2,5,10$ \\[2pt]\hline 

		$(1,2,8)^T$ & $4$ & $(0,1,0)^T$ & $\left(\begin{smallmatrix}2&0&0\\0&16&0\\0&0&128\end{smallmatrix}\right)$ & $2$ & $ 1,2,3,6  $& all & $1,2,3,6$ \\[2pt]\hline 
		
		$(1,2,8)^T$ & $4$ & $(1,0,1)^T$ & $\left(\begin{smallmatrix}9&0&-23\\0&32&0\\-23&0&73\end{smallmatrix}\right)$ & $2$ & $ 1,2,5,10$ & all & $1,2,5,10$ \\[2pt]\hline 
		
		$(1,2,16)^T$ & $2$ & $(1,0,0)^T$ & $\left(\begin{smallmatrix}1&0&0\\0&8&0\\0&0&64\end{smallmatrix}\right)$ & $2$ & $ 1,2,3,6 $ & all & $1,2,3,6$ \\[2pt]\hline 
		
		\multirow{2}{*}{$(1,3,3)^T$} & \multirow{2}{*}{$6$} & \multirow{2}{*}{$(0,0,1)^T$}& \multirow{2}{*}{ $\left(\begin{smallmatrix}3&0&0\\0&36&0\\0&0&108 \end{smallmatrix}\right)$} & $2$ & $ 1,3,5,7 $ & all & $ 1,3,5,7 $ \\\cline{5-8}
		& & & & $3$ & $ 1,3 $ & all & $1,3$ \\\hline 
		
		\multirow{2}{*}{$(1,3,3)^T$} & \multirow{2}{*}{$12$} & \multirow{2}{*}{$(3,3,4)^T$}& \multirow{2}{*}{ $\left(\begin{smallmatrix}84&108&72\\108&324& 216\\72&216&252 \end{smallmatrix}\right)$} & $2$ & $ 1,3,5,7 $ & all & $ 1,3,5,7 $ \\\cline{5-8}
		& & & & $3$ & $ 1,3 $ & all & $1,3$ \\\hline 
		
		\multirow{2}{*}{$(1,3,9)^T$} & \multirow{2}{*}{$4$} & \multirow{2}{*}{$(1,1,0)^T$} & \multirow{2}{*}{ $\left(\begin{smallmatrix}4&-8&0\\-8&28&0\\0&0&144 \end{smallmatrix}\right)$ } & $2$ & $1,3,5,7$ & all & $1,3,5,7$ \\\cline{5-8}
		& & & & $3$ & $1,3$ & $1,3$ & $1,3$ \\\hline
		
		\multirow{2}{*}{$(1,3,9)^T$} & \multirow{2}{*}{$6$} & \multirow{2}{*}{$(0,0,1)^T$} & \multirow{2}{*}{ $\left(\begin{smallmatrix}9&0&0\\0&36&0\\0&0&108 \end{smallmatrix}\right)$ } & $2$ & $1,3,5,7$ & all & $1,3,5,7$ \\\cline{5-8}
		& & & & $3$ & all & $1,3$ & $1,3$ \\\hline 
		
		\multirow{2}{*}{$(1,3,12)^T$} & \multirow{2}{*}{$3$} & \multirow{2}{*}{$(0,1,0)^T$} & \multirow{2}{*}{ $\left(\begin{smallmatrix}3&0&0\\0&9&0\\0&0&108 \end{smallmatrix}\right)$} & $2$ & $1,3,5,7$ & $1,3,5,7$ & $1,3,5,7$ \\\cline{5-8}
		& & & & $3$ & $1,3$ & all & $1,3$ \\\hline 
		
		\multirow{2}{*}{$(1,3,12)^T$} & \multirow{2}{*}{$6$} & \multirow{2}{*}{$(3,1,0)^T$} &\multirow{2}{*}{ $\left(\begin{smallmatrix}12&-60&0\\-6&12&0\\0&0&432 \end{smallmatrix}\right)$} & $2$ & $1,3,5,7$ & $1,3,5,7$ & $1,3,5,7$  \\\cline{5-8}
		& & & & $3$ & $1,3$ & all & $1,3$\\\hline 
		
		\multirow{2}{*}{$(1,3,12)^T$} & \multirow{2}{*}{$12$} & \multirow{2}{*}{$(0,4,3)^T$} &\multirow{2}{*}{ $\left(\begin{smallmatrix}156&0&324\\0&144&0\\324&0&972 \end{smallmatrix}\right)$} & $2$ & $1,3,5,7$ & $1,3,5,7$ & $1,3,5,7$  \\\cline{5-8}
		& & & & $3$ & $1,3$ & all & $1,3$\\\hline 
		
		$(1,4,4)^T$ & $2$ & $(1,0,0)^T$ & $\left(\begin{smallmatrix}1&0&0\\0&16&0\\0&0&16\end{smallmatrix}\right)$ & $2$ & $1,2,5,10$ & $1,2,5,10$ & $1,2,5,10$ \\[2pt] \hline 
		
		$(1,4,4)^T$ & $4$ & $(0,0,1)^T$ & $\left(\begin{smallmatrix}4&0&0\\0&16&0\\0&0&64\end{smallmatrix}\right)$ & $2$ & $1,5$ & $1,2,5,10$ & $1,5$\\[2pt]\hline 
		
		$(1,4,4)^T$ & $4$ & $(1,1,1)^T$ & $\left(\begin{smallmatrix}9&-7&-7\\-7&41&-23\\-7&-23&41\end{smallmatrix}\right)$ & $2$ & $1,2,5,10$ & $1,2,5,10$ & $1,2,5,10$\\[2pt]\hline 
		
		$(1,4,4)^T$ & $4$ & $(4,1,2)^T$ & $\left(\begin{smallmatrix}36&40&16\\40&80&32\\16&32&64\end{smallmatrix}\right)$ & $2$ & $1,5$ & $1,2,5,10$ & $1,5$\\[2pt]\hline 
		
		$(1,4,4)^T$ & $4$ & $(4,3,2)^T$ & $\left(\begin{smallmatrix}68&104&48\\104&208&96\\48&96&64\end{smallmatrix}\right)$ & $2$ & $1,5$ & $1,2,5,10$ & $1,5$\\[2pt]\hline 
		
		$(1,4,8)^T$ & $4$ & $(2,1,0)^T$ & $\left(\begin{smallmatrix}8&0&0\\0&8&0\\0&0&128 \end{smallmatrix}\right)$ & $2$ &$ 1,2,5,10$ & all & $1,2,5,10$ \\[2pt]\hline 
		 
		\multirow{2}{*}{$(1,6,6)^T$} & \multirow{2}{*}{$12$} & \multirow{2}{*}{$(0,1,5)^T$} & \multirow{2}{*}{ $\left(\begin{smallmatrix}156&0&-204\\0&144&0\\-204&0&300\end{smallmatrix}\right)$ } & $2$ & $1,3,5,7$ & all & $1,3,5,7$ \\\cline{5-8}
		& & & & $3$ & $1,3$ & all & $1,3$\\\hline 
		
		$(1,8,8)^T$ & $2$ & $(1,0,1)^T$ & $\left(\begin{smallmatrix}9&0&-7\\0&32&0\\-7&0&9 \end{smallmatrix}\right)$ & $2$ &$ 1,2,5,10$ & all & $1,2,5,10$ \\[2pt]\hline
		
		$(1,8,8)^T$ & $8$ & $(0,1,3)^T$ & $\left(\begin{smallmatrix}80&0&-112\\0&64&0\\-112&0&208x \end{smallmatrix}\right)$ & $2$ &$ 1,5$ & all & $1,5$ \\[2pt] \hline 
		
		\multirow{2}{*}{$(1,12,12)^T$} & \multirow{2}{*}{$12$} & \multirow{2}{*}{$(6,3,4)^T$} & \multirow{2}{*}{ $\left(\begin{smallmatrix}336&408&192\\408&624&192\\192&192&192\end{smallmatrix}\right)$ } & $2$ & $1,3,5,7$ & $1,3,5,7$ & $1,3,5,7$ \\\cline{5-8}
		& & & & $3$ & $1,3$ & all & $1,3$\\\hline 
		
		\pagebreak 
		
		\multirow{2}{*}{$(3,4,12)^T$} & \multirow{2}{*}{$3$} & \multirow{2}{*}{$(1,0,0)^T$} & \multirow{2}{*}{ $\left(\begin{smallmatrix}3&0&0\\0&36&0\\0&0&108\end{smallmatrix}\right)$ } & $2$ & $1,3,5,7$ & $1,3,5,7$ & $1,3,5,7$ \\\cline{5-8}
		& & & & $3$ & $1,3$ & all & $1,3$\\\hline 
		
		\multirow{2}{*}{$(3,4,12)^T$} & \multirow{2}{*}{$6$} & \multirow{2}{*}{$(0,0,1)^T$} & \multirow{2}{*}{ $\left(\begin{smallmatrix}12&0&0\\0&108&0\\0&0&144 \end{smallmatrix}\right)$} & $2$ & all & $1,3,5,7$ & $1,3,5,7$ \\\cline{5-8}
		& & & & $3$ & $1,3$ & all & $1,3$\\\hline 
		
		\multirow{2}{*}{$(3,4,12)^T$} & \multirow{2}{*}{$12$} & \multirow{2}{*}{$(2,3,0)^T$} & \multirow{2}{*}{ $\left(\begin{smallmatrix}48&-96&0\\-96&336&0\\0&0&1728 \end{smallmatrix}\right)$} & $2$ & $1,3,5,7$ & $1,3,5,7$ & $1,3,5,7$ \\\cline{5-8}
		& & & & $3$ & $1,3$ & all & $1,3$\\\hline  
				
		\multirow{2}{*}{$(3,4,36)^T$} & \multirow{2}{*}{$4$} & \multirow{2}{*}{$(2,1,0)^T$} & \multirow{2}{*}{ $\left(\begin{smallmatrix}16&-8&0\\-8&16&0\\0&0&576 \end{smallmatrix}\right)$} & $2$ & $1,3,5,7$ & all & $1,3,5,7$ \\\cline{5-8}
		& & & & $3$ & $1,3$ & $1,3$ & $1,3$\\\hline
	\end{longtable}
	\end{small}
}
	\section{Grouping Lattice Cosets into Spinor Genera}\label{sctn::grouping_spn_gen}
	With reference to the proofs of Propositions \ref{prop::3+cls_in_genus_spnnum=1}, \ref{prop::spnnum=2}, and \ref{prop::spnnum=3}, for a given $\vec a,N$, and $\vec h$, the following table gives some elements $\sigma \Sigma$ of $O^+(V)O_\A'(V)$ (where $\sigma\in O^+(V)$ and $\Sigma \in O_\A'(V)$) and their corresponding actions on certain lattice cosets in the genus of $NL_{\vec a} + \vec h_0$. This gives an immediate proof that two lattice cosets are in the same spinor genus, and allows us to group the various lattice cosets in the genus of $NL_{\vec a}+\vec h_0$ into the various spinor genera.
	
	Here, for a given $\vec a$, $N$, and $\vec h_0$, the values of $\vec h_i$ and $\vec h_j$ are such that $\sigma \Sigma (NL+\vec h_i) = NL+\vec h_j$, where both $NL+\vec h_i$ and $NL+\vec h_j$ are in the genus of $NL+\vec h_0$. The indices $i$ and $j$ are in reference to Table \ref{table::genus_info} in Appendix \ref{sctn::gen_info}. If not explicitly mentioned, $\Sigma_p$ is defined to be the identity on $O^+(V_p)$. We describe the rotations in terms of symmetries $\tau_{\vec u}$, where $\vec u$ will be given in terms of the standard basis vectors $\vec e_1, \vec e_2, \vec e_3$ for $V_p$.  
	
	\begin{longtable}[c]{|c|c|c|c|c|c|c|}
		\caption{Action of $O^+(V)O_\A'(V)$ on lattice cosets}\label{table::spnorbit}\\
		\hline 
		$\vec a$ & $N$ & $\vec h_0$ & $\sigma$ & $\Sigma$ & $\vec h_i$ & $\vec h_j$ \\\hline  
		\endfirsthead
		\hline 
		$\vec a$ & $N$ & $\vec h_0$ & $\sigma$ & $\Sigma$ & $\vec h_i$ & $\vec h_j$ \\\hline  
		\endhead
		
		$(1,1,1)^T$ & 8 & $(1,2,2)^T$ & Id & $\Sigma_2 = \tau_{\vec e_1} \tau_{\vec e_1 + 2\vec e_2 - 2\vec e_3}$ & \footnotesize{$\begin{matrix} (3,0,0)^T \\ (3,4,4)^T \end{matrix}$} & \footnotesize{$\begin{matrix} (3,4,4)^T \\ (3,4,0)^T \end{matrix}$}\\\hline 
		
		$(1,1,1)^T$ & 8 & $(3,2,2)^T$ & Id & $\Sigma_2 = \tau_{\vec e_1} \tau_{\vec e_1+2\vec e_2 - 2\vec e_3}$ & \footnotesize{$\begin{matrix} (1,0,0)^T \\ (1,4,4)^T \end{matrix}$} & \footnotesize{$\begin{matrix} (1,4,4)^T \\ (1,4,0)^T \end{matrix}$}\\\hline 
			
		$(1,1,2)^T$ & 8 & $(0,4,1)^T$ & Id & $\Sigma_2 = \tau_{\vec e_1 + 3\vec e_2} \tau_{3\vec e_1 + 5\vec e_2 -2\vec e_3}$ & \footnotesize{$(0,0,3)^T$} & \footnotesize{$(4,4,3)^T$} \\\hline 
		
		$(1,1,2)^T$ & 8 & $(0,4,3)^T$ & Id & $\Sigma_2 = \tau_{\vec e_1 + 3\vec e_2} \tau_{3\vec e_1 + 5\vec e_2 -2\vec e_3}$ & \footnotesize{$(0,0,1)^T$} & \footnotesize{$(4,4,1)^T$} \\\hline 
		
		\multirow{2}{*}{$(1,1,2)^T$} & \multirow{2}{*}{8} & \multirow{2}{*}{$(1,3,2)^T$} & $\tau_{\vec e_2}\tau_{\vec e_3}$ & $\Sigma_2 = \tau_{\vec e_3} \tau_{5\vec e_1 + \vec e_2 - 2\vec e_3}$ & \footnotesize{$\begin{matrix} (1,3,2)^T \\ (1,1,0)^T \\ (1,1,4)^T \end{matrix}$} & \footnotesize{$\begin{matrix} (1,5,2)^T \\ (3,3,4)^T \\ (3,3,0)^T \end{matrix}$} \\\cline{4-7}
		& & & Id & $\Sigma_2 = \tau_{\vec e_1+\vec e_2} \tau_{3\vec e_1+\vec e_2 - 2\vec e_3}$ & \footnotesize{$(1,1,0)^T$} & \footnotesize{$(3,3,0)^T$} \\\hline 
		
		$(1,1,4)^T$ & 4 & $(1,2,1)^T$ & Id & $\Sigma_2 = \tau_{\vec e_1 - \vec e_2} \tau_{\vec e_1 + \vec e_2 - 2\vec e_3}$ & \footnotesize{$(1,0,0)^T$} & \footnotesize{$(1,0,2)^T$}\\\hline 
		
		$(1,1,8)^T$ & 4 & $(1,1,1)^T$ & Id & $\Sigma_2 = \tau_{\vec e_2} \tau_{3\vec e_2 - \vec e_3}$ & \footnotesize{$(1,1,0)^T$} & \footnotesize{$(1,1,2)^T$} \\\hline 
		
		$(1,2,2)^T$ & 4 & $(1,0,2)^T$ & Id & $\Sigma_2 = \tau_{\vec e_1 - \vec e_2} \tau_{\vec e_1 - \vec e_3}$ & \footnotesize{$(1,0,0)^T$} & \footnotesize{$(1,2,2)^T$} \\\hline  
		
		$(1,2,2)^T$ & 8 & $(4,1,3)^T$ & Id & $\Sigma_2 = \tau_{2\vec e_1 - \vec e_2} \tau_{2\vec e_1 - 4\vec e_2 - 3\vec e_3}$ & \footnotesize{$(0,1,1)^T$} & \footnotesize{$(0,3,3)^T$} \\\hline 
		
		\multirow{2}{*}{$(1,2,4)^T$} & \multirow{2}{*}{8} & \multirow{2}{*}{$(0,1,2)^T$} & Id & $\Sigma_2 = \tau_{\vec e_2} \tau_{\vec e_2 - 2\vec e_3}$ & \footnotesize{$\begin{matrix} (4,1,0)^T \\ (0,3,0)^T \end{matrix}$} & \footnotesize{$\begin{matrix} (4,1,4)^T \\ (0,3,4)^T  \end{matrix}$} \\\cline{4-7}
		& & & Id & $\Sigma_2 = \tau_{\vec e_1+\vec e_2} \tau_{2\vec e_2 - \vec e_3}$ & \footnotesize{$\begin{matrix} (4,3,2)^T \\ (0,1,2)^T \end{matrix}$} & \footnotesize{$\begin{matrix} (0,3,0)^T \\ (4,1,4)^T  \end{matrix}$}  \\\hline 
		
		\multirow{2}{*}{$(1,2,4)^T$} & \multirow{2}{*}{8} & \multirow{2}{*}{$(0,3,2)^T$} & Id & $\Sigma_2 = \tau_{\vec e_2} \tau_{\vec e_2 - 2\vec e_3}$ & \footnotesize{$\begin{matrix} (4,3,0)^T \\ (0,1,0)^T \end{matrix}$} & \footnotesize{$\begin{matrix} (4,3,4)^T \\ (0,1,4)^T  \end{matrix}$} \\\cline{4-7}
		& & & Id & $\Sigma_2 = \tau_{\vec e_1+\vec e_2} \tau_{2\vec e_2 - \vec e_3}$ & \footnotesize{$\begin{matrix} (4,1,2)^T \\ (0,3,2)^T \end{matrix}$} & \footnotesize{$\begin{matrix} (0,1,4)^T \\ (4,3,0)^T  \end{matrix}$}  \\\hline 
		
		$(1,2,8)^T$ & 4 & $(1,0,1)^T$ & Id & $\Sigma_2 = \tau_{\vec e_1 - \vec e_2} \tau_{2\vec e_1+\vec e_3}$ & \footnotesize{$\begin{matrix} (1,0,1)^T \\ (1,0,2)^T \\ (1,0,0)^T \\ (1,2,1)^T \end{matrix}$} & \footnotesize{$\begin{matrix} (1,2,2)^T \\ (1,0,1)^T\\ (1,2,1)^T \\ (1,0,2)^T \end{matrix}$}\\\hline 
		
		$(1,3,3)^T$ & 12 & $(3,3,4)^T$ & Id & $\Sigma_3 = \tau_{\vec e_2} \tau_{\vec e_3}$ & \footnotesize{$\begin{matrix} (3,3,4)^T \\ (3,0,1)^T  \end{matrix}$} & \footnotesize{$\begin{matrix} (3,3,8)^T \\ (3,0,5)^T \end{matrix}$}\\\hline 
		
		$(1,3,12)^T$ & 12 & $(0,4,3)^T$ & Id & $\Sigma_3 = \tau_{\vec e_2} \tau_{\vec e_3}$ & \footnotesize{$(0,0,1)^T$} & \footnotesize{$(0,0,5)^T$}\\\hline 
		
		\multirow{2}{*}{$(1,4,4)^T$} & \multirow{2}{*}{4} & \multirow{2}{*}{$(1,1,1)^T$} & Id & $\Sigma_2 = \tau_{\vec e_1 + \vec e_2} \tau_{\vec e_1 + 2\vec e_2 - \vec e_3}$ & \footnotesize{$(1,0,0)^T$} & \footnotesize{$(1,2,0)^T$} \\\cline{4-7}
		& & & Id & $\Sigma_2 = \tau_{\vec e_2} \tau_{-20\vec e_1 +\vec e_2 + 2\vec e_3}$ & \footnotesize{$(1,0,0)^T$} & \footnotesize{$(1,2,0)^T$}  \\\hline 
		
		$(1,4,4)^T$ & 8 & $(4,1,2)^T$ & Id & $\Sigma_2 = \tau_{2\vec e_2 - \vec e_3} \tau_{\vec e_1 + \vec e_2 + 2\vec e_3}$ & \footnotesize{$(0,3,0)^T$} & \footnotesize{$(0,3,4)^T$} \\\hline 
		
		$(1,4,4)^T$ & 8 & $(4,3,2)^T$ & Id & $\Sigma_2 = \tau_{2\vec e_2 - \vec e_3} \tau_{\vec e_1 + \vec e_2 + 2\vec e_3}$ & \footnotesize{$(0,1,0)^T$} & \footnotesize{$(0,1,4)^T$} \\\hline 
		
		$(1,6,6)^T$ & 12 & $(0,1,5)^T$ & Id & $\Sigma_3 = \tau_{\vec e_2 +\vec e_3} \tau_{3\vec e_1 + 2\vec e_2 + \vec e_3}$ & \footnotesize{$(0,1,1)^T$} & \footnotesize{$(0,5,5)^T$} \\\hline 
		
		$(1,8,8)^T$ & 2 & $(1,0,1)^T$ & Id & $\Sigma_2 = \tau_{12\vec e_1 -\vec e_2} \tau_{28\vec e_1 + 4\vec e_2 + 19\vec e_3}$ & \footnotesize{$(1,0,0)^T$} & \footnotesize{$(1,1,1)^T$} \\\hline 
		
		$(1,8,8)^T$ & 8 & $(0,1,3)^T$ & Id & $\Sigma_2 = \tau_{12\vec e_1 -\vec e_2} \tau_{28\vec e_1 + 4\vec e_2 + 19\vec e_3}$ & \footnotesize{$(0,1,1)^T$} & \footnotesize{$(0,3,3)^T$} \\ \hline 
		
		$(1,12,12)^T$ & 12 & $(6,3,4)^T$ & Id & $\Sigma_3 = \tau_{\vec e_2} \tau_{\vec e_3}$ & \footnotesize{$(6,0,1)^T$} & \footnotesize{$(6,0,5)^T$} \\\hline 
	\end{longtable}

	\printbibliography
\end{document}